\documentclass{article}

\usepackage{graphicx}
\usepackage{float}
\usepackage{amsthm}
\usepackage{amsmath}
\usepackage{amssymb}
\usepackage{mathrsfs}
\usepackage{dsfont}
\usepackage[all]{xy}
\usepackage{tikz-cd}
\usepackage{stmaryrd}
\usepackage{enumitem}
\usepackage{placeins}
\usepackage{pdflscape}
\usepackage{hyperref}

\theoremstyle{definition}
\newtheorem{Definition}{Definition}[subsection]

\theoremstyle{plain}
\newtheorem{Theorem}[Definition]{Theorem}

\theoremstyle{plain}
\newtheorem{Proposition}[Definition]{Proposition}

\theoremstyle{plain}
\newtheorem{Lemma}[Definition]{Lemma}

\theoremstyle{plain}
\newtheorem{Corollary}[Definition]{Corollary}

\theoremstyle{plain}

\theoremstyle{definition}
\newtheorem{Example}[Definition]{Example}

\theoremstyle{remark}
\newtheorem{Remark}[Definition]{Remark}

\theoremstyle{plain}
\newcommand{\thistheoremname}{}
\newtheorem*{genericthm*}{\thistheoremname}
\newenvironment{namedthm*}[1]
  {\renewcommand{\thistheoremname}{#1}%
   \begin{genericthm*}}
  {\end{genericthm*}}

\title{Rigid and Separable Algebras in Fusion 2-Categories}
\author{Thibault D. D\'ecoppet}
\date{June 2023}

\AtEndDocument{
    \noindent\textsc{Mathematical Institute, University of Oxford, Oxford OX2 6GG, United Kingdom}\\
    \textit{E-mail address}: \texttt{thibault.decoppet@maths.ox.ac.uk}\\
    \textit{URL}: \texttt{https://www.thibaultdecoppet.com}\\
}

\begin{document}

\bibliographystyle{alpha}

    \maketitle
    \hspace{1cm}
    \begin{abstract}
        Rigid monoidal 1-categories are ubiquitous throughout quantum algebra and low-dimensional topology. We study a generalization of this notion, namely rigid algebras in an arbitrary monoidal 2-category. Examples of rigid algebras include $G$-graded fusion 1-categories, and $G$-crossed fusion 1-categories. We explore the properties of the 2-categories of modules and of bimodules over a rigid algebra, by giving a criterion for the existence of right and left adjoints. Then, we consider separable algebras, which are particularly well-behaved rigid algebras. Specifically, given a fusion 2-category, we prove that the 2-categories of modules and of bimodules over a separable algebra are finite semisimple. Finally, we define the dimension of a connected rigid algebra in a fusion 2-category, and prove that such an algebra is separable if and only if its dimension is non-zero.
    \end{abstract}
    
\tableofcontents

\section*{Introduction}
\addcontentsline{toc}{section}{Introduction}

Rigid monoidal 1-categories, that is monoidal 1-categories with right and left duals, have first appeared in \cite{Saa} following ideas of Grothendieck. Since then, these objects have found applications in many different areas of mathematics such as algebraic geometry \cite{Del}, representation theory \cite{Dri}, quantum algebra \cite{ENO}, and low-dimensional topology \cite{TV}. In all of these applications, rigidity may be thought of as a categorical finiteness assumption. When specialized to fusion 1-categories, i.e. finite semisimple rigid monoidal 1-categories with simple unit (introduced in \cite{ENO}), this idea can be made precise via the cobordism hypothesis of Baez-Dolan \cite{BD1} (see also \cite{L}). Namely, the cobordism hypothesis identifies symmetric monoidal functors, from a higher category of cobordisms to a higher category $\mathscr{C}$, with fully dualizable objects in $\mathscr{C}$. As explained in \cite{L}, full dualizability is a strong finiteness condition. Further, fusion 1-categories over an algebraically closed field of characteristic zero have been shown in \cite{DSPS13} to be fully dualizable objects of an appropriate symmetric monoidal 3-category of finite tensor 1-categories, denoted by $\mathbf{TC}$.

More recently, Gaitsgory has introduced in \cite{G} a slightly different notion of rigidity. A linear monoidal 1-category $\mathcal{A}$ is rigid in the sense of \cite{G} if the monoidal product of $\mathcal{A}$ has a right adjoint as an $\mathcal{A}$-$\mathcal{A}$-bimodule functor. The advantage of this second definition of rigidity is that it makes sense in any monoidal 2-category. In order to make sure that this definition is sensible, let $\mathbf{2Vect}$ denote the symmetric monoidal 2-category of finite semisimple 1-categories (over an algebraically closed field of characteristic zero). Then, it was established in \cite{BJS} that an algebra in $\mathbf{2Vect}$, that is a finite semisimple monoidal 1-category, is rigid in the sense of \cite{G} if and only if it is a multifusion 1-category, i.e.\ has right and left duals. Motivated by the quest for interesting fully dualizable objects initiated by the cobordism hypothesis, we are interested in replacing $\mathbf{2Vect}$ by an arbitrary finite semisimple monoidal 2-category.

Particularly interesting examples of finite semisimple monoidal 2-categories are given by the fusion 2-categories introduced in \cite{DR} over an algebraically closed field of characteristic zero. Familiar examples include $\mathbf{2Vect}_G$, the 2-category of finite semisimple 1-categories graded by a finite group $G$, and $\mathbf{Mod}(\mathcal{B})$, the 2-category of finite semisimple module categories over the braided fusion 1-category $\mathcal{B}$. It is instructive to look at rigid algebras in these fusion 2-categories. Namely, rigid algebras in $\mathbf{2Vect}_G$ correspond exactly to $G$-graded mutlifusion 1-categories, and rigid algebras in $\mathbf{Mod}(\mathcal{B})$ are precisely multifusion 1-categories equipped with a central functor from $\mathcal{B}$. Furthermore, by \cite{KTZ}, the monoidal 2-category $\mathscr{Z}(\mathbf{2Vect}_G)$, the Drinfeld center of $\mathbf{2Vect}_G$, is finite semisimple. One finds that rigid algebras in $\mathscr{Z}(\mathbf{2Vect}_G)$ are exactly $G$-crossed multifusion 1-categories, of which the so-called $G$-crossed braided fusion 1-categories are examples. All of these examples have been extensively studied in the fusion 1-category literature (for instance, see \cite{DGNO} and \cite{BJS}), but always separately. Our approach allows us to study these different objects simultaneously as rigid algebras in finite semisimple monoidal 2-categories.

We have yet another motivation for studying rigid algebras in fusion 2-categories. Namely, these 2-categories categorify (in the sense of Baez-Dolan, see \cite{BD2}) the notion of fusion 1-categories. Now, just as modules are ubiquitous in the study of rings, finite semisimple module 1-categories are essential in the theory of fusion 1-categories. But, by a classical result of Ostrik (see \cite{O}), such a finite semisimple module 1-category over a fusion 1-category $\mathcal{A}$ is equivalent to the category of modules over an algebra in $\mathcal{A}$. In \cite{D4}, we have proven a categorified version of this result, that is, we have proven that a finite semisimple module 2-category over a fusion 2-category $\mathfrak{C}$ is equivalent to the 2-category of modules over a rigid algebra in $\mathfrak{C}$. Furthermore, categorifying the main result of \cite{DSPS13}, it has been conjectured that fusion 2-categories over an algebraically closed field of characteristic zero are fully dualizable objects of the symmetric monoidal 4-category of fusion 2-categories, finite semisimple bimodule 2-categories, etc. Proving this conjecture requires having a very thorough understanding of the properties of finite semisimple (bi)module 2-categories. Through the aforementioned categorification of Ostrik's theorem, this understanding is closely linked to studying rigid algebras in fusion 2-categories and their associated 2-categories of (bi)modules.

Before explaining the properties of the 2-categories of (bi)modules over a rigid algebra, we need to explain the definition in more detail. To this end, let $\mathfrak{C}$ be a monoidal 2-category with monoidal product $\Box$. An algebra in $\mathfrak{C}$ is an object $A$ of $\mathfrak{C}$ equipped with a unit, a multiplication $m:A\Box A\rightarrow A$, and coherent 2-isomorphisms witnessing unitality and associativity (these objects are called pseudo-monoids in \cite{DS}). Following \cite{G} (see also \cite{JFR}), we say that an algebra $A$ is rigid if $m$ has a right adjoint $m^*$ as a map of $A$-$A$-bimodules. Then, fixing a rigid algebra $A$, we can explore the properties of $\mathbf{Mod}_{\mathfrak{C}}(A)$, the 2-category of (right) $A$-modules in $\mathfrak{C}$. In particular, given $f:M\rightarrow N$ an $A$-module 1-morphism, we prove that if $f$ has a left adjoint in $\mathfrak{C}$, then it has a left adjoint as an $A$-module 1-morphism. This generalizes many different results found in the literature (for instance, see section 3.3 of \cite{EO} and theorem 5.2.1 of \cite{BJS}). A similar result holds in $\mathbf{Bimod}_{\mathfrak{C}}(A)$, the 2-category of $A$-$A$-bimodules in $\mathfrak{C}$. In fact, under some assumptions on $\mathfrak{C}$, which hold if $\mathfrak{C}$ is a finite semisimple monoidal 2-category, we can precisely characterize rigid algebras in terms of a property of $\mathbf{Bimod}_{\mathfrak{C}}(A)$, as follows.

\begin{namedthm*}{Theorem \ref{thm:modulerightdajoints}}
Let $\mathfrak{C}$ be a monoidal 2-category that has left and right adjoints for 1-morphisms, and let $A$ be an arbitrary algebra in $\mathfrak{C}$. Then, $A$ is rigid if and only if $\mathbf{Bimod}_{\mathfrak{C}}(A)$ has right adjoints. If either of these conditions is satisfied, $\mathbf{Bimod}_{\mathfrak{C}}(A)$ has left adjoints, and $\mathbf{Mod}_{\mathfrak{C}}(A)$ has left and right adjoints.
\end{namedthm*}

Now, whilst multifusion 1-categories over an algebraically closed field of characteristic zero are fully dualizable objects of the symmetric monoidal 3-category $\mathbf{TC}$ of finite tensor 1-categories, this result no longer holds over an algebraically closed field of positive characteristic. One of the many key insights of \cite{DSPS13} is that, in order to remedy this issue, we ought to consider separable multifusion 1-categories. Namely, separable multifusion 1-categories are much better behaved than general ones. For instance, given a multifusion 1-category $\mathcal{A}$, its Drinfeld center $\mathcal{Z}(\mathcal{A})$, the 1-category of $\mathcal{A}$-$\mathcal{A}$-bimodule endofunctors of $\mathcal{A}$, is finite semisimple if and only if $\mathcal{A}$ is separable. Furthermore, a separable multifusion 1-category is a fully dualizable object of $\mathbf{TC}$ over any algebraically closed field.

Abstracting the definition of a separable multifusion 1-category, we are lead to consider the notion of separable algebra in a general finite semisimple monoidal 2-category, that is a rigid algebra $A$ such that the counit of the adjunction between $m$ and $m^*$ has a section as an $A$-$A$-bimodule 2-morphism. One may generalize this story further. Namely, working over an arbitrary field, we gave a definition of compact semisimple 2-category in \cite{D5}, which specializes to that of finite semisimple 2-category when the base field is algebraically closed. In this general context, by analogy with the case of separable multifusion 1-categories, one expects that separable algebras have particularly nice properties. This turns out to be true. As an example, given a separable algebra $A$ in a compact semisimple monoidal 2-category $\mathfrak{C}$, we show that $\mathbf{Mod}_{\mathfrak{C}}(A)$ is a compact semisimple 2-category. Using this fact, one can show that, over an algebraically closed field, a fusion 2-category is a fully dualizable object of a certain symmetric monoidal 4-category if and only if a canonical rigid algebra associated to it is separable (see \cite{D9} for the construction). This provides an approach to solving the aforementioned conjecture asserting that every fusion 2-category over an algebraically closed field of characteristic zero is a fully dualizable object in a certain symmetric monoidal 4-category. Namely, this conjecture is reduced to proving that some rigid algebras are separable. Thus, it is important to characterize separable algebras in a finite semisimple monoidal 2-category precisely. This is achieved in a slightly more general context in the following result.

\begin{namedthm*}{Theorem \ref{thm:characterizationseparablealgebra}}
Let $A$ be a rigid algebra in a compact semisimple monoidal 2-category $\mathfrak{C}$, and write $Z(A)$ for the 1-category of $A$-$A$-bimodule endomorphisms of $A$ in $\mathfrak{C}$. Then, $A$ is separable if and only if $Z(A)$ is finite semisimple. If either of these conditions is satisfied, $\mathbf{Bimod}_{\mathfrak{C}}(A)$ is a compact semisimple 2-category.
\end{namedthm*}

It was proven in \cite{DSPS13} that, if we work over an algebraically closed field of characteristic zero, every fusion 1-category is separable. In fact, they prove that a fusion 1-category $\mathcal{A}$ over an algebraically closed field is separable if and only if $\mathrm{dim}(\mathcal{A})$, its categorical (or global) dimension, is non-zero. The fact that every fusion 1-category is a fully dualizable object of $\mathbf{TC}$ then follows from the result established in \cite{ENO} that the categorical dimension of a fusion 1-category over an algebraically closed field of characteristic zero does not vanish. As observed above, studying the dualizability properties of fusion 2-categories is intimately linked to understanding when rigid algebras are separable, it is therefore important to ask whether this story can be adapted to an abstract compact semisimple monoidal 2-category. More precisely, given a compact semisimple monoidal 2-category $\mathfrak{C}$ over an arbitrary field, one can define connected rigid algebras in $\mathfrak{C}$ as those rigid algebras whose unit 1-morphism is simple. Over an algebraically closed field, connected rigid algebras in $\mathfrak{C}=\mathbf{2Vect}$ are exactly fusion 1-categories. This motivates the conjecture made in \cite{JFR} that any connected rigid algebra in a fusion 2-category over an algebraically closed field of characteristic zero is separable. We note that their conjecture is stronger than the conjecture asserting that every fusion 2-category is fully dualizable over an algebraically closed field of characteristic zero.

Towards a proof of the conjecture of \cite{JFR}, we associate a scalar $\mathrm{Dim}_{\mathfrak{C}}(A)$ to every connected rigid algebra $A$ in a finite semisimple monoidal 2-category $\mathfrak{C}$ over an algebraically closed field. We wish to emphasize that the definition of $\mathrm{Dim}_{\mathfrak{C}}(A)$ makes heavy use of the properties of the 2-category of bimodules over a rigid algebra obtained previously. Further, we also note that, while the definition of $\mathrm{Dim}_{\mathfrak{C}}(A)$ makes sense over an arbitrary field, it is in general not well-defined. We use the scalar $\mathrm{Dim}_{\mathfrak{C}}(A)$ to characterize when $A$ is separable as follows.

\begin{namedthm*}{Theorem \ref{thm:dimensionseparable}}
Over an algebraically closed field, a connected rigid algebra $A$ in a finite semisimple monoidal 2-category $\mathfrak{C}$ is separable if and only if its dimension $\mathrm{Dim}_{\mathfrak{C}}(A)$ is non-zero.
\end{namedthm*}

\noindent Moreover, we prove that if $\mathfrak{C}=\mathbf{2Vect}$ and the base field is algebraically closed, then our notion of dimension coincides with the categorical dimension of a fusion 1-category. Thus, the above theorem does generalize the fact recalled above that fusion 1-categories of non-zero global dimension are separable. Further, this reduces the conjecture of \cite{JFR} to proving that every connected rigid algebra in a fusion 2-category over an algebraically closed field of characteristic zero has non-zero dimension. We end by proving that, over an algebraically closed field of characteristic zero, every connected rigid algebra in $\mathbf{2Vect}_G$, or $\mathbf{Mod}(\mathcal{B})$ is separable, and that every strongly connected rigid algebra in $\mathscr{Z}(\mathbf{2Vect}_G)$ is separable.

\subsection*{Outline}

In section \ref{sec:prelim}, we begin by explaining the graphical language that we will use to work with monoidal 2-categories in the present article. We go on to recall the definitions of algebras and right modules in a monoidal 2-category, as well as those of left modules and bimodules. Then, we review the notions of Cauchy complete 2-categories and of compact semisimple 2-categories. We end this preliminary section by supplying many examples of algebras in compact semisimple monoidal 2-categories.

Next, in section \ref{sec:rigid}, we unravel the definition of a rigid algebra in a monoidal 2-category $\mathfrak{C}$, and identify rigid algebras in the compact semisimple monoidal 2-categories $\mathbf{2Vect}$ of 2-vector spaces, $\mathbf{2Vect}_G$ of 2-vector spaces graded by the finite group $G$, $\mathbf{Mod}(\mathcal{B})$ of separable module categories over the braided fusion 1-category $\mathcal{B}$, and $\mathscr{Z}(\mathbf{2Vect}_G)$, the Drinfeld center of $\mathbf{2Vect}_G$. For instance, over an algebraically closed field, rigid algebras in $\mathbf{2Vect}$ are precisely multifusion 1-categories. Then, we go on to show that if the underlying monoidal 2-category $\mathfrak{C}$ has left adjoints, so do the 2-categories of modules and of bimodules over a rigid algebra in $\mathfrak{C}$. Furthermore, if $\mathfrak{C}$ has right and left adjoints, we prove that an algebra in $\mathfrak{C}$ is rigid if and only if its 2-category of bimodules has right adjoints.

In section \ref{sec:separable}, given a compact semisimple monoidal 2-category $\mathfrak{C}$, we use the results of the previous section to precisely characterize separable algebras in $\mathfrak{C}$ as those algebras whose associated 2-categories of bimodules are compact semisimple. Then, we go on to define the dimension of a connected rigid algebra in $\mathfrak{C}$, and prove that such an algebra is separable if and only if its dimension is non-zero. A fusion 1-category $\mathcal{C}$ is a connected rigid algebra in $\mathbf{2Vect}$; we show that the dimension of $\mathcal{C}$ in $\mathbf{2Vect}$ is exactly the categorical (or global) dimension of the fusion 1-category $\mathcal{C}$. We use these results to give many examples of separable algebras in the compact semisimple monoidal 2-categories $\mathbf{2Vect}_G$, $\mathbf{Mod}(\mathcal{B})$, and $\mathscr{Z}(\mathbf{2Vect}_G)$.

Finally, in appendix \ref{sec:diag}, we give the string diagram manipulations, on which section \ref{sec:rigid} relies.

\subsection*{Acknowledgments}

I am particularly thankful to Christopher Douglas for his guidance during this project. In addition, I would also like to thank David Reutter and Christoph Weis for our conversations around the content of this article. Finally, I am grateful towards Theo Johnson-Freyd for pointing out a gap in the proof of theorem \ref{thm:dimensionseparable}.

\section{Preliminaries}\label{sec:prelim}

\subsection{Graphical Conventions}

Throughout this article, we will be working within certain (weak) 2-categories. In this context, there is a well-established graphical calculus of string diagram (see \cite{GS}). More precisely, regions correspond to objects, strings to 1-morphisms, and coupons to 2-morphisms. We take the convention that our string diagrams are read from top to bottom (composition of 1-morphisms) and from left to right (composition of 2-morphisms). We use the symbol $1$ to denote the identity 1-morphism on a given object, but we will usually omit such 1-morphisms. As an example, the composite of the 2-morphisms $\gamma:f\Rightarrow g$ and the 1-morphism $h$ in some 2-category $\mathfrak{C}$, provided it is defined, is depicted using the following diagram:

\newlength{\prelim}
\settoheight{\prelim}{\includegraphics[width=30mm]{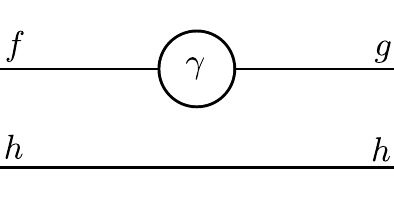}}

\begin{center}
    \includegraphics[width=30mm]{Pictures/prelim/graph/f.pdf}.
\end{center}

In fact, we will work with a monoidal 2-category $\mathfrak{C}$, that is a 2-category $\mathfrak{C}$ equipped with a 2-functor $\Box:\mathfrak{C}\times\mathfrak{C}\rightarrow \mathfrak{C}$, and unit object $I$ in $\mathfrak{C}$, as well as various adjoint 2-natural equivalence and modifications ensuring the unitality of $I$ and that $\Box$ is coherently associative (see \cite{SP}, \cite{GS}, or \cite{D4} for details). Thanks to the coherence theorem for monoidal 2-categories proven by Gurski in \cite{Gur}, every monoidal 2-category is equivalent to one that is strict cubical. Recall that a strict cubical monoidal 2-category $\mathfrak{C}$ is a strict 2-category, equipped with a monoidal structure for which the unit $I$ is strict, and $\Box$ is strictly associative. In particular, we may omit the symbol $I$ as well as the parentheses. Further, the 2-functor $\Box$ becomes strict when it is restricted to either variable. However, it is not a strict 2-functor in general. Namely, given a pair of composable 1-morphisms $f_1,f_2$, and $g_1,g_2$ in $\mathfrak{C}$, the interchanger $$\phi^{\Box}_{(f_2,g_2),(f_1,g_1)}:(f_2\Box g_2)\circ (f_1\Box g_1)\cong (f_2\circ f_1)\Box (g_2\circ g_1),$$ the 2-isomorphism ensuring that $\Box$ respects composition, is non-trivial in general. The strict cubical hypothesis on $\mathfrak{C}$ guarantees that $\phi^{\Box}_{(f_2,g_2),(f_1,g_1)}$ is an identity 2-morphism if either $f_2=1$ or $g_1=1$. Given 1-morphisms $f$ and $g$ in $\mathfrak{C}$, the composite $$(\phi^{\Box}_{(1,g),(f,1)})^{-1}\cdot\phi^{\Box}_{(f,1),(1,g)}:(f\Box 1)\circ (1\Box g)\cong (1\Box g)\circ (f\Box 1)$$ will be depicted in our graphical language by the string diagram below on the left, and its inverse by that on the right: $$\begin{tabular}{c c c c}
\includegraphics[width=20mm]{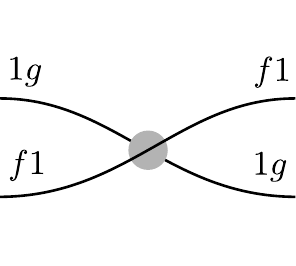},\ \ \ \   & \ \ \ \  \includegraphics[width=20mm]{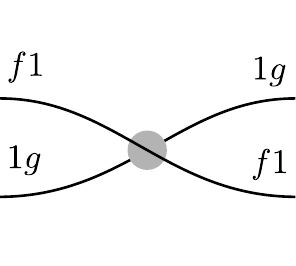}.
\end{tabular}$$ Note that we have omitted the symbol $\Box$ to make the diagrams more easily readable.

We will be interested in studying right (and left) adjoints for 1-morphisms in the 2-category $\mathfrak{C}$. If $f$ is a 1-morphism in $\mathfrak{C}$ that has a right adjoint $f^*$, we will use string diagram depicted below on the left for the unit $\eta^f$ of this adjunction, and the one on the right for the counit $\epsilon^f$: $$\begin{tabular}{c c}
\includegraphics[width=20mm]{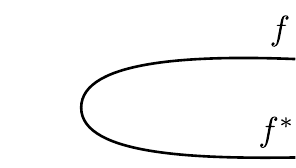},\ \ \ \ \ \ \ \ 
\ \ \ \ \ \ \ \ \includegraphics[width=20mm]{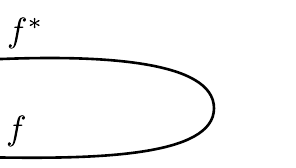}.
\end{tabular}$$ Note that the triangle identities for the adjunction correspond graphically to the so-called snake-equations. We use a similar convention for left adjoints. Finally, given $f$ as above, the 1-morphism $1f$ obtained by taking the monoidal product of $f$ with some identity 1-morphism has a canonical right adjoint given by $1f^*$ with unit $1\eta^f$ and counit $1\epsilon^f$. We shall always consider adjuntions data of this form on such a 1-morphism, and we adopt the analogous convention for the 1-morphism $f1$. Finally, given two 1-morphisms $f$ and $g$ in $\mathfrak{C}$ with right adjoints $f^*$ and $g^*$ together with a 2-morphism $\gamma:f\Rightarrow g$, we set

\settoheight{\prelim}{\includegraphics[width=30mm]{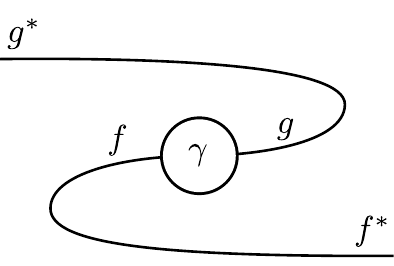}}

\begin{center}
\begin{tabular}{@{}cc@{}}

\raisebox{0.45\prelim}{$\gamma^*:=$} &
\includegraphics[width=30mm]{Pictures/prelim/graph/rightadjoint2morphism.pdf}.
\end{tabular}
\end{center}

\subsection{Algebras and Right Modules}

Throughout, we assume that $\mathfrak{C}$ is strict cubical. We begin by recalling the definition of an algebra (also called a pseudo-monoid in \cite{DS}). For the definition presented using the same graphical calculus in a general monoidal 2-category, we invite the reader to consult \cite{D4}.

\begin{Definition}\label{def:algebra}
An algebra in $\mathfrak{C}$ consists of:
\begin{enumerate}
    \item An object $A$ of $\mathfrak{C}$;
    \item Two 1-morphisms $m:A\Box A\rightarrow A$ and $i:I\rightarrow A$;
    \item Three 2-isomorphisms
\end{enumerate}
\begin{center}
\begin{tabular}{@{}c c c@{}}
$\begin{tikzcd}[sep=small]
A \arrow[rrrr, equal] \arrow[rrdd, "i1"'] &  & {} \arrow[dd, Rightarrow, "\lambda"', near start, shorten > = 1ex] &  & A \\
                                   &  &                           &  &   \\
                                   &  & AA, \arrow[rruu, "m"']     &  &  
\end{tikzcd}$

&

$\begin{tikzcd}[sep=small]
AAA \arrow[dd, "1m"'] \arrow[rr, "m1"]    &  & AA \arrow[dd, "m"] \\
                                            &  &                      \\
AA \arrow[rr, "m"'] \arrow[rruu, Rightarrow, "\mu", shorten > = 2.5ex, shorten < = 2.5ex] &  & A,                   
\end{tikzcd}$

&

$\begin{tikzcd}[sep=small]
                                  &  & AA \arrow[rrdd, "m"] \arrow[dd, Rightarrow, "\rho", shorten > = 1ex, shorten < = 2ex] &  &   \\
                                  &  &                                             &  &   \\
A \arrow[rruu, "1i"] \arrow[rrrr,equal] &  & {}                                          &  & M,
\end{tikzcd}$

\end{tabular}
\end{center}

satisfying:

\begin{enumerate}
\item [a.] We have:
\end{enumerate}

\settoheight{\prelim}{\includegraphics[width=52.5mm]{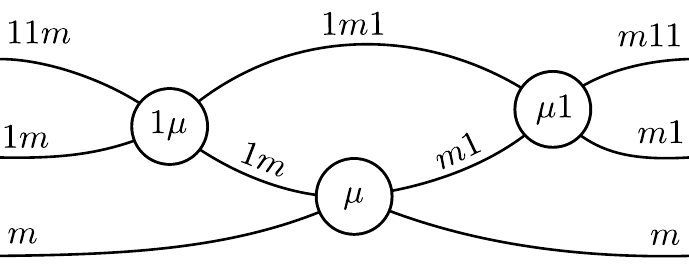}}

\begin{equation}\label{eqn:algebraassociativity}
\begin{tabular}{@{}ccc@{}}

\includegraphics[width=52.5mm]{Pictures/prelim/algebra/associativity1.pdf} & \raisebox{0.45\prelim}{$=$} &
\includegraphics[width=45mm]{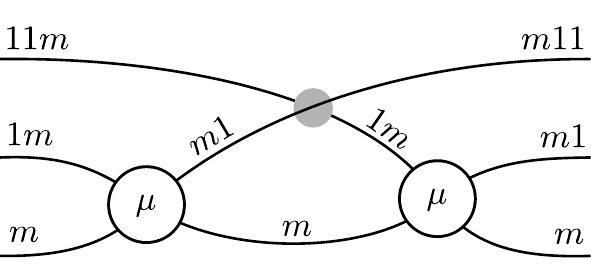},

\end{tabular}
\end{equation}

\begin{enumerate}
\item [b.] We have:
\end{enumerate}

\settoheight{\prelim}{\includegraphics[width=22.5mm]{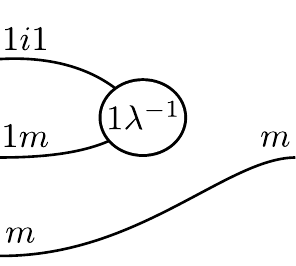}}

\begin{equation}\label{eqn:algebraunitality}
\begin{tabular}{@{}ccc@{}}

\includegraphics[width=22.5mm]{Pictures/prelim/algebra/unitality1.pdf} & \raisebox{0.45\prelim}{$=$} &

\includegraphics[width=37.5mm]{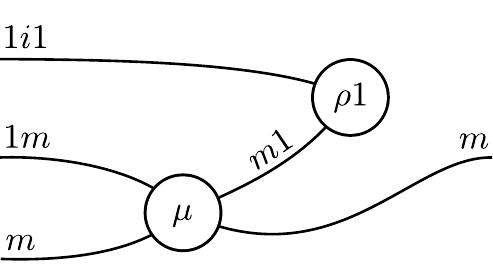}.

\end{tabular}
\end{equation}
\end{Definition}

We will use the following coherence properties in what follows.

\begin{Lemma}
Given any algebra $A$, the following two equalities hold:

\settoheight{\prelim}{\includegraphics[width=37.5mm]{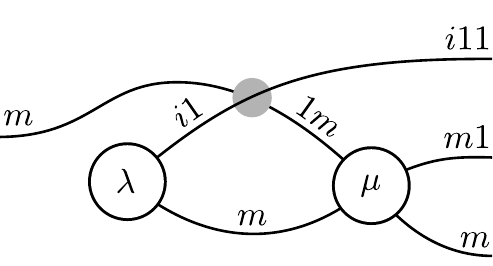}}

\begin{equation}\label{eqn:coherenceleft}
\begin{tabular}{@{}ccc@{}}
\includegraphics[width=37.5mm]{Pictures/prelim/coherenceleft1.pdf}&
\raisebox{0.45\prelim}{$=$} &
\includegraphics[width=22.5mm]{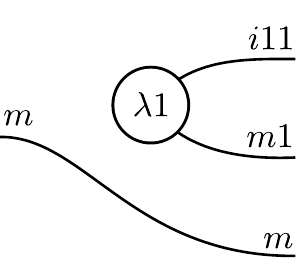},
\end{tabular}
\end{equation}

\settoheight{\prelim}{\includegraphics[width=30mm]{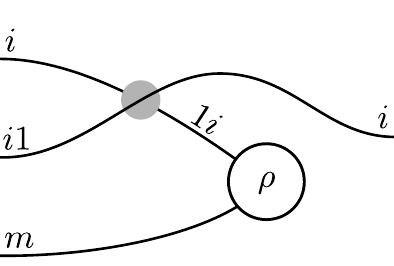}}

\begin{equation}\label{eqn:coherencemiddle}
\begin{tabular}{@{}ccc@{}}
\includegraphics[width=30mm]{Pictures/prelim/coherencemixte1.pdf}&
\raisebox{0.45\prelim}{$=$} &
\includegraphics[width=22.5mm]{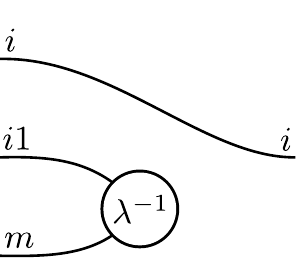}.
\end{tabular}
\end{equation}

\end{Lemma}
\begin{proof}
See section 6.3 of \cite{Hou}.
\end{proof}

Let us now recall the definition of a right $A$-module in $\mathfrak{C}$. Once, again the definition in a general monoidal 2-category may be found in \cite{D4}.

\begin{Definition}\label{def:module}
A right $A$-module in $\mathfrak{C}$ consists of:
\begin{enumerate}
    \item An object $M$ of $\mathfrak{C}$;
    \item A 1-morphism $n^M:M\Box A\rightarrow M$;
    \item Two 2-isomorphisms
\end{enumerate}
\begin{center}
\begin{tabular}{@{}c c@{}}
$\begin{tikzcd}[sep=small]
MAA \arrow[dd, "1m"'] \arrow[rr, "n^M1"]    &  & MA \arrow[dd, "n^M"] \\
                                            &  &                      \\
MA \arrow[rr, "n^M"'] \arrow[rruu, Rightarrow, "\nu^M", shorten > = 2.5ex, shorten < = 2.5ex] &  & M,                   
\end{tikzcd}$

&

$\begin{tikzcd}[sep=small]
                                  &  & MA \arrow[rrdd, "n^M"] \arrow[dd, Rightarrow, "\rho^M", shorten > = 1ex, shorten < = 2ex] &  &   \\
                                  &  &                                             &  &   \\
M \arrow[rruu, "1i"] \arrow[rrrr,equal] &  & {}                                          &  & M,
\end{tikzcd}$
\end{tabular}
\end{center}

satisfying:

\begin{enumerate}
\item [a.] We have:
\end{enumerate}

\settoheight{\prelim}{\includegraphics[width=52.5mm]{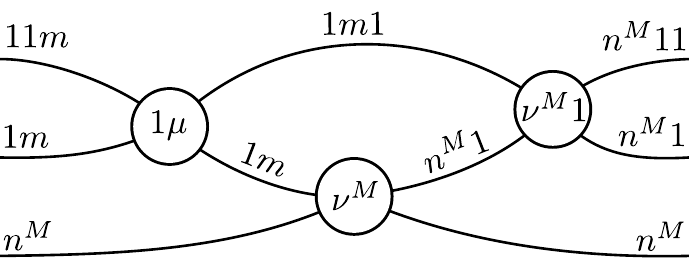}}

\begin{equation}\label{eqn:moduleassociativity}
\begin{tabular}{@{}ccc@{}}

\includegraphics[width=52.5mm]{Pictures/prelim/module/associativity1.pdf} & \raisebox{0.45\prelim}{$=$} &
\includegraphics[width=45mm]{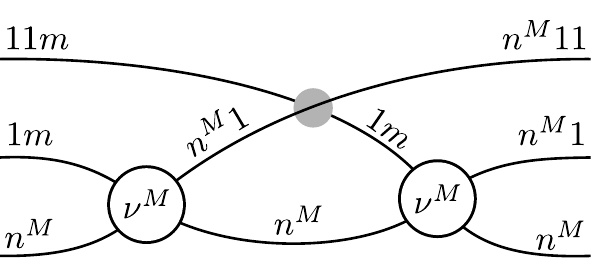},

\end{tabular}
\end{equation}

\begin{enumerate}
\item [b.] We have:
\end{enumerate}

\settoheight{\prelim}{\includegraphics[width=22.5mm]{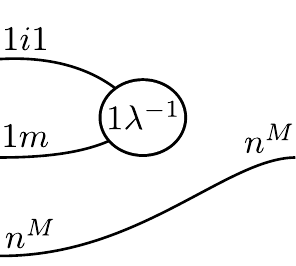}}

\begin{equation}\label{eqn:moduleunitality}
\begin{tabular}{@{}ccc@{}}

\includegraphics[width=22.5mm]{Pictures/prelim/module/unitality1.pdf} & \raisebox{0.45\prelim}{$=$} &

\includegraphics[width=37.5mm]{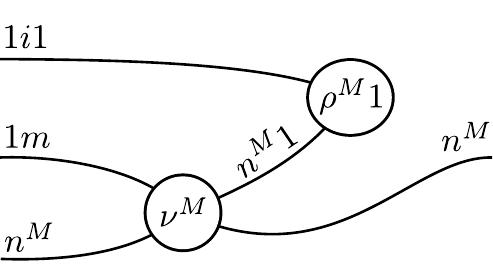}.

\end{tabular}
\end{equation}
\end{Definition}

We will need the following coherence result for right $A$-modules. As we have found no proof in the literature, we include one for completeness.

\begin{Lemma}\label{lem:coherenceright}
Given any right $A$-module $M$, we have the following equality:

\settoheight{\prelim}{\includegraphics[width=37.5mm]{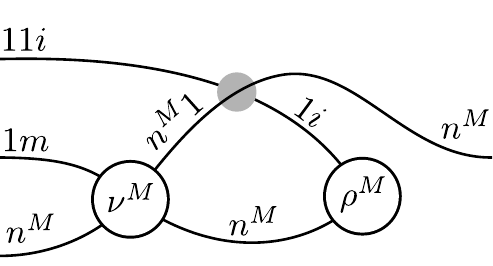}}

\begin{equation}\label{eqn:coherenceright}
\begin{tabular}{@{}ccc@{}}
\includegraphics[width=37.5mm]{Pictures/prelim/coherenceright1.pdf}&
\raisebox{0.45\prelim}{$=$} &
\includegraphics[width=22.5mm]{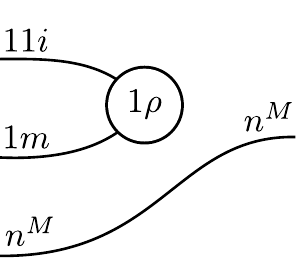}.
\end{tabular}
\end{equation}
\end{Lemma}
\begin{proof}
We begin by proving the following equality:

\settoheight{\prelim}{\includegraphics[width=45mm]{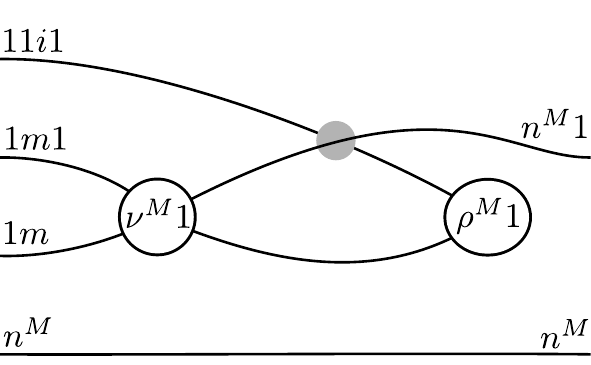}}

\begin{center}
\begin{tabular}{@{}ccc@{}}
\includegraphics[width=45mm]{Pictures/coherenceright/equation1.pdf}&
\raisebox{0.45\prelim}{$=$} &
\includegraphics[width=30mm]{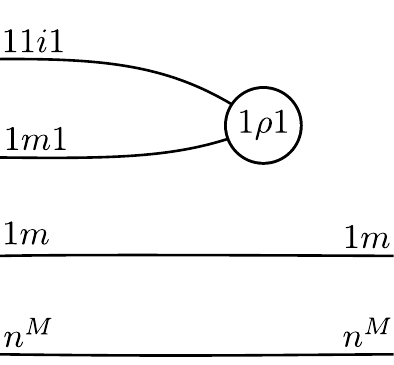}.
\end{tabular}
\end{center}

In order to establish this equality, we use the diagrams depicted in appendix \ref{sub:coherencerightdiagrams}. Let us begin with figure \ref{fig:cohright1}. Applying equation (\ref{eqn:moduleunitality}) to the blue coupon, we arrive at figure \ref{fig:cohright2}. Then, moving the coupon labelled $1\lambda^{-1}$ up, we get to figure \ref{fig:cohright3}. Using equation (\ref{eqn:moduleassociativity}) on the blue coupons yields figure \ref{fig:cohright4}. Finally, equation (\ref{eqn:algebraunitality}) applied to the blue coupons brings us to figure \ref{fig:cohright5}. The desired equality follows from that between figures \ref{fig:cohright1} and \ref{fig:cohright5} by cancelling the bottom coupons labelled $\nu^M$.

Let us now prove that the coherence equation (\ref{eqn:coherenceright}) holds. Figure \ref{fig:cohright10} depicts the left hand-side of equation (\ref{eqn:coherenceright}). Creating a pair of cancelling coupons labelled $\rho^{M^{-1}}$ and $\rho^M$ in the blue region, we obtain figure \ref{fig:cohright11}. Moving the coupon labelled $\rho^{M^{-1}}$ to the left and the top string attached to it up, we arrive at figure \ref{fig:cohright12}. Now, we can use the equality derived above on the blue coupons to get to figure \ref{fig:cohright13}. Bringing the indicated string down gets us to figure \ref{fig:cohright14}. Finally, we can cancel the pair of blue coupons, and this brings us to the right hand-side of equation (\ref{eqn:coherenceright}).
\end{proof}

\begin{Remark}
As $A$ is canonically a right $A$-module, we get a coherence relation for $A$.
\end{Remark}

We also recall the definitions of 1-morphisms and of 2-morphisms of right $A$-modules.

\begin{Definition}\label{def:modulemap}
Let $M$ and $N$ be two right $A$-modules. A right $A$-module 1-morphism consists of a 1-morphism $f:M\rightarrow N$ in $\mathfrak{C}$ together with an invertible 2-morphism

$$\begin{tikzcd}[sep=small]
MA \arrow[dd, "f1"'] \arrow[rr, "n^M"]    &  & M \arrow[dd, "f"] \\
                                            &  &                      \\
NA \arrow[rr, "n^N"'] \arrow[rruu, Rightarrow, "\psi^f", shorten > = 2.5ex, shorten < = 2.5ex] &  & N,                   
\end{tikzcd}$$

subject to the coherence relations:

\begin{enumerate}
\item [a.] We have:
\end{enumerate}

\settoheight{\prelim}{\includegraphics[width=52.5mm]{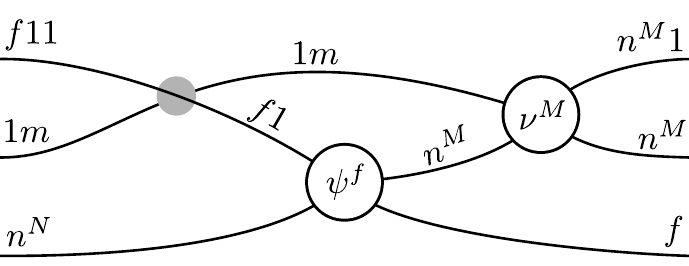}}

\begin{equation}\label{eqn:modulemapassociativity}
\begin{tabular}{@{}ccc@{}}

\includegraphics[width=52.5mm]{Pictures/prelim/module/map1.pdf} & \raisebox{0.45\prelim}{$=$} &

\includegraphics[width=52.5mm]{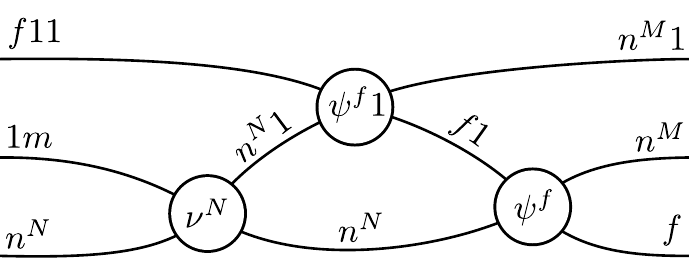},

\end{tabular}
\end{equation}

\begin{enumerate}
\item [b.] We have:
\end{enumerate}

\settoheight{\prelim}{\includegraphics[width=30mm]{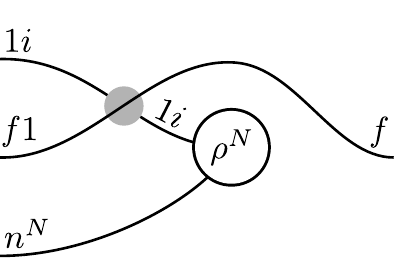}}

\begin{equation}\label{eqn:modulemapunitality}
\begin{tabular}{@{}ccc@{}}

\includegraphics[width=30mm]{Pictures/prelim/module/map3.pdf} & \raisebox{0.45\prelim}{$=$} &

\includegraphics[width=30mm]{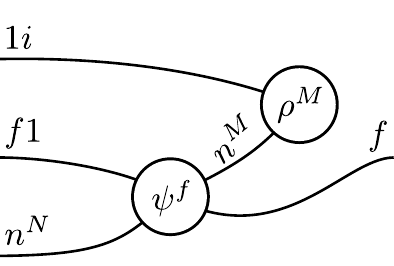}.

\end{tabular}
\end{equation}
\end{Definition}

\begin{Definition}\label{def:moduleintertwiner}
Let $M$ and $N$ be two right $A$-modules, and $f,g:M\rightarrow M$ two right $A$-module 1-morphisms. A right $A$-module 2-morphism $f\Rightarrow g$ is a 2-morphism $\gamma:f\Rightarrow g$ in $\mathfrak{C}$ that satisfies the following equality:

\settoheight{\prelim}{\includegraphics[width=30mm]{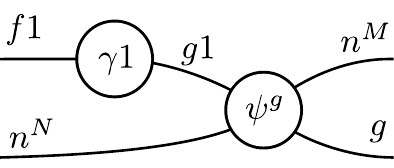}}

\begin{center}
\begin{tabular}{@{}ccc@{}}

\includegraphics[width=30mm]{Pictures/prelim/module/2morphism1.pdf} & \raisebox{0.45\prelim}{$=$} &

\includegraphics[width=30mm]{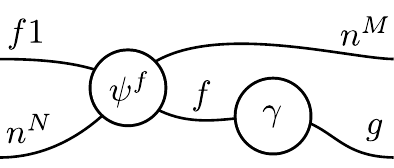}.

\end{tabular}
\end{center}
\end{Definition}

Finally, let us recall lemma 3.2.10 of \cite{D4}.

\begin{Lemma}
Right $A$-modules, right $A$-module 1-morphisms, and right $A$-module 2-morphisms form a 2-category, which we denote by $\mathbf{Mod}_{\mathfrak{C}}(A)$.
\end{Lemma}

\subsection{Left Modules and Bimodules}\label{sub:leftmodulesbimodules}

Let $A$ be an algebra in a strict cubical monoidal 2-category $\mathfrak{C}$. We review the definition of a left $A$-module in $\mathfrak{C}$.

\begin{Definition}\label{def:leftmodule}
A left $A$-module in $\mathfrak{C}$ consists of:
\begin{enumerate}
    \item An object $M$ of $\mathfrak{C}$;
    \item A 1-morphism $l^M:A\Box M\rightarrow M$;
    \item Two 2-isomorphisms
\end{enumerate}
\begin{center}
\begin{tabular}{@{}c c@{}}
$\begin{tikzcd}[sep=small]
M \arrow[rrrr, equal] \arrow[rrdd, "i1"'] &  & {} \arrow[dd, Rightarrow, "\lambda^M"', near start, shorten > = 1ex] &  & M \\
                                   &  &                           &  &   \\
                                   &  & AM, \arrow[rruu, "l^M"']     &  &  
\end{tikzcd}$&

$\begin{tikzcd}[sep=small]
AAM \arrow[dd, "1l^M"'] \arrow[rr, "m1"]    &  & AM \arrow[dd, "l^M"] \\
                                            &  &                      \\
AM \arrow[rr, "l^M"'] \arrow[rruu, Rightarrow, "\kappa^M", shorten > = 2.5ex, shorten < = 2.5ex] &  & M,                   
\end{tikzcd}$
\end{tabular}
\end{center}

satisfying:

\begin{enumerate}
\item [a.] We have:
\end{enumerate}

\settoheight{\prelim}{\includegraphics[width=52.5mm]{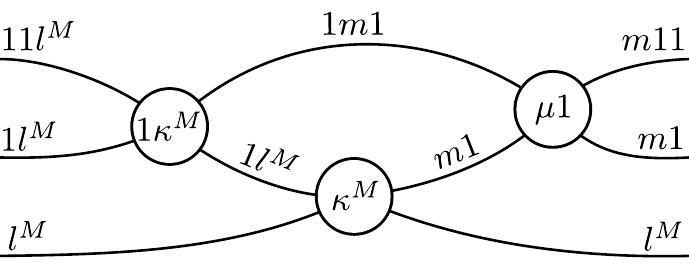}}

\begin{equation}\label{eqn:leftmoduleassociativity}
\begin{tabular}{@{}ccc@{}}

\includegraphics[width=52.5mm]{Pictures/prelim/leftmodule/associativity1.pdf} & \raisebox{0.45\prelim}{$=$} &
\includegraphics[width=45mm]{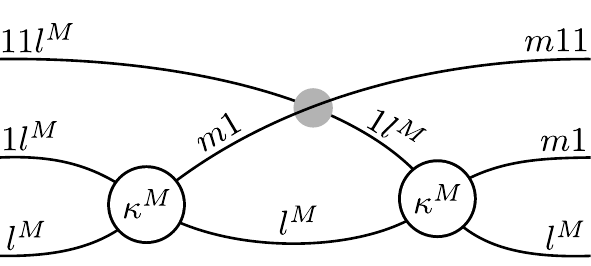},

\end{tabular}
\end{equation}

\begin{enumerate}
\item [b.] We have:
\end{enumerate}

\settoheight{\prelim}{\includegraphics[width=22.5mm]{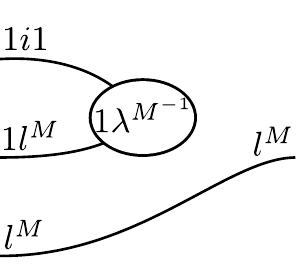}}

\begin{equation}\label{eqn:leftmoduleunitality}
\begin{tabular}{@{}ccc@{}}

\includegraphics[width=22.5mm]{Pictures/prelim/leftmodule/unitality1.pdf} & \raisebox{0.45\prelim}{$=$} &

\includegraphics[width=37.5mm]{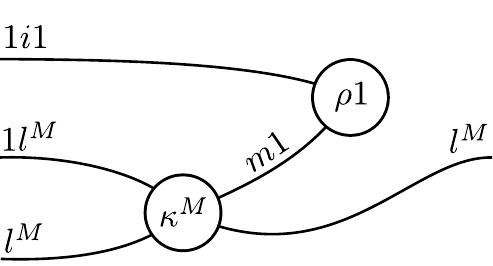}.

\end{tabular}
\end{equation}
\end{Definition}

We also review the definitions 1-morphism and 2-morphism of left $A$-modules.

\begin{Definition}\label{def:leftmodulemap}
Let $M$ and $N$ be two left $A$-modules. A left $A$-module 1-morphism consists of a 1-morphism $f:M\rightarrow N$ in $\mathfrak{C}$ together with an invertible 2-morphism

$$\begin{tikzcd}[sep=small]
AM \arrow[dd, "1f"'] \arrow[rr, "l^M"]    &  & M \arrow[dd, "f"] \\
                                            &  &                      \\
AN \arrow[rr, "l^N"'] \arrow[rruu, Rightarrow, "\chi^f", shorten > = 2.5ex, shorten < = 2.5ex] &  & N,                   
\end{tikzcd}$$

subject to the coherence relations:

\begin{enumerate}
\item [a.] We have:
\end{enumerate}

\settoheight{\prelim}{\includegraphics[width=45mm]{Pictures/prelim/module/map1.pdf}}

\begin{equation}\label{eqn:leftmodulemapassociativity}
\begin{tabular}{@{}ccc@{}}

\includegraphics[width=45mm]{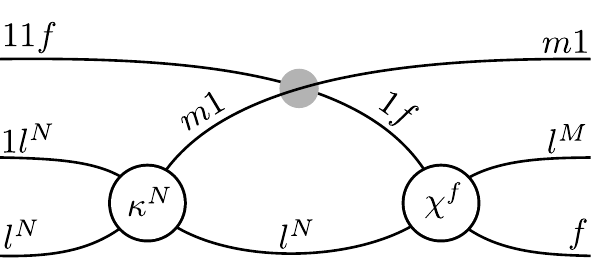} & \raisebox{0.45\prelim}{$=$} &

\includegraphics[width=45mm]{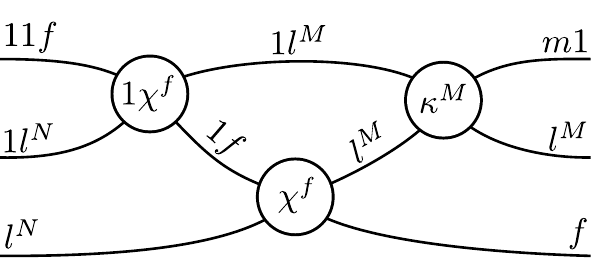},

\end{tabular}
\end{equation}

\begin{enumerate}
\item [b.] We have:
\end{enumerate}

\settoheight{\prelim}{\includegraphics[width=30mm]{Pictures/prelim/module/map3.pdf}}

\begin{equation}\label{eqn:leftmodulemapunitality}
\begin{tabular}{@{}ccc@{}}

\includegraphics[width=30mm]{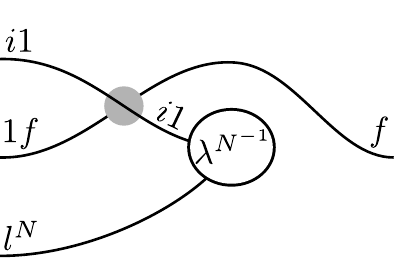} & \raisebox{0.45\prelim}{$=$} &

\includegraphics[width=30mm]{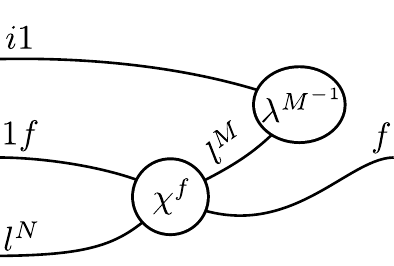}.

\end{tabular}
\end{equation}
\end{Definition}

\begin{Definition}\label{def:leftmoduleintertwiner}
Let $M$ and $N$ be two left $A$-modules, and $f,g:M\rightarrow M$ two left $A$-module 1-morphisms. A left $A$-module 2-morphism $f\Rightarrow g$ is a 2-morphism $\gamma:f\Rightarrow g$ in $\mathfrak{C}$ that satisfies the following equality:

\settoheight{\prelim}{\includegraphics[width=30mm]{Pictures/prelim/module/2morphism1.pdf}}

\begin{center}
\begin{tabular}{@{}ccc@{}}

\includegraphics[width=30mm]{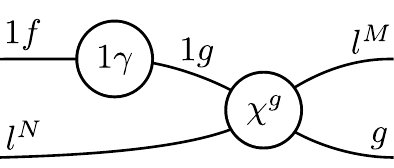} & \raisebox{0.45\prelim}{$=$} &

\includegraphics[width=30mm]{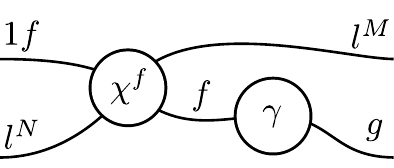}.

\end{tabular}
\end{center}
\end{Definition}

Let now $A$ and $B$ be two algebras in the strict cubical monoidal 2-category $\mathfrak{C}$. In the definition below, we write $m^A$ for the multiplication 1-morphism of $A$ and $m^B$ for that of $B$. We now recall the definition of an $A$-$B$-bimodule in $\mathfrak{C}$. Below, we review the definitions of 1-morphisms and 2-morphisms of $A$-$B$-bimodules.

\begin{Definition}\label{def:bimodule}
An $A$-$B$-bimodule in $\mathfrak{C}$ consists of:
\begin{enumerate}
    \item An object $P$ of $\mathfrak{C}$;
    \item The data $(P,l^P,\lambda^P,\kappa^P)$ of a left $A$-module structure on $P$;
    \item The data $(P,n^P,\nu^P,\rho^P)$ of a right $B$-module structure on $P$;
    \item A 2-isomorphism
\end{enumerate}

$$\begin{tikzcd}[sep=small]
APB \arrow[dd, "1n^P"'] \arrow[rr, "l^P1"]    &  & PB \arrow[dd, "n^P"] \\
                                            &  &                      \\
AP \arrow[rr, "l^P"'] \arrow[rruu, Rightarrow, "\beta^P", shorten > = 2.5ex, shorten < = 2.5ex] &  & M,                   
\end{tikzcd}$$

satisfying:

\begin{enumerate}
\item [a.] We have
\end{enumerate}

\settoheight{\prelim}{\includegraphics[width=52.5mm]{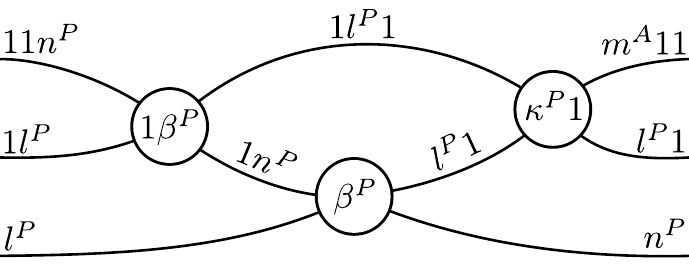}}

\begin{equation}\label{eqn:leftbimoduleassociativity}
\begin{tabular}{@{}ccc@{}}

\includegraphics[width=52.5mm]{Pictures/prelim/leftmodule/bimodule1.pdf} & \raisebox{0.45\prelim}{$=$} &
\includegraphics[width=45mm]{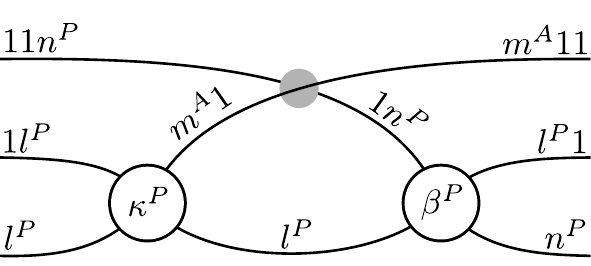},

\end{tabular}
\end{equation}

\begin{enumerate}
\item [b.] We have:
\end{enumerate}

\settoheight{\prelim}{\includegraphics[width=52.5mm]{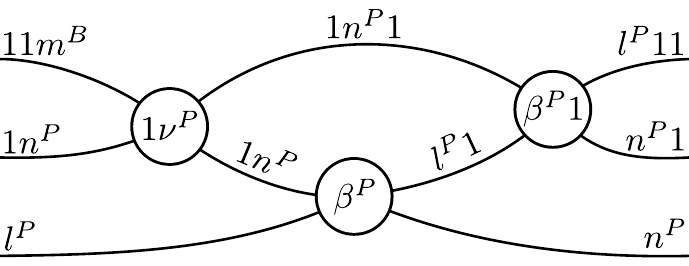}}

\begin{equation}\label{eqn:rightbimoduleassociativity}
\begin{tabular}{@{}ccc@{}}

\includegraphics[width=52.5mm]{Pictures/prelim/leftmodule/bimodule3.pdf} & \raisebox{0.45\prelim}{$=$} &
\includegraphics[width=45mm]{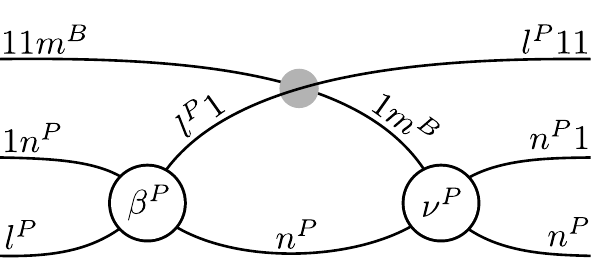}.

\end{tabular}
\end{equation}
\end{Definition}

\begin{Definition}\label{def:bimodulemap}
Let $P$ and $Q$ be two $A$-$B$-bimodules. An $A$-$B$-bimodule 1-morphism consists of a 1-morphism $f:P\rightarrow Q$ in $\mathfrak{C}$ together with the data $(f,\xi^f)$ of a left $A$-module structure and $(f, \psi^f)$ of a right $B$-module structure satisfying:

\settoheight{\prelim}{\includegraphics[width=45mm]{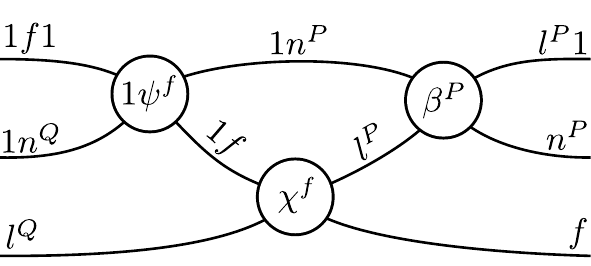}}

\begin{equation}\label{eqn:bimodulemapassociativity}
\begin{tabular}{@{}ccc@{}}

\includegraphics[width=45mm]{Pictures/prelim/leftmodule/bimodulemap1.pdf} & \raisebox{0.45\prelim}{$=$} &

\includegraphics[width=52.5mm]{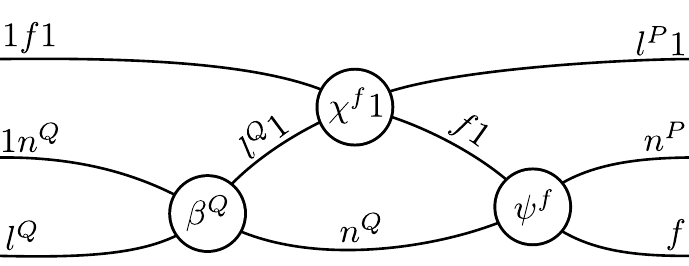}.

\end{tabular}
\end{equation}
\end{Definition}

\begin{Definition}\label{def:bimoduleintertwiner}
Let $P$ and $Q$ be two $A$-$B$-bimodules, and $f,g:P\rightarrow Q$ two $A$-$B$-bimodule 1-morphisms. An $A$-$B$-bimodule 2-morphism $f\Rightarrow g$ is a 2-morphism $\gamma:f\Rightarrow g$ in $\mathfrak{C}$, which is both a left $A$-module 2-morphism and a right $B$-module 2-morphism.
\end{Definition}

We also record the follwoing result.

\begin{Lemma}
Given two algebras $A$ and $B$ in $\mathfrak{C}$, $A$-$B$-bimodules, $A$-$B$-bimodule 1-morphisms, and $A$-$B$-bimodule 2-morphisms form a 2-category, which we denote by $\mathbf{Bimod}_{\mathfrak{C}}(A,B)$.
\end{Lemma}
\begin{proof}
The proof is easy and left to the reader. We only point out that, as we have assumed that $\mathfrak{C}$ is strict cubical, $\mathbf{Bimod}_{\mathfrak{C}}(A,B)$ is in fact a strict 2-category.
\end{proof}

\subsection{Compact Semisimple 2-Categories}

As we will use it later, we begin by recalling from \cite{D1} the unpacked definitions of a 2-condensation and of 2-condensation monad (see \cite{GJF} for the original definitions). As is extensively explained in the aforementionned references, these objects are categorifications of the notions of split idempotent, and idempotent respectively.

\begin{Definition}
Let $\mathfrak{C}$ be a 2-category. A 2-condensation in $\mathfrak{C}$ consists of two objects $A$ and $B$ in $\mathfrak{C}$, two 1-morphisms $f:A\rightarrow B$ and $g:B\rightarrow A$, and two 2-morphisms $\phi:f\circ g\Rightarrow Id_B$ and $\gamma:Id_B\Rightarrow f\circ g$ such that $\phi\cdot \gamma = Id_{Id_B}$.
\end{Definition}

The data of 2-condensation as in the above definition induces a 2-condensation monad on the object $A$. We spell out what this means below.

\begin{Definition}
Let $\mathfrak{C}$ be a 2-category. A 2-condensation monad in $\mathfrak{C}$ consists of an object $A$ together with a 1-morphism $e:A\rightarrow A$, and 2-morphisms $\mu:e\circ e\Rightarrow e$ and $\delta:e\Rightarrow e\circ e$ such that $\mu$ is associative, $\delta$ is coassociative, the Frobenius relations hold, and $\mu\cdot\delta = Id_e$.
\end{Definition}

Let $\mathds{k}$ be a field, and let $\mathfrak{C}$ be a $\mathds{k}$-linear 2-category that is locally Cauchy complete, meaning that its $Hom$-categories have direct sums and splittings of idempotents. We say that $\mathfrak{C}$ is Cauchy complete if it has direct sum for objects and every 2-condensation monad can be extended to a 2-condensation.

We now recall the definition of a semisimple 2-category from \cite{DR}. Then, we will review the definition of a compact semisimple 2-category introduced in \cite{D5}, which is a generalization of the notion of finite semisimple 2-category over an algebraically closed field of characteristic zero introduced in \cite{DR}, and developed in \cite{D1}.

\begin{Definition}
A $\mathds{k}$-linear 2-category $\mathfrak{C}$ is called semisimple if it is locally semisimple, has left and right adjoints for 1-morphisms, and is Cauchy complete.
\end{Definition}

Let $\mathfrak{C}$ be a semisimple 2-category. We call an object $C$ in $\mathfrak{C}$ simple if $Id_C$ is a simple object of the semisimple category $End_{\mathfrak{C}}(C)$. We say that two simple objects $C$ and $D$ in $\mathfrak{C}$ are connected if there exists a nonzero 1-morphism between them. Thanks to the categorical Schur lemma (see \cite{D5} lemma 1.1.5), this defines an equivalence relation on the set of simple objects of $\mathfrak{C}$, and we write $\pi_0(\mathfrak{C})$ for the quotient.

\begin{Definition}
A semisimple $\mathds{k}$-linear 2-category $\mathfrak{C}$ is called compact if it is locally finite semisimple and $\pi_0(\mathfrak{C})$ is finite.
\end{Definition}

If $\mathfrak{C}$ is locally finite semisimple and the set of equivalence classes of simple object of $\mathfrak{C}$ is in fact finite, then we call $\mathfrak{C}$ a finite semisimple 2-category. As explained in \cite{D5}, over a general field, there exists no finite semisimple 2-category. However, by corollary 2.2.3 of \cite{D5}, if $\mathds{k}$ is algebraically closed or real closed, then every compact semisimple 2-category is finite. Finally, let us also recall the following definition.

\begin{Definition}
A tensor 2-category is a rigid monoidal 2-category. A multifusion 2-category is a finite semisimple tensor 2-category. A fusion 2-category is a multifusion 2-category whose monoidal unit is a simple object.
\end{Definition}

\subsection{Examples}

Let $\mathds{k}$ be a field. We give various examples of algebras in monoidal 2-categories, and in compact semisimple monoidal 2-categories over $\mathds{k}$.

\begin{Example}\label{ex:algebrasfinite}
Let $\mathds{k}$ be an arbitrary field. Recall that a $\mathds{k}$-linear 1-category is called finite if it is equivalent to the 1-category $Mod(R)$ of finite dimensional left $R$-modules with $R$ a finite dimensional $\mathds{k}$-algebra. We write $\mathbf{FinCat}$ for the 2-category of finite $\mathds{k}$-linear 1-categories, right exact functors, and natural transformations. It is well-known that the Deligne tensor product $\boxtimes$, defined by $Mod(R)\boxtimes Mod(S)\simeq Mod(R\otimes_{\mathds{k}} S)$ for any finite dimensional $\mathds{k}$-algebras $R$ and $S$, endows $\mathbf{FinCat}$ with a symmetric monoidal structure (for instance, see \cite{EGNO}). It is easy to check that algebras in $\mathbf{FinCat}$ are precisely finite monoidal 1-categories, whose monoidal product is right exact in both arguments. Further, fixing such a finite monoidal 1-category $\mathcal{A}$, right $\mathcal{A}$-modules in $\mathbf{FinCat}$ are exactly finite right $\mathcal{A}$-module 1-categories, for which the action is right exact in both variables.
\end{Example}

\begin{Example}\label{ex:algebras2Vect}
Let us call a $\mathds{k}$-linear 1-category $\mathcal{C}$ is perfect if it is finite semisimple and the endomorphism algebra of any object is separable. This corresponds to requiring that $\mathcal{C}$ is equivalent to $Mod(R)$ with $R$ a separable $\mathds{k}$-algebra. We write $\mathbf{2Vect}$ for the compact semisimple 2-category of perfect 1-categories, linear functors, and natural transformations. As separable $\mathds{k}$-algebras are closed under $\otimes_{\mathds{k}}$, we find that the Deligne tensor product preserves perfect 1-categories. This shows that $\mathbf{2Vect}$ is a symmetric monoidal full sub-2-category of $\mathbf{FinCat}$. Algebras in $\mathbf{2Vect}$ correspond precisely to perfect monoidal 1-categories. Fixing such an algebra $\mathcal{A}$, right $\mathcal{A}$-modules in $\mathbf{2Vect}$ are exactly perfect right $\mathcal{A}$-module 1-categories.
\end{Example}

\begin{Remark}
If we assume that the field $\mathds{k}$ is perfect, then every finite semisimple 1-category is perfect, so that $\mathbf{2Vect}$ coincides with the symmetric monoidal 2-category of finite semisimple 1-categories. In fact, it follows from the fact that a $\mathds{k}$-algebra $R$ is separable if and only $R\otimes R$ is semisimple that a finite semisimple 1-category $\mathcal{A}$ is perfect if and only if $\mathcal{A}\boxtimes\mathcal{A}$ is finite semisimple. However, for an arbitrary field, $\mathbf{FinCat_{ss}}$, the full sub-2-category on the finite semisimple 1-categories, is not closed under the Deligne tensor product. This last 2-category also fails to be a semisimple 2-category, though it is Cauchy complete as we now prove for later use.
\end{Remark}

\begin{Lemma}\label{lem:sscatCauchy}
The 2-category $\mathbf{FinCat_{ss}}$ is Cauchy complete.
\end{Lemma}
\begin{proof}
Let us begin by noting that $\mathbf{FinCat_{ss}}$ has right and left adjoints for 1-morphisms, as every functor is exact. Further, recall that $\mathbf{FinCat_{ss}}$ is equivalent to the 2-category of finite semisimple $\mathds{k}$-algebras, finite dimensional $\mathds{k}$-bimodules, and homomorphisms of bimodules. Under this perspective, it is clear that $\mathbf{FinCat_{ss}}$ is locally Cauchy complete and has direct sum for objects. It therefore only remains to prove that every 2-condensation monad splits. By theorem 3.1.4 of \cite{GJF}, it is in fact enough to check that unital 2-condensation monads splits.

Let $(A,e,\mu,\delta,\iota)$ be a unital 2-condensation monad in $\mathbf{FinCat_{ss}}$, i.e. a 2-condensation monad such that $\iota$ is a unit for $\mu$. Unfolding what this means, we find that the $A$-$A$-bimodule $e$ has the structure of a finite dimensional $\mathds{k}$-algebra, which we denote by $E$. Further, the algebra $E$ is equipped with a morphism $A\rightarrow E$ of $\mathds{k}$-algebras such that the canonical map $E\otimes_AE\rightarrow E$ is split by $\delta$ as a map of $E$-$E$-bimodules. We claim that $E$ is a semisimple algebra. Namely, let $M$ be a left $E$-module. Then, the canonical map $E\otimes_AM\rightarrow M$ splits. But, as $A$ is semisimple, $M$ viewed as a left $A$-module is a direct summand of a $A^n$ for some $n\geq 0$. This implies that $M$ is a direct summand of $E^n$, so that $E$ is semisimple.

In order to conclude the proof, let $B:= E$ as an object of $\mathbf{FinCat_{ss}}$, $f$ be the 1-morphism from $A$ to $B$ given by the $A$-$E$-bimodule $E$, $g$ the 1-morphism from $B$ to $A$ given by the $E$-$A$-bimodule $E$, $\phi$ the 2-morphism corresponding to the map of $E$-$E$-bimodules ${_EE}\otimes_A E_E\rightarrow {_EE_E}$ induced by the multiplication of $E$, and $\gamma$ the 2-morphism corresponding to the map of $E$-$E$-bimodules ${_EE_E}\rightarrow {_EE}\otimes_A E_E$ given by $\delta$. This defines a 2-condensation, which extends $(A,e,\mu,\delta,\iota)$ as desired.
\end{proof}

Example \ref{ex:algebras2Vect} can be generalized in various different directions.

\begin{Example}\label{ex:algebras2VectG}
Let us fix $G$ a finite group. We write $\mathbf{2Vect}_G$ for the compact semisimple tensor 2-category of $G$-graded perfect 1-categories. It is easy to check that algebras in $\mathbf{2Vect}_G$ are exactly given by $G$-graded perfect monoidal categories such that the monoidal structure preserves the grading. Right modules admit a similar explicit description. More generally, if we are in addition given a 4-cocycle $\omega$ for $G$ with coefficient in $\mathds{k}^{\times}$, we can use $\omega$ to twist the coherence of the monoidal 2-category $\mathbf{2Vect}_G$ (see construction 2.1.16 of \cite{DR} and \cite{Delc}). Doing so, we get another compact semisimple tensor 2-category, which we denote by $\mathbf{2Vect}_G^{\omega}$. If $H\subseteq G$ is a subgroup such that there exists a 3-cochain $\gamma:H\times H\times H\rightarrow \mathds{k}^{\times}$ with $d\gamma = \omega|_H$, then $\mathbf{Vect}_H^{\gamma}$, the monoidal 1-category of finite dimensional $H$-graded $\mathds{k}$-vector spaces with associator twisted by $\gamma$, is an algebra in $\mathbf{2Vect}_G^{\omega}$. The 2-category of right modules over $\mathbf{Vect}_H^{\gamma}$ is equivalent to $\mathbf{2Vect}_{G/H}$.
\end{Example}

\begin{Example}\label{ex:algebrasModB}
Let $\mathcal{B}$ be a braided finite semisimple tensor 1-category. Recall that a finite semisimple right $\mathcal{B}$-module 1-category is called separable if it is equivalent to the 1-category of left modules over a separable algebra in $\mathcal{B}$. We let $\mathbf{Mod}(\mathcal{B})$ denote the compact semisimple 2-category of separable right $\mathcal{B}$-module 1-categories. The relative Deligne tensor product $\boxtimes_{\mathcal{B}}$ endows this 2-category with a rigid monoidal structure, see \cite{DR} and \cite{D5}. Following \cite{BJS}, a $\mathcal{B}$-central monoidal 1-category is a monoidal linear 1-category $\mathcal{C}$ equipped with a braided monoidal functor $F:\mathcal{B}\rightarrow \mathcal{Z}(\mathcal{C})$ to the Drinfeld center of $\mathcal{C}$. This notion has also appeared under different names in \cite{DGNO}, \cite{HPT} and \cite{MPP}. Note that this provides both a right and a left $\mathcal{B}$-module structure on $\mathcal{C}$. We will call a finite semisimple $\mathcal{B}$-central monoidal 1-category $\mathcal{C}$ separable if $\mathcal{C}$ is a separable as a (left, or equivalently right) $\mathcal{B}$-module 1-category. It follows from proposition 3.2 of \cite{BJS} that algebras in $\mathbf{Mod}(\mathcal{B})$ correspond precisely to separable $\mathcal{B}$-central monoidal 1-categories. If $\mathcal{C}$ is such an algebra, it is easy to check that right $\mathcal{C}$-modules in $\mathbf{Mod}(\mathcal{B})$ are exactly finite semisimple right $\mathcal{C}$-module 1-categories, which are separable as a right $\mathcal{B}$-module 1-category.
\end{Example}

\begin{Example}\label{ex:algebrascenter}
If $\mathds{k}$ is an algebraically closed field of characteristic zero and $G$ is a finite group. Then, following \cite{BN}, we can consider the Drinfel'd center $\mathscr{Z}(\mathbf{2Vect}_G)$ of the fusion 2-category $\mathbf{2Vect}_G$, which is a braided monoidal 2-category. This Drinfeld center was shown to be a finite semisimple monoidal 2-category\footnote{It is in fact a fusion 2-category. Rigidity was established in corollary 2.2.2 of \cite{D9}.} in \cite{KTZ}. Algebras in $\mathscr{Z}(\mathbf{2Vect}_G)$ are in particular $G$-graded finite semisimple monoidal 1-categories. But, objects of this Drinfel'd center are in addition equipped with a coherent action of $G$, which is given by conjugation on grading. Thus, algebras in $\mathscr{Z}(\mathbf{2Vect}_G)$ are exactly finite semisimple $G$-crossed monoidal 1-categories (originally introduced in section 2.1 of \cite{T}, see also definition 5.1 of \cite{Gal}). This means that $G$-crossed braided fusion 1-categories are algebras in $\mathscr{Z}(\mathbf{2Vect}_G)$.
\end{Example}

\section{Properties of Rigid Algebras}\label{sec:rigid}

\subsection{Definitions in Detail}

We fix $\mathfrak{C}$ a monoidal 2-category, which we assume to be strict cubical without loss of generality. Following \cite{G} (see also \cite{JFR}), we say that an algebra $A$ in $\mathfrak{C}$ is rigid if its multiplication map $m:A\Box A\rightarrow A$ has a right adjoint as a map of $A$-$A$-bimodules. For later use, we spell out what this means in details.

\begin{Definition}
A rigid algebra in $\mathfrak{C}$ consists of:
\begin{enumerate}
    \item An algebra $A$ in $\mathfrak{C}$ as in definition \ref{def:algebra};
    \item A right adjoint $m^*:A\rightarrow A\Box A$  in $\mathfrak{C}$ to the multiplication map $m$ with unit $\eta^m$ and counit $\epsilon^m$ (depicted below as a cup and a cap);
    \item Two 2-isomorphisms
\end{enumerate}
    
\begin{center}
\begin{tabular}{cc}
$\begin{tikzcd}[sep=tiny]
AA \arrow[rrr, "m"] \arrow[ddd, "m^*1"'] &                                     &    & A \arrow[ddd, "m^*"] \\
                                                     &  &  {}    &                              \\
                                                   &       {} \arrow[ur, "\psi^r", Rightarrow]                              &  &                              \\
AAA \arrow[rrr, "1m"']                                &                                     &    & AA,                         
\end{tikzcd}$ & 
$\begin{tikzcd}[sep=tiny]
AA \arrow[rrr, "m"] \arrow[ddd, "1m^*"'] &                                     &    & A \arrow[ddd, "m^*"] \\
                                                     &  &  {}    &                              \\
                                                   &       {} \arrow[ur, "\psi^l", Rightarrow]                              &  &                              \\
AAA \arrow[rrr, "m1"']                                &                                     &    & AA;                       
\end{tikzcd}$
\end{tabular}
\end{center}
    
\begin{enumerate}
\item[] satisfying:
\item [a.] The 2-morphism $\psi^l$ endow $m^*$ with the structure of a left $A$-module 1-morphism:
\end{enumerate}

\newlength{\algebra}
\settoheight{\algebra}{\includegraphics[width=45mm]{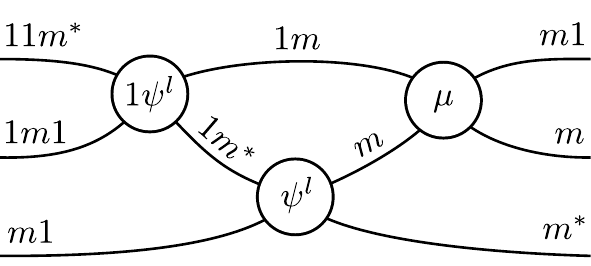}}

\begin{equation}\label{eqn:rigidleftassociativity}
\begin{tabular}{@{}ccc@{}}
\includegraphics[width=45mm]{Pictures/rigidalgebra/leftmodulemap1.pdf}&
\raisebox{0.45\algebra}{$=$} &
\includegraphics[width=45mm]{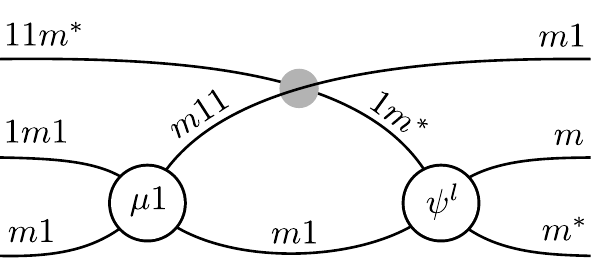},
\end{tabular}
\end{equation}

\settoheight{\algebra}{\includegraphics[width=30mm]{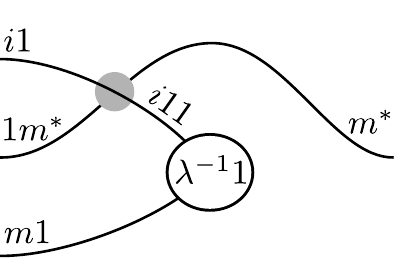}}

\begin{equation}\label{eqn:rigidleftunit}
\begin{tabular}{@{}ccc@{}}
\includegraphics[width=30mm]{Pictures/rigidalgebra/leftmodulemap3.pdf}&
\raisebox{0.45\algebra}{$=$} &
\includegraphics[width=30mm]{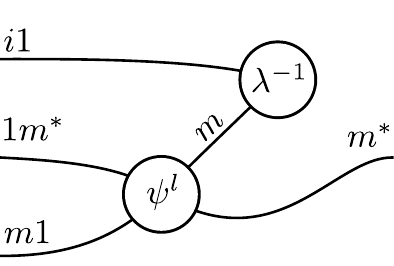},
\end{tabular}
\end{equation}

\begin{enumerate}
\item [b.] The 2-morphism $\psi^r$ endow $m^*$ with the structure of a right $A$-module 1-morphism:
\end{enumerate}

\settoheight{\algebra}{\includegraphics[width=45mm]{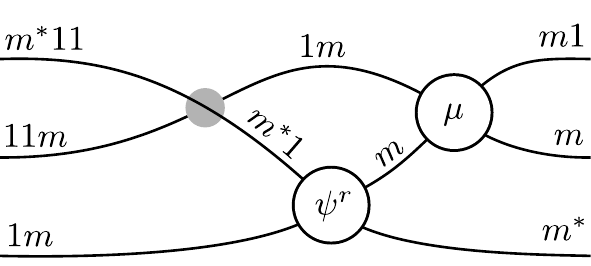}}

\begin{equation}\label{eqn:rigidrightassociativity}
\begin{tabular}{@{}ccc@{}}
\includegraphics[width=45mm]{Pictures/rigidalgebra/rightmodulemap1.pdf}&
\raisebox{0.45\algebra}{$=$} &
\includegraphics[width=45mm]{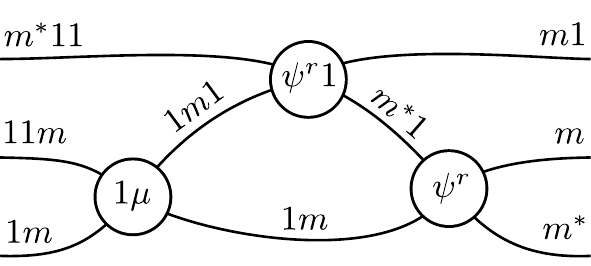},
\end{tabular}
\end{equation}

\settoheight{\algebra}{\includegraphics[width=30mm]{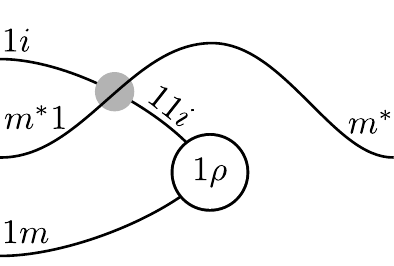}}

\begin{equation}\label{eqn:rigidrightunit}
\begin{tabular}{@{}ccc@{}}
\includegraphics[width=30mm]{Pictures/rigidalgebra/rightmodulemap3.pdf}&
\raisebox{0.45\algebra}{$=$} &
\includegraphics[width=30mm]{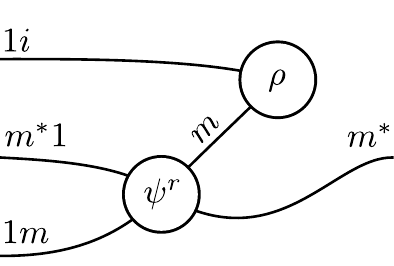},
\end{tabular}
\end{equation}

\begin{enumerate}
\item [c.] The structures of left and right $A$-module 1-morphisms on $m^*$ constructed above are compatible, i.e. they turn $m^*$ into an $A$-$A$-bimodule 1-morphism:
\end{enumerate}

\settoheight{\algebra}{\includegraphics[width=45mm]{Pictures/rigidalgebra/leftmodulemap1.pdf}}

\begin{equation}\label{eqn:rigidbimodule}
\begin{tabular}{@{}ccc@{}}
\includegraphics[width=45mm]{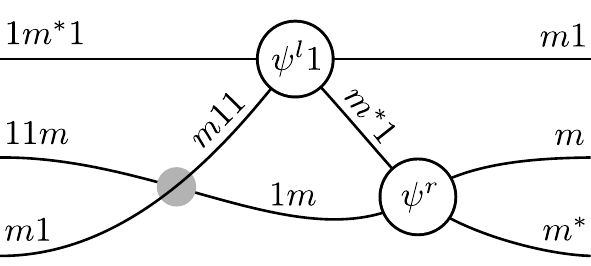}&
\raisebox{0.45\algebra}{$=$} &
\includegraphics[width=45mm]{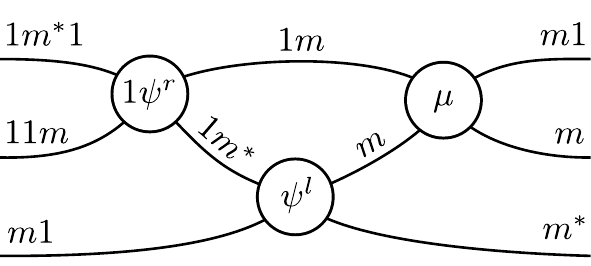},
\end{tabular}
\end{equation}

\begin{enumerate}
\item [d.] The 2-morphism $\epsilon^m$ is an $A$-$A$-bimodule 2-morphism:
\end{enumerate}

\settoheight{\algebra}{\includegraphics[width=22.5mm]{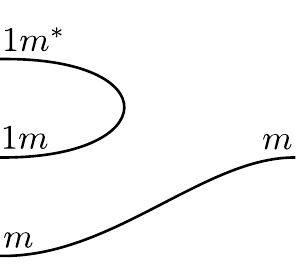}}

\begin{equation}\label{eqn:epsilonleft}
\begin{tabular}{@{}ccc@{}}
\includegraphics[width=22.5mm]{Pictures/rigidalgebra/epsilon1.pdf}&
\raisebox{0.45\algebra}{$=$} &
\includegraphics[width=45mm]{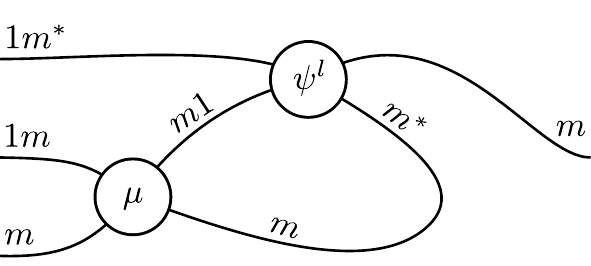},
\end{tabular}
\end{equation}

\settoheight{\algebra}{\includegraphics[width=22.5mm]{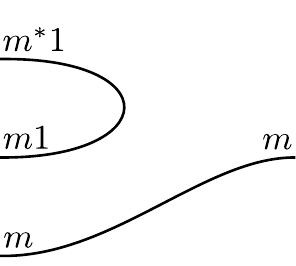}}

\begin{equation}\label{eqn:epsilonright}
\begin{tabular}{@{}ccc@{}}
\includegraphics[width=22.5mm]{Pictures/rigidalgebra/epsilon3.pdf}&
\raisebox{0.45\algebra}{$=$} &
\includegraphics[width=45mm]{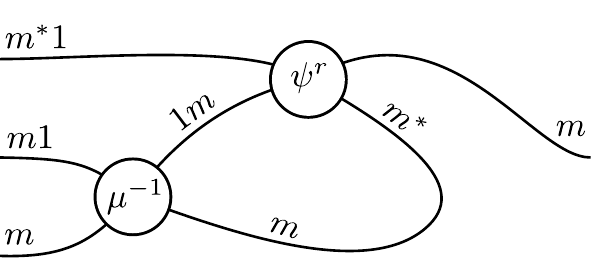},
\end{tabular}
\end{equation}

\begin{enumerate}
\item [e.] The 2-morphism $\eta^m$ is an $A$-$A$-bimodule 2-morphism:
\end{enumerate}

\settoheight{\algebra}{\includegraphics[width=45mm]{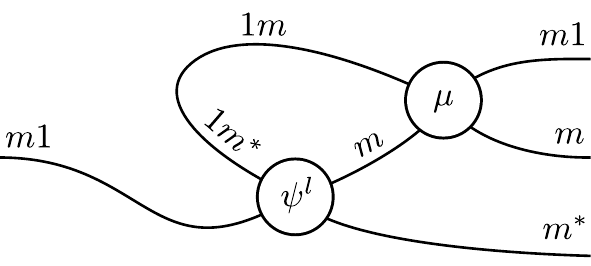}}

\begin{equation}\label{eqn:etaleft}
\begin{tabular}{@{}ccc@{}}
\includegraphics[width=45mm]{Pictures/rigidalgebra/eta1.pdf}&
\raisebox{0.45\algebra}{$=$} &
\includegraphics[width=22.5mm]{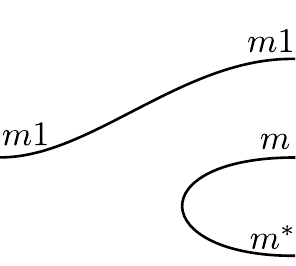},
\end{tabular}
\end{equation}

\settoheight{\algebra}{\includegraphics[width=45mm]{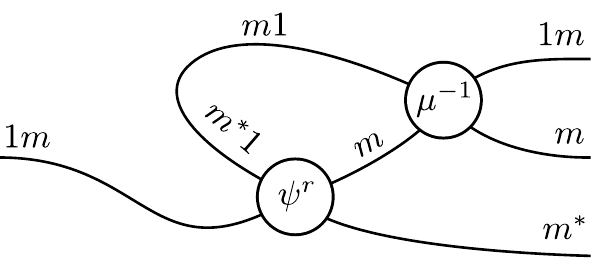}}

\begin{equation}\label{eqn:etaright}
\begin{tabular}{@{}ccc@{}}
\includegraphics[width=45mm]{Pictures/rigidalgebra/eta3.pdf}&
\raisebox{0.45\algebra}{$=$} &
\includegraphics[width=22.5mm]{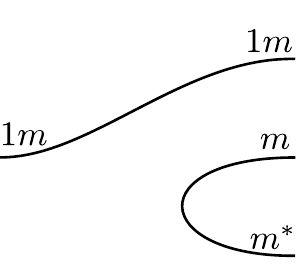}.
\end{tabular}
\end{equation}

\end{Definition}

As we now explain, many familiar objects are rigid algebra in a certain monoidal 2-category.

\begin{Example}
By proposition 1.3 of \cite{BJS}, a finite monoidal 1-category $\mathcal{A}$ is tensor, i.e.\ has right and left duals for objects, if and only if $\mathcal{A}$ is a rigid algebra in the sense of the above definition. In particular, every finite semisimple tensor 1-category is a rigid algebra in $\mathbf{FinCat}$.
\end{Example}

\begin{Example}\label{ex:gradedrigid}
Fixing $G$ a finite group, a slight elaboration on the argument of the above example shows that rigid algebras in $\mathbf{2Vect}_G$ correspond exactly to perfect $G$-graded tensor 1-categories.
\end{Example}

Let $\mathcal{B}$ be a finite semisimple braided tensor 1-category. We have seen in example \ref{ex:algebrasModB} above that algebras in $\mathbf{Mod}(\mathcal{B})$ are exactly given by separable $\mathcal{B}$-central monoidal 1-categories. We now characterize those algebras that are rigid.

\begin{Lemma}
A separable $\mathcal{B}$-central monoidal 1-category $\mathcal{C}$ is a rigid algebra in $\mathbf{Mod}(\mathcal{B})$ if $\mathcal{C}$ is tensor. If the underlying 1-category of $\mathcal{C}$ is perfect, this holds if and only if $\mathcal{C}$ is tensor.
\end{Lemma}
\begin{proof}
As $\mathcal{C}$ is separable as a $\mathcal{B}$-module 1-category, $\mathcal{C}\boxtimes_{\mathcal{B}}\mathcal{C}$ is finite semisimple. In particular, the canonical functor $m:\mathcal{C}\boxtimes_{\mathcal{B}}\mathcal{C}\rightarrow \mathcal{C}$ has a right adjoint. Thus, if $\mathcal{C}$ is rigid, it follows from \cite{DSPS14} that $m$ has a right adjoint as a $\mathcal{C}$-$\mathcal{C}$-bimodule functor. This proves the first part of the statement.

Let us now assume that $\mathcal{C}$ is perfect. Observe that it is enough to prove that the canonical $\mathcal{C}$-$\mathcal{C}$-bimodule functor $G:\mathcal{C}\boxtimes\mathcal{C}\rightarrow \mathcal{C}\boxtimes_{\mathcal{B}}\mathcal{C}$ has a $\mathcal{C}$-$\mathcal{C}$-bimodule right adjoint. As both the target and the domain of $G$ are finite semisimple, we find that $G$ is exact, so that it has a right adjoint $G^*$ as a linear functor. Further, by separability of $\mathcal{C}$ as a $\mathcal{B}$-module 1-category, $End_{\mathcal{B}}(\mathcal{C})$ is a finite semisimple tensor 1-category, thence it follows from \cite{DSPS14} that $G^*$ has an $End_{\mathcal{B}}(\mathcal{C})$-$End_{\mathcal{B}}(\mathcal{C})$-bimodule structure. Restricting this bimodule structure on $G^*$ along the monoidal functor $\mathcal{C}\rightarrow End_{\mathcal{B}}(\mathcal{C})$ proves the claim.
\end{proof}

\begin{Example}
If $\mathds{k}$ is an algebraically closed field of characteristic zero and $G$ is a finite group, then rigid algebras in $\mathscr{Z}(\mathbf{2Vect}_G)$ are exactly $G$-crossed multifusion 1-categories.
\end{Example}

\begin{Remark}
It is well-known that algebras in $\mathbf{Cat}$ the 2-category of 1-categories equipped with the Cartesian product as monoidal structure are precisely monoidal 1-categories. One can check directly that rigid algebras in $\mathbf{Cat}$ are all equivalent to the monoidal 1-category with one object and one morphism. This also follows from the general observation that the object underlying a rigid algebra is always self-dual. Thus, in this context, the classical notion of rigidity is more appropriate.
\end{Remark}

We also give in details the definition of a separable algebra 

\begin{Definition}
A separable algebra in $\mathfrak{C}$ is a rigid algebra $A$ in $\mathfrak{C}$ equipped with a 2-morphism $\gamma^m:Id_A\Rightarrow m\circ m^*$ such that:

\begin{enumerate}
\item [a.] The 2-morphism $\gamma^m$ is a section of $\epsilon^m$, i.e. $\epsilon^m\cdot\gamma^m = Id_{Id_A}$,
\item [b.] The 2-morphism $\gamma^m$ is an $A$-$A$-bimodule 2-morphism:
\end{enumerate}

\settoheight{\algebra}{\includegraphics[width=45mm]{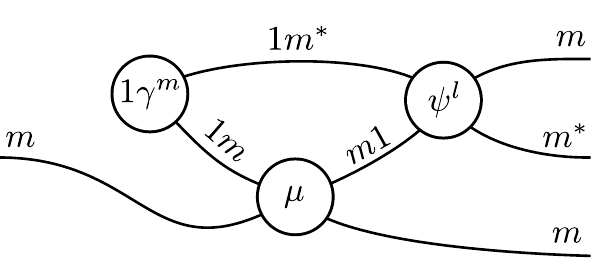}}

\begin{equation}\label{eqn:gammaleft}
\begin{tabular}{@{}ccc@{}}
\includegraphics[width=45mm]{Pictures/rigidalgebra/gamma1.pdf}&
\raisebox{0.45\algebra}{$=$} &
\includegraphics[width=22.5mm]{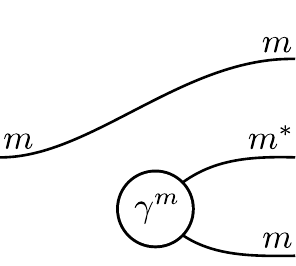},
\end{tabular}
\end{equation}

\settoheight{\algebra}{\includegraphics[width=45mm]{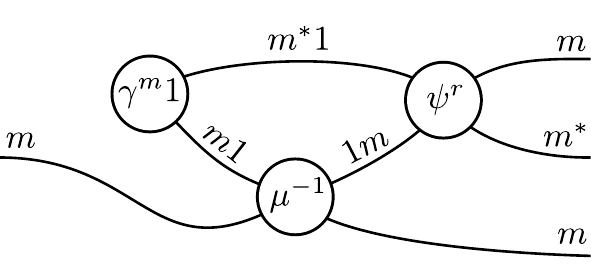}}

\begin{equation}\label{eqn:gammaright}
\begin{tabular}{@{}ccc@{}}
\includegraphics[width=45mm]{Pictures/rigidalgebra/gamma3.pdf}&
\raisebox{0.45\algebra}{$=$} &
\includegraphics[width=22.5mm]{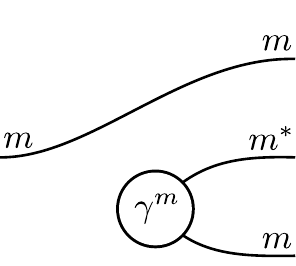}.
\end{tabular}
\end{equation}
\end{Definition}

We will explore this notion further in the next section, which will allow us to give many examples in section \ref{sub:separableexamples}.

\begin{Remark}
Any separable algebra is a 3-condensation monad in the sense of \cite{GJF}. In fact, if $\mathfrak{C}$ is a compact semisimple rigid monoidal 2-category, then it follows from the proof of theorem 3.17 of \cite{GJF} that every 3-condensation monad is equivalent to a separable algebra. We decide to work with separable algebras because they are more familiar algebraic objects.
\end{Remark}

\subsection{Adjoints in 2-Categories of (Bi)Modules}\label{sub:adjoints1morphisms}

Let us fix $A$ a rigid algebra in the strict cubical monoidal 2-category $\mathfrak{C}$, and let $M$ be a right $A$-module. Note that the 1-morphism $n^M$ has a canonical structure of a right $A$-module 1-morphism supplied by $\nu^M$. Our first goal is to show that $n^M$ has a right adjoint as a right $A$-module 1-morphism. 
We begin by defining $$p^M:= (n\Box A)\circ (M\Box m^*)\circ (M\Box i):M\rightarrow M\Box A,$$ a 1-morphism in $\mathfrak{C}$, together with a 2-morphism in $\mathfrak{C}$ given by

\settoheight{\algebra}{\includegraphics[width=75mm]{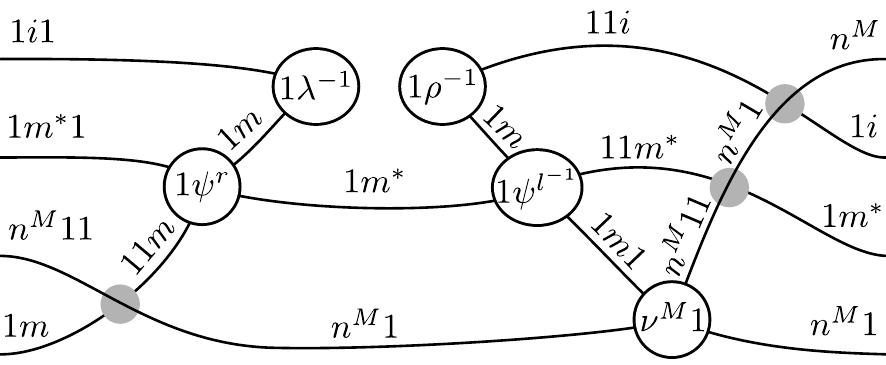}}

\begin{center}
\begin{tabular}{@{}cc@{}}

\raisebox{0.45\algebra}{$\psi^p:=$} &
\includegraphics[width=75mm]{Pictures/technicallemmas/psip.pdf}.
\end{tabular}
\end{center}

\begin{Lemma}\label{lem:modulep}
The 2-morphism $\psi^p$ endows the 1-morphism $p^M$ with the structure of a right $A$-module 1-morphism.
\end{Lemma}
\begin{proof}
We need to check that $\psi^p$ satisfies the two coherence axioms of definition \ref{def:modulemap}. In order to do this, we will use the diagrams depicted in appendix \ref{sub:modulepdiagrams}. We begin by checking axiom a: Figure \ref{fig:cohpsip1} depicts the right hand-side of axiom a for $\psi^p$. Using naturality to move the coupons labelled $1\mu$ and $\nu^M11$, we find that this 2-isomorphism is equal to that depicted in figure \ref{fig:cohpsip2}. Now, we use equation (\ref{eqn:rigidrightassociativity}) on the blue coupons, and equation (\ref{eqn:moduleassociativity}) on the green ones to get that this is equal to the 2-isomorphism given in figure \ref{fig:cohpsip3}. In order to get to figure \ref{fig:cohpsip4}, we apply equation (\ref{eqn:rigidleftassociativity}) to the blue coupons, and use naturality to move the coupon labelled $\nu^M11$. Using equation (\ref{eqn:rigidbimodule}) on the blue coupons, we find that this 2-isomorphism is equal to that depicted in figure \ref{fig:cohpsip5}. Equation (\ref{eqn:algebraunitality}) shows that this string diagram represents the same 2-isomorphism as that given in figure \ref{fig:cohpsip6}. Using equation (\ref{eqn:coherenceleft}) on the blue coupons, and equation (\ref{eqn:coherenceright}) for $A$ on the green coupons, we get that this is equal to figure \ref{fig:cohpsip7}. But, figure \ref{fig:cohpsip7} manifestly depicts the left hand-side of axiom a of definition \ref{def:modulemap}.

We proceed to check  that axiom b also holds. Firstly, note that figure \ref{fig:cohupsip1} depicts the right hand-side of axiom b of definition \ref{def:modulemap}. Naturality implies that the 2-isomorphism described by this string diagram is the same as that described by figure \ref{fig:cohupsip2}. Using equation (\ref{eqn:moduleunitality}) on the blue coupons, we find that this also corresponds to the figure \ref{fig:cohupsip3}. Equation (\ref{eqn:rigidleftunit}) above applied to the blue coupons shows further that it is equal to the 2-isomorphism depicted by figure \ref{fig:cohupsip4}. Using equation \ref{eqn:coherencemiddle} twice, once on the blue coupon, and once on the green coupons, we find that figure \ref{fig:cohupsip5} also describes the same 2-isomorphism. To derive the equality with figure \ref{fig:cohupsip6}, we use equation (\ref{eqn:rigidrightunit}) on the blue coupons. Finally, naturality allows us to assert that the 2-isomorphism described by figure \ref{fig:cohupsip7} agrees with the previous ones. As this last string diagram also corresponds to the left hand-side of axiom b of definition \ref{def:modulemap}, we are done.
\end{proof}

We now define two 2-morphisms in $\mathfrak{C}$, which will be the unit and counit of the desired adjunction:

\settoheight{\algebra}{\includegraphics[width=60mm]{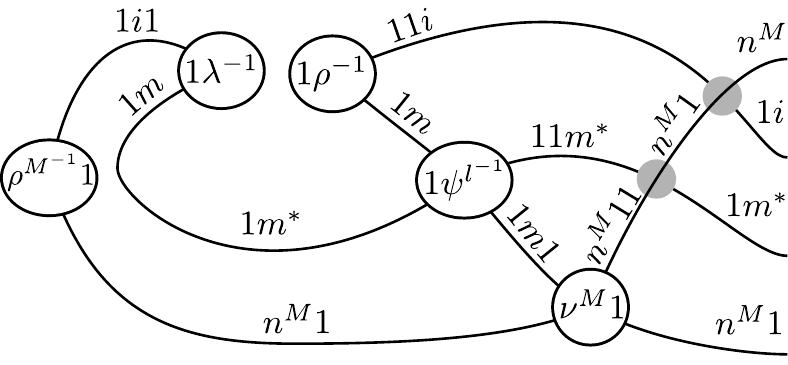}}

\begin{center}
\begin{tabular}{@{}cc@{}}

\raisebox{0.45\algebra}{$\eta^M:=$} &
\includegraphics[width=60mm]{Pictures/technicallemmas/etaM.pdf}.
\end{tabular}
\end{center}

\settoheight{\algebra}{\includegraphics[width=45mm]{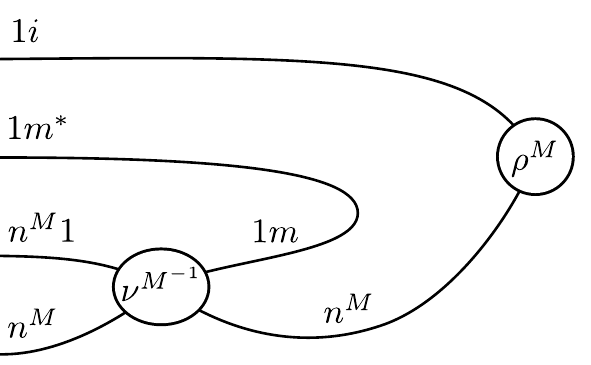}}

\begin{center}
\begin{tabular}{@{}cc@{}}

\raisebox{0.45\algebra}{$\epsilon^M:=$} &
\includegraphics[width=45mm]{Pictures/technicallemmas/epsilonM.pdf},
\end{tabular}
\end{center}

\begin{Lemma}\label{lem:moduleepsiloneta}
The 2-morphisms $\eta^M$ and $\epsilon^M$ in $\mathfrak{C}$ depicted above are right $A$-module 2-morphisms.
\end{Lemma}
\begin{proof}
The proof that $\epsilon^M$ is right $A$-module 2-morphisms follows by successively applying equations (\ref{eqn:moduleassociativity}), (\ref{eqn:epsilonleft}), (\ref{eqn:coherenceright}), (\ref{eqn:epsilonright}), and finally (\ref{eqn:moduleunitality}). We now prove that $\eta^M$ is right $A$-module 2-morphisms. As this is more subtle, we explain the proof in details using the figures given in appendix \ref{sub:moduleepsilonetadiagrams}. Figure \ref{fig:etamodule1} depicts the left hand-side of the equation of definition \ref{def:moduleintertwiner} for $\epsilon^M$. Using naturality, we move the indicated coupons along the blue arrows, which gets us to figure \ref{fig:etamodule2}. Then, applying equation (\ref{eqn:moduleassociativity}) to the blue coupons, we arrive at figure \ref{fig:etamodule3}. Using equation (\ref{eqn:rigidleftassociativity}) on the blue coupons, leads us to contemplate figure \ref{fig:etamodule4}. In order to get to figure \ref{fig:etamodule5}, we use equation (\ref{eqn:rigidbimodule}) on the blue coupons, as well as equation (\ref{eqn:coherenceright}) for the right $A$-module $A$ on the green ones. Applying equation (\ref{eqn:coherencemiddle}) to the blue coupons produces figure \ref{fig:etamodule6}. Now, we can use equation (\ref{eqn:coherenceleft}) in the blue region, to obtain figure \ref{fig:etamodule7}. Then, the two blue coupons are manifestly each other's inverse, so we can fuse them together to obtain figure \ref{fig:etamodule8}. Using equation (\ref{eqn:etaright}) on the blue coupons gets us to figure \ref{fig:etamodule9}. Finally, we can use naturality to move the left most coupon to the right, which produces the right hand-side of the equation of definition \ref{def:moduleintertwiner} for $\epsilon^M$. The proof is therefore complete.
\end{proof}

\begin{Lemma}\label{lem:rigidactionrightadjoint}
The right $A$-module 2-morphisms $\eta^M$ and $\epsilon^M$ witness that $p^M$ is a right adjoint of $n^M$ in the 2-category $\mathbf{Mod}_{\mathfrak{C}}(A)$ of right $A$-modules.
\end{Lemma}
\begin{proof}
We have to check that $\epsilon^M$ and $\eta^M$ satisfy the triangle identities. In order to do so, we use the diagrams depicted in appendix \ref{sub:rigidactionrightadjointdiagrams}. One of the triangle identities corresponds to showing that the 2-isomorphism depicted in figure \ref{fig:triangle1} is the identity on $n^M$. Using equation (\ref{eqn:moduleassociativity}) on the blue coupons, we find that the 2-isomorphism depicted by this string diagram is equal to that given in figure \ref{fig:triangle2}. The latter is equal to the one depicted in figure \ref{fig:triangle3}, as can be seen by first using naturality to move the cap to the left, and then equation (\ref{eqn:epsilonleft}). Using the triangle identity to straighten the blue line, we get to the string diagram depicted in figure \ref{fig:triangle4}. Finally, the blue coupons cancel each other out by equation (\ref{eqn:modulemapunitality}), and similarly the green coupons cancel each other out by equation (\ref{eqn:coherenceright}). Thus we are left with a single string labelled $n^M$, which establishes one of the triangle identity.

We now show that the second triangle identity holds, that is we show that the 2-isomorphism depicted in figure \ref{fig:triangle10} is the identity on $p^M$. Using naturality to move some coupons, we find that this first string diagram represents the same 2-morphism as that given in figure \ref{fig:triangle11}. Now, equation (\ref{eqn:coherenceright}) applied to the blue coupon, and equation (\ref{eqn:moduleassociativity}) applied to the green ones show that this is also equal to the 2-morphism described by figure \ref{fig:triangle12}. Cancelling the two blue coupons out, we arrive at figure \ref{fig:triangle13}. We make use equation (\ref{eqn:etaleft}) to arrive at figure \ref{fig:triangle14}. Using equation (\ref{eqn:algebraunitality}) on the blue coupons, and equation (\ref{eqn:moduleunitality}) on the green ones gets us to figure \ref{fig:triangle15}. Equation (\ref{eqn:rigidleftassociativity}) proves that the 2-morphism described by this last figure is equal to that represented by figure \ref{fig:triangle16}. Now, applying equation (\ref{eqn:coherenceright}) to the blue coupons, and equation (\ref{eqn:rigidleftunit}) to the green ones imply that this is equal to figure \ref{fig:triangle17}. Finally, we can naturality followed by the triangle equation for $m$ to straighten the blue line, and equation (\ref{eqn:coherencemiddle}) on the green coupons, resulting in a string diagrams representing the identity 2-morphism on $p$. This proves the second triangle identity holds.
\end{proof}

Putting the above lemmas together, we get the following result, which generalizes  proposition 3.11 of \cite{BZBJ} (see also proposition D.2.2 of \cite{G} for a similar result in a different context).

\begin{Proposition}\label{prop:actionrightadjoint}
Let $A$ be a rigid algebra in $\mathfrak{C}$, and $M$ a right $A$-module with right action $n^M:M\Box A\rightarrow M$. Then, the right $A$-module 1-morphism $p^M:M\rightarrow M\Box A$ is a right adjoint to $n^M$ in $\mathbf{Mod}_{\mathfrak{C}}(A)$.
\end{Proposition}

We will use the above proposition in our proof of the following result.

\begin{Proposition}\label{prop:moduleleftdajoints}
Let $A$ be a rigid algebra in $\mathfrak{C}$, and $f:M\rightarrow N$ a 1-morphism of right $A$-modules. If $f$ has a left adjoint in $\mathfrak{C}$, then $f$ has a left adjoint in $\mathbf{Mod}_{\mathfrak{C}}(A)$.
\end{Proposition}
\begin{proof}
Let $f:M\rightarrow N$ be a right $A$-module 1-morphism, which has a left $^*f$ adjoint as a 1-morphism in $\mathbf{Mod}_{\mathfrak{C}}(A)$. We have to endow the 1-morphism $^*f$ in $\mathfrak{C}$ with a right $A$-module structure such that the unit $\eta^f:Id_N\Rightarrow f\circ {^*f}$ and counit $\epsilon^f:{^*f}\circ f\Rightarrow Id_M$ of this adjunction in $\mathfrak{C}$ are right $A$-module 2-morphisms. We begin by making the following definition:

\settoheight{\algebra}{\includegraphics[width=37.5mm]{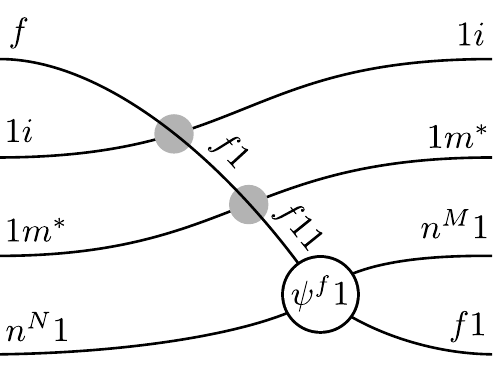}}

\begin{center}
\begin{tabular}{@{}cc@{}}

\raisebox{0.45\algebra}{$\zeta^f:=$} &
\includegraphics[width=37.5mm]{Pictures/technicallemmas/zetaf.pdf}.
\end{tabular}
\end{center}

Then, by proposition \ref{prop:actionrightadjoint}, we have explicit left adjoints for $p^M$ and $p^N$. Using these, we let $$\psi^{^*f}:= {^*(\zeta^f)}.$$ As $\zeta^f$ is clearly invertible, so is $\psi^{^*f}$. It is possible to check the coherence axioms of definition \ref{def:modulemap} for $\psi^{^*f}$ directly, but this is tedious. We give an alternative approach. Let us set

\settoheight{\algebra}{\includegraphics[width=37.5mm]{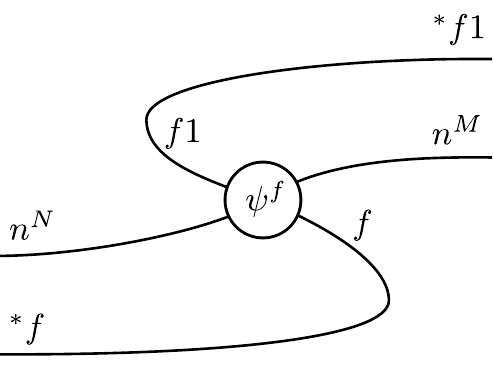}}

\begin{center}
\begin{tabular}{@{}cc@{}}

\raisebox{0.45\algebra}{$\xi^f:=$} &
\includegraphics[width=37.5mm]{Pictures/technicallemmas/xif.pdf}.
\end{tabular}
\end{center}

\noindent We claim that $\xi^f$ is the inverse of $\psi^{^*f}$. Assuming this is true, it is easy to see that $\xi^f$ satisfy the inverse of the equations giving the axioms of definition \ref{def:modulemap}, and so proves that $^*f$ is a right $A$-module map. Moreover, it is not hard to check that $\eta^f$ and $\epsilon^f$ are right $A$-module 2-morphisms.

We prove the above claim using the diagrams depicted in appendix \ref{sub:moduleleftadjointsdiagrams}. The composite $\xi^f \cdot \psi^{^*f}$ is depicted in figure \ref{fig:adjinverse1}. By moving the coupon labelled $\psi^f$ down and to the left, whilst using the triangle identity for $f$, we arrive at figure \ref{fig:adjinverse2}. Equation (\ref{eqn:modulemapassociativity}) applied to the blue coupons yields figure \ref{fig:adjinverse3}. Moving the coupon labelled $1\psi^{l^{-1}}$ to the right along the blue arrow, and using equation (\ref{eqn:moduleassociativity}) on the green coupons brings us to figure \ref{fig:adjinverse4}. Now, we can move the cap to the left along the blue arrow, and then appeal to equation (\ref{eqn:epsilonleft}) on the green coupons. Then, taking the coupon labelled $1\rho^{-1}$ to the right along the red arrows leads us to contemplate figure \ref{fig:adjinverse5}. At this point, we can move the coupon labelled $1\lambda^{-1}$ to the right along the blue arrow, straighten the green curve, and use equation (\ref{eqn:coherenceright}) on the red coupons to arrive at figure \ref{fig:adjinverse6}. Appealing to equation (\ref{eqn:modulemapunitality}) on the blue coupons produces figure \ref{fig:adjinverse7}. Finally, straightening the blue curve, and using equation (\ref{eqn:moduleunitality}) on the green coupons shows that the last figure depicts the identity 2-morphism of $n^M\circ {^*f1}$, which proves that $\xi^f \cdot \psi^{^*f}=Id_{n^M\circ {^*f1}}$.

We now consider the composite $\psi^{^*f} \cdot \xi^f$ depicted in figure \ref{fig:adjinverse10}. Dragging the coupon labelled $\psi^f$ to the right as well as straightening the green curve brings us to figure \ref{fig:adjinverse11}. Appealing to equation (\ref{eqn:modulemapassociativity}) on the blue coupons produces figure \ref{fig:adjinverse12} for us. Then, we move the indicated cap along the blue arrow, which gives us figure \ref{fig:adjinverse13} after rearranging the strings a little. Using equation (\ref{eqn:moduleassociativity}) for $N$ on the blue coupons gives figure \ref{fig:adjinverse14}. We can now apply equation (\ref{eqn:epsilonleft}) to the blue coupons, and arrive at figure \ref{fig:adjinverse15}. Straightening the curve in blue, and using equation (\ref{eqn:modulemapunitality}) on the green coupons gives figure \ref{fig:adjinverse16}. Finally, applying equation (\ref{eqn:moduleunitality}) to the blue coupons, equation (\ref{eqn:coherenceright}) applied to the green ones, and straightening the red curve shows that this last figure depicts the identity 2-morphism on ${^*f}\circ n^N$. Thus, we have prove that $\psi^{^*f} \cdot \xi^f = Id_{{^*f}\circ n^N}$, which finishes the proof of the claim, whence of the proposition.
\end{proof}

\begin{Remark}
If we take $\mathfrak{C}=\mathbf{FinCat}$, $\mathcal{A}$ a finite tensor category, and $F:\mathcal{M}\rightarrow \mathcal{N}$ a right $\mathcal{A}$-module functor between finite categories, then proposition \ref{prop:moduleleftdajoints} recovers the statement that if $F$ has a left adjoint as a linear functor (i.e. is right exact), it has a left adjoints as an $\mathcal{A}$-module functor (see section 3.3 of \cite{EO}). This also generalizes theorem 5.21 of \cite{BJS}.
\end{Remark}

Building on proposition \ref{prop:moduleleftdajoints} and its proof, we obtain the following result.

\begin{Proposition}\label{prop:bimoduleleftdajoints}
Let $A$, $B$ be rigid algebras in $\mathfrak{C}$, and $f:P\rightarrow Q$ an $A$-$B$-bimodule 1-morphism. If $f$ has a left adjoint in $\mathfrak{C}$, then it has a left adjoint as an $A$-$B$-bimodule 1-morphism.
\end{Proposition}
\begin{proof}
The left $A$-module 1-morphism $l^P:A\Box P\rightarrow P$ has a right adjoint as a left $A$-module 1-morphism. A similar observation holds with $P$ replaced by $Q$, and these can be used to endow $^*f$ with a canonical left $A$-module structure. The details can either be checked directly, or one may appeal to proposition \ref{prop:actionrightadjoint} in the monoidal 2-category $\mathfrak{C}^{\Box op}$, that is $\mathfrak{C}$ equipped with its opposite monoidal structure. It therefore only remains to show that the canonical left $A$-module and right $B$-module structures on $^*f$ are compatible.

In order to see this, we note that, similarly to what is established in the proof of proposition \ref{prop:moduleleftdajoints}, the 2-morphism $\chi^{^*f}$ providing $^*f$ with its left $A$-module structure is the inverse of the 2-morphism $\theta^f$ given by

\settoheight{\algebra}{\includegraphics[width=37.5mm]{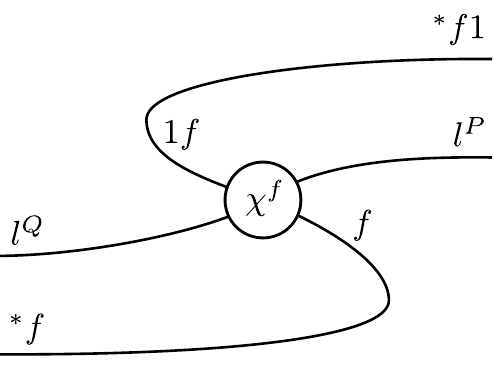}}

\begin{center}
\begin{tabular}{@{}cc@{}}

\raisebox{0.45\algebra}{$\theta^f:=$} &
\includegraphics[width=37.5mm]{Pictures/technicallemmas/thetaf.pdf}.
\end{tabular}
\end{center}

\noindent Now, using the unit $\eta^f$ and counit $\epsilon^f$ witnessing the adjunction between $^*f$ and $f$, we see that equation (\ref{def:bimodulemap}) for $f$ yields an equation involving $\xi^f$ and $\theta^f$. But, these 2-morphisms are the inverses of $\psi^{^*f}$ and of $\chi^{^*f}$ respectively, so that the aforementioned equation is equivalent to equation (\ref{def:bimodulemap}) for $^*f$.
\end{proof}

The following theorem is the main application of the above results.

\begin{Theorem}\label{thm:modulerightdajoints}
Assume that the monoidal 2-category $\mathfrak{C}$ has left and right adjoints, and let $A$ be an arbitrary algebra in $\mathfrak{C}$. Then, $A$ is rigid if and only if $\mathbf{Bimod}_{\mathfrak{C}}(A)$ has right adjoints. If either of these conditions is satisfied, $\mathbf{Bimod}_{\mathfrak{C}}(A)$ has left adjoints, and $\mathbf{Mod}_{\mathfrak{C}}(A)$ has left and right adjoints.
\end{Theorem}
\begin{proof}
Let us begin by the following observation: We can consider the strict opcubical monoidal 2-category $\mathfrak{C}^{2op}$ obtained from $\mathfrak{C}$ by reversing the direction of the 2-morphisms. If $A$ is an algebra in $\mathfrak{C}$, then $(A, m, i, \alpha^{-1}, \rho^{-1}, \lambda^{-1})$ is an algebra in $\mathfrak{C}^{2op}$, which we denote by $A^{2op}$. Furthermore, inspection shows that $$\big(\mathbf{Mod}_{\mathfrak{C}}(A)\big)^{2op}= \mathbf{Mod}_{\mathfrak{C}^{2op}}(A^{2op})\text{ and }\big(\mathbf{Bimod}_{\mathfrak{C}}(A)\big)^{2op}= \mathbf{Bimod}_{\mathfrak{C}^{2op}}(A^{2op}).$$

Let us now assume that $A$ is rigid, it follows from proposition \ref{prop:bimoduleleftdajoints} that $\mathbf{Bimod}_{\mathfrak{C}}(A)$ has left adjoints. Via the second of the above equivalences, this implies that $\mathbf{Bimod}_{\mathfrak{C}^{2op}}(A^{2op})$ has right adjoints, so that $A^{2op}$ is a rigid algebra. Propositions \ref{prop:moduleleftdajoints} and \ref{prop:bimoduleleftdajoints} then show that $\mathbf{Mod}_{\mathfrak{C}^{2op}}(A^{2op})$ and $\mathbf{Bimod}_{\mathfrak{C}^{2op}}(A^{2op})$ have left adjoints. The first of the two equivalences above proves that $\mathbf{Mod}_{\mathfrak{C}}(A)$ has right adjoints, while the second demonstrates that $\mathbf{Bimod}_{\mathfrak{C}}(A)$ has right adjoints.

Assume that $\mathbf{Bimod}_{\mathfrak{C}}(A)$ has right adjoints. Then as $m:A\Box A\rightarrow A$ is an $A$-$A$-bimodule map, it has a right adjoint. This shows that $A$ is rigid, and completes the present proof.
\end{proof}

\section[Separable Algebras in Compact Semisimple Monoidal \texorpdfstring{\newline}{} 2-Categories]{Separable Algebras in Compact Semisimple Monoidal 2-Categories}\label{sec:separable}

Throughtout, we fix a compact semisimple monoidal 2-category $\mathfrak{C}$ over an arbitrary field $\mathds{k}$. Without loss of generality, we assume that $\mathfrak{C}$ is strict cubical.

\subsection{Characterizing Separable Algebras}

We begin by the following lemma, which elaborates on the results of the previous section.

\begin{Lemma}\label{lem:separablecondensations}
Let $A$ be a separable algebra, there exists a right $A$-module 2-morphism $\gamma^M$ such that $(M\Box A, M, n^M, p^M, \epsilon^M, \gamma^M)$ is a 2-condensation in $\mathbf{Mod}_{\mathfrak{C}}(A)$.
\end{Lemma}
\begin{proof}
We let $\gamma^M$ be the 2-morphism in $\mathfrak{C}$ given by the following string diagram:

\settoheight{\algebra}{\includegraphics[width=45mm]{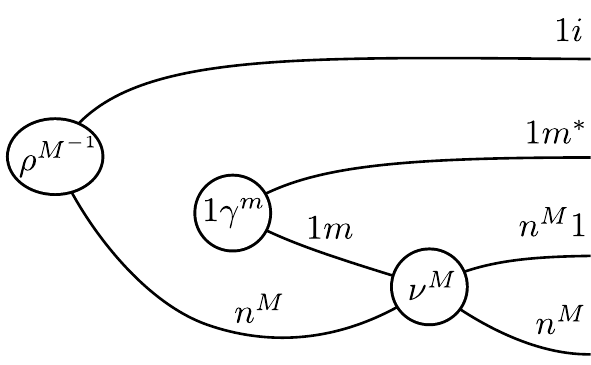}}

\begin{center}
\begin{tabular}{@{}cc@{}}

\raisebox{0.45\algebra}{$\gamma^M:=$} &
\includegraphics[width=45mm]{Pictures/technicallemmas/gammaM.pdf}.
\end{tabular}
\end{center}

\noindent It follows immediately from the definitions that $\epsilon^M\cdot\gamma^M = Id_{n^M\circ p^M}$. Thus, in order to show that $(M\Box A, M, n^M, p^M, \epsilon^M, \gamma^M)$ is a 2-condensation in $\mathbf{Mod}_{\mathfrak{C}}(A)$, it only remains to prove that $\gamma^M$ is a right $A$-module 2-morphism. Starting from the left hand-side of the equation of definition \ref{def:moduleintertwiner} for $\gamma^M$, we can successively apply equations (\ref{eqn:moduleassociativity}), (\ref{eqn:gammaright}), (\ref{eqn:moduleassociativity}), (\ref{eqn:moduleunitality}), (\ref{eqn:gammaleft}), and (\ref{eqn:coherenceright}) to get to the right hand-side. This finishes the proof of the lemma.
\end{proof}

\begin{Proposition}\label{prop:sepmodfss}
Let $A$ be a separable algebra in $\mathfrak{C}$. Then, $\mathbf{Mod}_{\mathfrak{C}}(A)$ is a compact semisimple 2-category.
\end{Proposition}
\begin{proof}
It follows from theorem \ref{thm:modulerightdajoints} that $\mathbf{Mod}_{\mathfrak{C}}(A)$ has right and left adjoints. Further, $\mathbf{Mod}_{\mathfrak{C}}(A)$ is Cauchy complete as was proven in proposition 3.3.8 of \cite{D4}.

We now show that $\mathbf{Mod}_{\mathfrak{C}}(A)$ is locally finite semisimple. Given any right $A$-module $N$, and object $C$ of $\mathfrak{C}$, lemma 3.2.13 of \cite{D4} provides us with the following equivalence of linear 1-categories $$Hom_A(C\Box A, N)\simeq Hom_{\mathfrak{C}}(C,N).$$ By definition, the right hand-side is a finite semisimple 1-category. Now, we use the fact proven in lemma \ref{lem:separablecondensations} that every right $A$-module $M$ is the splitting of a 2-condensation monad on the free right $A$-module $M\Box A$. Together with the above equivalence and the fact that 2-condensations are preserved by every 2-functor, this implies that $Hom_A(M, N)$ is the splitting of a 2-condensation monad on $Hom_{\mathfrak{C}}(M,N)$ in $\mathbf{FinCat_{ss}}$. But we have shown in lemma \ref{lem:sscatCauchy} that $\mathbf{FinCat_{ss}}$ is Cauchy complete, so that $Hom_A(M, N)$ is indeed a finite semisimple 1-category.

We have just proven that $\mathbf{Mod}_{\mathfrak{C}}(A)$ is a semisimple 2-category, so it only remains to prove that $\pi_0(\mathbf{Mod}_{\mathfrak{C}}(A))$ is finite. Let us write $C_i$ for a finite set of simple objects of $\mathfrak{C}$ representing the equivalence classes in $\pi_0(\mathfrak{C})$. We claim that, for every simple object $M$ in $\mathbf{Mod}_{\mathfrak{C}}(A)$, there exists a non-zero 1-morphism $C_i\Box A\rightarrow M$ for some $i$. Assuming that this holds, it follows from lemma 1.1.6 of \cite{D5} that every equivalence classes in $\pi_0(\mathbf{Mod}_{\mathfrak{C}}(A))$ is represented by a simple right $A$-module summand of $C_i\Box A$ for some $i$. But, as $\mathbf{Mod}_{\mathfrak{C}}(A)$ is locally finite semisimple and $\pi_0(\mathfrak{C})$ is finite, there are only finitely many such summands.

Let us prove the claim. For any simple summand $D$ of $M$ in $\mathfrak{C}$, the composite $D\Box A\hookrightarrow M\Box A\xrightarrow{n^M} M$ is a non-zero right $A$-module 1-morphism, as precomposition with $D\Box i$ is 2-isomorphic to the inclusion $D\hookrightarrow M$. By construction, there exists a non-zero 1-morphism $C_i\rightarrow D$ for some $i$, and the composite $C_i\rightarrow M$ is non-zero. Thus, the composite $C_i\Box A\rightarrow M\Box A\xrightarrow{n^M} M$ is a non-zero right $A$-module 1-morphism. This concludes the proof of the claim.
\end{proof}

\begin{Proposition}\label{prop:sepbimodfss}
Let $A$ and $B$ be separable algebras in the compact semisimple monoidal 2-category $\mathfrak{C}$. Then, $\mathbf{Bimod}_{\mathfrak{C}}(A,B)$ is a compact semisimple 2-category.
\end{Proposition}
\begin{proof}
The proof of proposition \ref{prop:sepmodfss} can be adapted to the bimodule case. We leave the details to the reader.
\end{proof}

\begin{Definition}
Let $B$ be an arbitrary algebra in an arbitrary monoidal 2-category $\mathfrak{D}$. The partial center of $B$, denoted by $Z(B)$, is the monoidal 1-category of $B$-$B$-bimodule endomorphisms of $B$.
\end{Definition}

\begin{Remark}
We call $Z(B)$ the partial center of $B$ to distinguish it from the full center, which is a braided algebra in the Drinfeld center of $\mathfrak{D}$. We will explore the properties of the full centre in an upcoming article, in the meantime, we refer the reader to \cite{Da} for the decategorified version of this story. Now, if $B$ is a rigid algebra in a compact semisimple 2-category, then, by theorem \ref{thm:modulerightdajoints}, $Z(B)$ is a rigid monoidal 1-category. If, in addition, $\mathfrak{D}$ is compact semisimple, we expect that $Z(B)$ is a finite 1-category.
\end{Remark}

\begin{Theorem}\label{thm:characterizationseparablealgebra}
Let $A$ be a rigid algebra in a compact semisimple monoidal 2-category $\mathfrak{C}$. Then, $A$ is separable if and only if $Z(A)$ is finite semisimple. If either of these conditions is satisfied, then $\mathbf{Bimod}_{\mathfrak{C}}(A)$ is a compact semisimple 2-category.
\end{Theorem}
\begin{proof}
If $A$ is separable, it follows immediately from proposition \ref{prop:sepbimodfss} that $\mathbf{Bimod}_{\mathfrak{C}}(A)$ is compact semisimple, so that, a fortiori, $Z(A)$ is finite semisimple. Now, for the backward direction, let us assume that $Z(A)$ is a finite semisimple 1-category. We first establish the result under the additional assumption that $A$ is indecomposable, i.e. that $Id_A:A\rightarrow A$ is a simple object of $Z(A)$. In that case, because $\mathbf{Bimod}_{\mathfrak{C}}(A)$ has right adjoint by theorem \ref{thm:modulerightdajoints}, the $A$-$A$-bimodule 1-morphism $m:A\Box A\rightarrow A$ has a right adjoint (as an $A$-$A$-bimodule 1-morphism), which we denote by $m^*$. Let us write $\epsilon^m$ for the counit of this adjunction. We claim that $\epsilon^m$ has a section in $\mathbf{Bimod}_{\mathfrak{C}}(A)$. Namely, we have assumed that $Id_A$ is a simple object of $Z(A)$, so that either $\epsilon^m:m\circ m^*\Rightarrow Id_A$ has a section or it is zero. But, as $m$ is not the zero map, the latter is not possible. This proves that $A$ is a separable algebra in this case.

Moving on to the general case, let us write $$Id_A = f_1\oplus ... \oplus f_n$$ for a decomposition of $Id_A$ into a direct sum of simple objects of $Z(A)$. As $\mathbf{Bimod}_{\mathfrak{C}}(A)$ is Cauchy complete (this follows from proposition 3.3.8 of \cite{D4}), we can use proposition 1.3.16 of \cite{DR} to obtain a direct sum decomposition $$A = A_1\boxplus ...\boxplus A_n$$ of $A$ in $\mathbf{Bimod}_{\mathfrak{C}}(A)$. In particular, each object $A_i$ defines an algebra in $\mathfrak{C}$ using the restriction of the algebra structure on $A$. Further, it follows from the above direct sum decomposition that the restriction of $m:A\Box A\rightarrow A$ to $A_i\Box A_j$ can only be non-zero if $i=j$. Thence, $A$ is the direct sum of the algebras $A_i$ in $\mathfrak{C}$. This implies that $Z(A) \cong Z(A_1)\oplus ...\oplus Z(A_n)$, so that $Id_{A_i}=f_i$ is simple as an $A$-$A$-bimodule if and only if it is simple as an $A_i$-$A_i$-bimodule. In particular, each $A_i$ is indecomposable. Finally, as $A$ was assumed to be rigid, so is each $A_i$. The special case proven above allows us to conclude the proof.
\end{proof}

The following result is well-known when $\mathds{k}$ is a perfect field (see corollary 2.5.9 of \cite{DSPS13}).

\begin{Corollary}
Over any field $\mathds{k}$, a perfect tensor 1-category $\mathcal{C}$ is separable as a rigid algebra in $\mathbf{2Vect}$ if and only if its Drinfel'd center $\mathcal{Z}(\mathcal{C})$ is finite semisimple.
\end{Corollary}
\begin{proof}
Recall that $\mathcal{Z}(\mathcal{C})$ is identified with the 1-category of $\mathcal{C}$-$\mathcal{C}$-bimodule endofunctors of $\mathcal{C}$. Thus, the result follows from theorem \ref{thm:characterizationseparablealgebra}.
\end{proof}

\subsection{The Dimension of a Rigid Algebra}

\begin{Definition}
An algebra $A$ in $\mathfrak{C}$ is called connected if the unit $i:I\rightarrow A$ is a simple 1-morphism.
\end{Definition}

To help with our intuition, let us examine an example.

\begin{Example}
A perfect monoidal 1-category $\mathcal{A}$ with unit $I$ viewed as an algebra in $\mathbf{2Vect}$ is connected if and only if $End_{\mathcal{A}}(I)$ is a simple $\mathds{k}$-algebra. Note that this implies that $End_{\mathcal{A}}(I)$ is a field, as it is necessarily a commutative algebra.
\end{Example}

We now fix a connected rigid algebra $A$ in $\mathfrak{C}$. In particular, $End_{\mathfrak{C}}(i)$, the finite dimensional $\mathds{k}$-algebra of 2-endomorphisms of $i$ in $\mathfrak{C}$, is a division $\mathds{k}$-algebra. In fact, as $A$ is an algebra, $End_{\mathfrak{C}}(i)$ has two compatible composition operations, so that the product is necessarily commutative by the Eckman-Hilton argument. Thence, $End_{\mathfrak{C}}(i)$ is a finite field extension of $\mathds{k}$. Further, under the equivalence $$\begin{tabular}{ccc}$Hom_{A\mathrm{-}A}(A\Box A,A)$&$\xrightarrow{\simeq}$& $Hom_{\mathfrak{C}}(I,A)$\\ $f$&$\mapsto$&$f\circ (i\Box i)$\end{tabular}$$ the unit $i:I\rightarrow A$ corresponds to $m:A\Box A\rightarrow A$ viewed as an $A$-$A$-bimodule 1-morphism. This implies that $\mathbb{K}:=End_{A\mathrm{-}A}(m)$, the endomorphism algebra of $m$ in $\mathbf{Bimod}_{\mathfrak{C}}(A)$, is a finite field extension of $\mathds{k}$.

As $\mathfrak{C}$ is a compact semisimple 2-category, it has right adjoints. In particular, $m^*:A\rightarrow A\Box A$ has a right adjoint in $\mathfrak{C}$, which we denote by $m^{**}:A\Box A\rightarrow A$. Thanks to theorem \ref{thm:modulerightdajoints}, $m^{**}$ is right adjoint to $m^*$ as an $A$-$A$-bimodule 1-morphism. But, in any finite semisimple tensor 1-category, the right and left dual of any object are non-canonically isomorphic. This means that we have $m^{**}\cong m$ in $\mathfrak{C}$. This can be substantially refined. Namely, the equivalence above factors as $$Hom_{A\mathrm{-}A}(A\Box A,A)\rightarrow Hom_{\mathfrak{C}}(A\Box A,A)\rightarrow Hom_{\mathfrak{C}}(I,A),$$ and $m$, $m^{**}$ are both objects of the left hand-side, which become isomorphic in the middle step. As the composite is an equivalence of 1-categories, there must exist a 2-isomorphism $\mu:m^{**}\cong m$ of $A$-$A$-bimodule 1-morphisms. Note that $\mu$ is not unique, it is determined up to a non-zero scalar in $\mathbb{K}$.

Let us now fix an $A$-$A$-bimodule 2-isomorphism $\mu:m^{**}\cong m$. We define a 2-morphism $\mathrm{Tr}(\mu)$, the trace of $\mu$, by

\settoheight{\algebra}{\includegraphics[width=45mm]{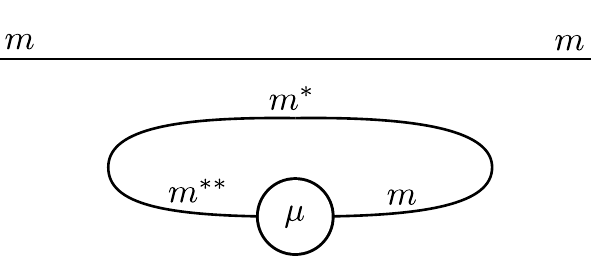}}

\begin{center}
\begin{tabular}{@{}cc@{}}

\raisebox{0.45\algebra}{$\mathrm{Tr}(\mu):=$} &
\includegraphics[width=45mm]{Pictures/dimension/tracemu.pdf}.
\end{tabular}
\end{center}

\noindent As all of the 2-morphisms used in the definition of $\mathrm{Tr}(\mu)$ are $A$-$A$-bimodule 2-morphisms, so is $\mathrm{Tr}(\mu)$. Thus, we think of $\mathrm{Tr}(\mu)$ as a scalar in $\mathbb{K}$. Analogously, we define the $A$-$A$-bimodule 2-morphism $\mathrm{Tr}((\mu^{-1})^*)$ by 

\settoheight{\algebra}{\includegraphics[width=45mm]{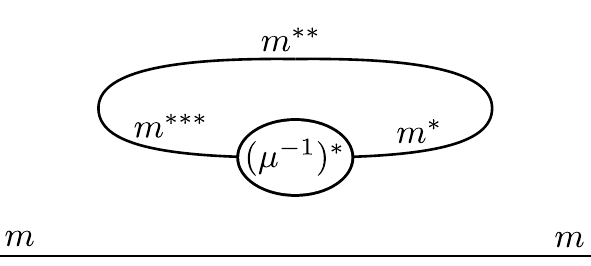}}

\begin{center}
\begin{tabular}{@{}cc@{}}

\raisebox{0.45\algebra}{$\mathrm{Tr}((\mu^{-1})^*):=$} &
\includegraphics[width=45mm]{Pictures/dimension/tracemustar.pdf}.
\end{tabular}
\end{center}

\noindent The scalars $\mathrm{Tr}(\mu)$ and $\mathrm{Tr}((\mu^{-1})^*)$ depend not only on the 2-isomorphism $\mu:m\cong m^{**}$, but also on the adjunction data for $m$, $m^*$, and $m^{**}$. However, if $\mathbb{K}\cong \mathds{k}$, which always holds if $\mathds{k}$ is algebraically closed, then by multiplying them together we obtain a scalar which is independent of any choice. This is precisely the same trick that is used to define the squared norm of a simple object in a multifusion 1-category (see definition 7.12.2 of \cite{EGNO}).

\begin{Definition}\label{def:dimension}
Provided $\mathbb{K}\cong \mathds{k}$, the dimension of a connected rigid algebra $A$ in $\mathfrak{C}$ is the scalar in $\mathds{k}$ defined by $$\mathrm{Dim}_{\mathfrak{C}}(A):=\mathrm{Tr}(\mu)\cdot \mathrm{Tr}((\mu^{-1})^*).$$
\end{Definition}

The importance of the dimension is witnessed by the following result, which is a generalization of theorem 2.6.7 of \cite{DSPS13}.

\begin{Theorem}\label{thm:dimensionseparable}
A connected rigid algebra $A$ in a compact semisimple monoidal 2-category $\mathfrak{C}$ is separable if and only if there exists a 2-isomorphism $\mu:m\cong m^{**}$ of $A$-$A$-bimodules such that $\mathrm{Tr}(\mu)$ is non-zero in $\mathbb{K}$. If $\mathbb{K}\cong \mathds{k}$, this holds if and only if $\mathrm{Dim}_{\mathfrak{C}}(A)$ is non-zero.
\end{Theorem}
\begin{proof}
Let us define a 2-morphism of $A$-$A$-bimodules $\tau$ by 

\settoheight{\algebra}{\includegraphics[width=30mm]{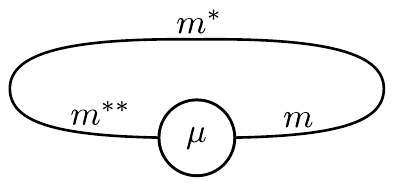}}

\begin{center}
\begin{tabular}{@{}cc@{}}

\raisebox{0.45\algebra}{$\tau:=$} &
\includegraphics[width=30mm]{Pictures/dimension/tau.pdf}.
\end{tabular}
\end{center}

\noindent As $\mathrm{Tr}(\mu)$ is a scalar in $\mathbb{K}$, if it is non-zero, then it is in fact invertible. Now, $\mathrm{Tr}(\mu)$ is obtained by precomposing $\tau$ with $m$. But, at the level of 2-morphisms in $\mathfrak{C}$, precomposition along $m$ induces a split inclusion $End_{\mathfrak{C}}(Id_A)\hookrightarrow End_{\mathfrak{C}}(m)$ of $\mathds{k}$-algebras. This implies that $\tau$ is invertible as a 2-morphism in $\mathfrak{C}$. A direct check then shows that $\tau$ is invertible as a 2-morphism of $A$-$A$-bimodules. Thus, we can set $$\gamma^m:=(\mu\circ m^*)\cdot \eta^{m^*}\cdot \tau^{-1},$$ which is an $A$-$A$-bimodule 2-morphism, and satisfies $\epsilon^m\cdot\gamma^m = Id_A$. Thus, $\gamma^m$ upgrades the rigid algebra $A$ to a separable algebra. Moreover, if $\mathds{k}$ is algebraically closed, and $\mathrm{Dim}_{\mathfrak{C}}(A)$ is non-zero, then $\mathrm{Tr}(\mu)$ is necessarily non-zero.

Conversely, if $A$ is separable, then it follows from theorem \ref{thm:characterizationseparablealgebra} that the 2-category $\mathbf{Bimod}_{\mathfrak{C}}(A)$ is compact semisimple. Moreover, as $A$ is connected, we have seen above that $End_{A\mathrm{-}A}(m)\cong End_{\mathfrak{C}}(i)$ is a field, which establishes that $m$ is a simple 1-morphism in $\mathbf{Bimod}_{\mathfrak{C}}(A)$. Now, left and right adjoints are non-canonically isomorphic in a compact semisimple 2-category. Given that $m$ is a simple $A$-$A$-bimodule, and $\gamma^m$ is non-zero, we find that $\gamma^m:Id_A\Rightarrow m\circ m^*$ is a unit witnessing that $m^*$ is left adjoint to $m$. Let us write $\delta^m:m^*\circ m\Rightarrow Id_{AA}$ for a compatible counit. Taking $\mu = Id_m$, we find that $\mathrm{Tr}(\mu) = Id_m$. This concludes the first part of the proof. 

Let us now assume that $\mathds{k}$ is algebraically closed. It is enough to show that $\mathrm{Tr}((\mu^{-1})^*)= m\circ (\delta^m\cdot\eta^m)$ is non-zero. In order to do so, we will argue that the $\mathds{k}$-linear pairing $$\begin{tabular}{c c c}
$\mathrm{P}:Hom_{A\mathrm{-}A}(m^*\circ m, Id_{AA})\times Hom_{A\mathrm{-}A}(Id_{AA}, m^*\circ m)$ & $\rightarrow$ & $End_{A\mathrm{-}A}(m)$\\
$\ \ \ \ \ (\alpha,\beta)$ & $\mapsto$ & $m\circ (\alpha \cdot \beta)$
\end{tabular}$$
is non-degenerate. Namely, recall that $\mathbb{K}=End_{A\mathrm{-}A}(m)$, but we have assumed $\mathbb{K}\cong\mathds{k}$. This implies that $Hom_{A\mathrm{-}A}(m^*\circ m, Id_{AA})\cong \mathds{k}$ and that $Hom_{A\mathrm{-}A}(Id_{AA}, m^*\circ m)\cong \mathds{k}$. Further, there is a right action
$$\begin{tabular}{c c c}
$Hom_{A\mathrm{-}A}(m^*\circ m, Id_{AA})\times End_{A\mathrm{-}A}(m)$ & $\rightarrow$ & $Hom_{A\mathrm{-}A}(m^*\circ m, Id_{AA}).$\\
$(\alpha,\gamma)$ & $\mapsto$ & $\alpha \cdot (m^*\circ \gamma)$
\end{tabular}$$ Likewise, there is a left action of $End_{A\mathrm{-}A}(m)$ on $Hom_{A\mathrm{-}A}(Id_{AA}, m^*\circ m)$, and the pairing $\mathrm{P}$ is balanced under these actions of $End_{A\mathrm{-}A}(m)$. It is therefore enough to show that $\mathrm{P}$ is non-zero. We claim that the map $$\begin{tabular}{c c c}
$Hom_{A\mathrm{-}A}(m^*\circ m, Id_{AA})$ & $\rightarrow$ & $Hom_{A\mathrm{-}A}(m\circ m^*\circ m, m).$\\
$\alpha$ & $\mapsto$ & $m\circ \alpha$
\end{tabular}$$ is non-zero. Namely, $\delta^m$ is a counit witnessing an adjunction between $m$ and $m^*$, so that $m\circ \delta^m$ is non-zero. In particular, there exists a simple $A$-$A$-bimodule 1-morphism $f:A\Box A\rightarrow A\Box A$, which is a summand of both $Id_{AA}$ and $m^*\circ m$, and such that $m\circ f$ is non-zero. This demonstrates that $\mathrm{P}$ is non-zero.
\end{proof}

\begin{Example}
Over a general field, the dimension of a connected rigid algebra might not be well-defined as we now explain. Let us take $\mathds{k}=\mathbb{R}$, and $\mathfrak{C}=\mathbf{2Vect}$. We consider the fusion 1-category $\mathcal{C}$ of $\mathbb{C}$-$\mathbb{C}$-bimodules in $\mathbf{Vect}$. We have that $End_{\mathcal{C}\mathrm{-}\mathcal{C}}(\mathcal{C})\simeq \mathbf{Vect}$. In particular, $\mathcal{C}$ is a separable algebra in $\mathbf{2Vect}$ by theorem \ref{thm:characterizationseparablealgebra}. On the other hand, for any choice of $\mathcal{C}$-$\mathcal{C}$-bimodule natural isomorphism $\mu$, the scalar $\tau$ defined in the proof of theorem \ref{thm:dimensionseparable} lies in $\mathbb{R}$. But $\mu$ is defined up to a non-zero element of $\mathbb{C}$, so that there must exist a non-zero $\mu$ such that $\mathrm{Tr}(\mu) = 0$. Nevertheless, there exists choices of $\mu$ for which $\mathrm{Tr}(\mu)\cdot \mathrm{Tr}((\mu^{-1})^*)$ is non-zero.
\end{Example}

\begin{Remark}
It has been conjectured in \cite{JFR} that every rigid algebra in a fusion 2-category over an algebraically closed field of characteristic zero is separable. Thanks to our result above, in order to establish their conjecture, it would suffice to prove that the dimension of every rigid algebra is non-zero. We discuss some examples below.
\end{Remark}

\begin{Remark}
The following question was brought to our attention by Theo Johnson-Freyd: Is the dimension of a connected rigid algebra a Morita invariant. More precisely, let $A$ and $B$ be two connected rigid algebras in a finite semisimple monoidal 2-category $\mathfrak{C}$ over an algebraically closed field. The associated 2-categories of right modules $\mathbf{Mod}_{\mathfrak{C}}(A)$, and $\mathbf{Mod}_{\mathfrak{C}}(B)$ are left $\mathfrak{C}$-module 2-categories by proposition 3.3.1 of \cite{D4}. Then, the precise question is: If $\mathbf{Mod}_{\mathfrak{C}}(A)\simeq\mathbf{Mod}_{\mathfrak{C}}(B)$ as left $\mathfrak{C}$-module 2-categories, do we have $\mathrm{Dim}_{\mathfrak{C}}(A)=\mathrm{Dim}_{\mathfrak{C}}(B)$? In the case $\mathfrak{C}=\mathbf{2Vect}$ (and $\mathds{k}$ is algebraically closed), this is a standard result in the theory of fusion 1-categories (see \cite{ENO}).
\end{Remark}

\begin{Remark}
For the sake of completeness, let us also comment on the case of non-connected rigid algebras. Let $A$ be an algebra with unit $i$ in a compact semisimple monoidal 2-category $\mathfrak{C}$. In particular, we may decompose $i = i_1\oplus ... \oplus i_n$ into a direct sum of simple 1-morphisms in $\mathfrak{C}$, and this induces a decomposition of the algebra $A$ as $$\begin{pmatrix}
_{1}A_1 & \cdots & _1A_n\\
\vdots & \ddots & \vdots\\
_nA_1&\cdots & _nA_n\\
\end{pmatrix},$$ where ${_jA_j}$ is a connected algebra with unit $i_j$, and ${_jA_k}$ is an ${_jA_j}$-${_kA_k}$-bimodule for every $j,k$. Now, if we assume that $A$ is rigid, then it is easy to check that ${_jA_j}$ is a rigid algebra for every $j$. In particular, if $\mathds{k}$ is algebraically closed and $\mathfrak{C}=\mathbf{2Vect}$, then this is precisely the well-known decomposition of a mutlifusion 1-category into a matrix of finite semisimple bimodule categories with fusion categories along the diagonal. Over an arbitrary compact semisimple 2-category, it is then possible to show that $A$ is separable if and only if ${_jA_j}$ is separable for every $j$. As the proof of this fact is non-trivial, we postpone it to a later article.
\end{Remark}

\subsection{Examples}\label{sub:separableexamples}

When working with fusion 1-categories over an algebraically closed field, there is a well-known notion of global (or categorical) dimension. We begin relating the notion of dimension we have just introduced with this classical notion.

\begin{Proposition}\label{prop:globaldimension}
Assume that $\mathds{k}$ is algebraically closed, and let $\mathcal{C}$ be a fusion 1-category with simple unit $I$, i.e.\ a connected rigid algebra in $\mathbf{2Vect}$. Then, the dimension of $\mathcal{C}$ in the sense of definition \ref{def:dimension} coincides with $\mathrm{dim}(\mathcal{C})$, its global dimension in the sense of definition 2.6.3 of \cite{DSPS13}.
\end{Proposition}
\begin{proof}
Instead of working with $\mathcal{C}$-$\mathcal{C}$-bimodules, we will equivalently work with left $\mathcal{C}\boxtimes\mathcal{C}^{\otimes op}$-modules. As $\mathcal{C}$ is a fusion 1-category, so is $\mathcal{C}\boxtimes\mathcal{C}^{\otimes op}$. We write $I$ for the simple unit of $\mathcal{C}\boxtimes\mathcal{C}^{\otimes op}$ and $\otimes$ for its monoidal product. By subsection 2.6 of \cite{DSPS13}, there exists a Frobenius algebra $R$ in $\mathcal{C}\boxtimes\mathcal{C}^{\otimes op}$ with multiplication $m:R\otimes R\rightarrow R$, unit $u:I\rightarrow R$, comultiplication $\Delta:R\rightarrow R\otimes R$, and counit $\lambda:R\rightarrow I$ such that $\mathcal{C}$ is equivalent to $Mod_{\mathcal{C}\boxtimes\mathcal{C}^{\otimes op}}(R)$, the 1-category of right $R$-modules in $\mathcal{C}\boxtimes\mathcal{C}^{\otimes op}$, as a left $\mathcal{C}\boxtimes\mathcal{C}^{\otimes op}$-module category. By the construction given in \cite{DSPS13}, it follows that $End_R(R)$, the algebra of endomorphisms of $R$ as a left $R$-module, is given by $\mathds{k}$. Furthermore, we may assume that $\lambda\circ u = 1$, so that the $R$-$R$-bimodule map $m\circ \Delta:R\rightarrow R$ is given by multiplication by $\mathrm{dim}(\mathcal{C})$.

Now, let us also identify $\mathcal{C}\boxtimes\mathcal{C}$ with $Mod_{\mathcal{C}\boxtimes\mathcal{C}^{\otimes op}}(I)$ as a left $\mathcal{C}\boxtimes\mathcal{C}^{\otimes op}$-module 1-category. Under these identifications, the left $\mathcal{C}\boxtimes\mathcal{C}^{\otimes op}$-module functor $\otimes:\mathcal{C}\boxtimes\mathcal{C}\rightarrow \mathcal{C}$ is given by $(-)\otimes R_R$. In particular, its right adjoint (as a $\mathcal{C}\boxtimes\mathcal{C}^{\otimes op}$-module functor) is given by $(-)\otimes_R R_I:Mod_{\mathcal{C}\boxtimes\mathcal{C}^{\otimes op}}(R)\rightarrow Mod_{\mathcal{C}\boxtimes\mathcal{C}^{\otimes op}}(I)$. The unit for this adjunction is induced by the unit $u:I\rightarrow R$ of $R$, and the counit of this adjunction is induced by the multiplication $m:R\otimes R\rightarrow R$ viewed as a map of $R$-$R$-bimodules. The triangle identities follow from the fact that $R$ is an algebra.

But, $(-)\otimes_R R_I$ is also the left adjoint of $(-)\otimes R_R$. Namely, $\Delta$ supplies a map of $R$-$R$-bimodules $R\rightarrow R\otimes R$, which gives us the desired unit, and $\lambda$ provides us with the desired counit. These maps clearly satisfy the triangle identities. This means that the double right adjoint of $(-)\otimes R_R$ is $(-)\otimes R_R$ itself. Thus, we may take $\mu$ to be the natural isomorphism of left $\mathcal{C}\boxtimes\mathcal{C}^{\otimes op}$-module induced by the identity map on $R$ (viewed as a map of left $R$-modules). Using the above data, we find that $\mathrm{Tr}(\mu)=\mathrm{dim}(\mathcal{C})$ and $\mathrm{Tr}((\mu^{-1})^*)=1,$ where we have used $End_{R}(R)=\mathds{k}$. Putting everything together, we find $\mathrm{Dim}(\mathcal{C})=\mathrm{dim}(\mathcal{C})$ as desired.
\end{proof}

\begin{Remark}
If $\mathds{k}$ is algebraically closed, it follows from theorem 2.6.7 of \cite{DSPS13} that the fusion 1-category $\mathcal{C}$ is separable as an algebra in $\mathbf{2Vect}$ if and only if it is separable in the sense of definition 2.5.8 of \cite{DSPS13}. In particular, if we assume in addition that $\mathds{k}$ has characteristic zero, then every fusion category has non-zero global dimension by \cite{ENO}, and so is separable.
\end{Remark}

\begin{Proposition}\label{prop:relativedimension}
Let $\mathcal{B}$ be a perfect braided tensor 1-category with simple unit, and $\mathcal{C}$ be a separable $\mathcal{B}$-central tensor 1-category with simple unit. If $\mathcal{B}$ is a separable tensor 1-category, then $\mathcal{C}$ is separable as an algebra in $\mathbf{Mod}(\mathcal{B})$ if and only if $\mathcal{C}$ is a separable tensor 1-category.
\end{Proposition}
\begin{proof}
We begin by proving that $\mathcal{C}$ is perfect. In order to do so, it is enough to show that $\mathcal{C}\boxtimes\mathcal{C}$ is a finite semisimple 1-category. Now, $\mathcal{C}$ is a separable right $\mathcal{B}$-module category, thence it follows from the definition that there exists a separable algebra $R$ in $\mathcal{B}$ such that $\mathcal{C}\simeq Mod_{\mathcal{B}}(R)$ as right $\mathcal{B}$-module 1-categories. By proposition 3.8 of \cite{DSPS14}, $\mathcal{C}\boxtimes\mathcal{C}\simeq Mod_{\mathcal{B}\boxtimes\mathcal{B}}(R\boxtimes R)$ is a separable right $\mathcal{B}\boxtimes\mathcal{B}$-module 1-category. But, $\mathcal{B}\boxtimes\mathcal{B}$ is finite semisimple, so we find that $\mathcal{C}\boxtimes\mathcal{C}$ is finite semisimple.

Let us denote by $P:\mathcal{B}\boxtimes\mathcal{B}\rightarrow \mathcal{B}$ the monoidal structure of $\mathcal{B}$, and $T:\mathcal{C}\boxtimes\mathcal{C}\rightarrow\mathcal{C}$ the monoidal structure of $\mathcal{C}$. As $\mathcal{C}$ is a $\mathcal{B}$-central tensor 1-category, we can factor the $\mathcal{C}$-$\mathcal{C}$-bimodule functor $T$ as $$\mathcal{C}\boxtimes\mathcal{C}\xrightarrow{Q}\mathcal{C}\boxtimes_{\mathcal{B}}\mathcal{C}\xrightarrow{\widetilde{T}}\mathcal{C}.$$ In fact, the canonical $\mathcal{C}$-$\mathcal{C}$-bimodule functor $Q:\mathcal{C}\boxtimes\mathcal{C}\rightarrow\mathcal{C}\boxtimes_{\mathcal{B}}\mathcal{C}$ coincides with the functor obtained by applying $\mathcal{C}\boxtimes_{\mathcal{B}} (-)\boxtimes_{\mathcal{B}}\mathcal{C}$ to $P$.

Let us pick $\pi:P^{**}\cong P$ a natural isomorphism of $\mathcal{B}$-$\mathcal{B}$-bimodule functors. As $\mathcal{B}$ is separable, we may assume that

\settoheight{\algebra}{\includegraphics[width=30mm]{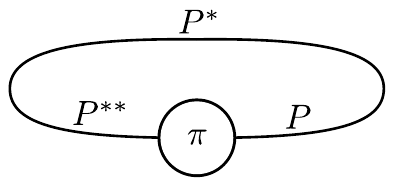}}

\begin{center}
\begin{tabular}{@{}cc@{}}
\includegraphics[width=30mm]{Pictures/dimension/tracepi.pdf} & \raisebox{0.45\algebra}{$=1$}.
\end{tabular}
\end{center}

\noindent We also pick $\nu:\widetilde{T}^{**}\cong\widetilde{T}$ a natural isomorphism of $\mathcal{C}$-$\mathcal{C}$-bimodule functors. Then, $\mu:=\nu\cdot (\mathcal{C}\boxtimes_{\mathcal{B}} \pi\boxtimes_{\mathcal{B}}\mathcal{C}):T^{**}\cong T$ is a natural isomorphism of $\mathcal{C}$-$\mathcal{C}$-bimodule functors. Thus, $\mathrm{Tr}(\mu)$ is given by

\settoheight{\algebra}{\includegraphics[width=60mm]{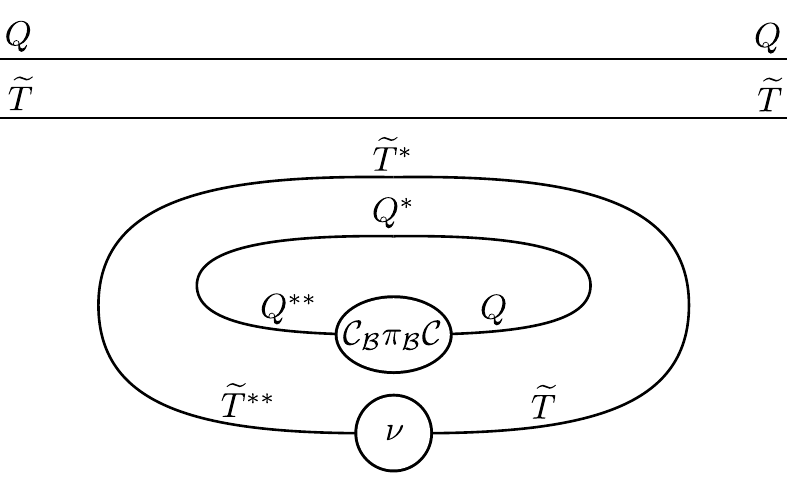}}

\begin{center}
\begin{tabular}{@{}cc@{}}
\raisebox{0.45\algebra}{$\mathrm{Tr}(\mu)=$} &\includegraphics[width=60mm]{Pictures/dimension/tracenupi.pdf}.
\end{tabular}
\end{center}

\noindent By construction, the inner loop of the diagram above evaluates to $1$, so we may remove it, and we find that $\mathrm{Tr}(\mu) = \mathrm{Tr}(\nu) \circ Q$. As precomposition by $Q$ is faithful, we find that $\mathrm{Tr}(\mu)$ is non-zero if and only if $\mathrm{Tr}(\nu)$ is non-zero, which concludes the proof of the result.
\end{proof}

\begin{Remark}
Let $\mathcal{B}$ be any braided finite semisimple tensor 1-category, then there exists connected separable algebras in $\mathbf{Mod}(\mathcal{B})$. For instance, $\mathcal{B}$ viewed as a $\mathcal{B}$-central tensor 1-category is such a connected separable algebra.
\end{Remark}

Now, let $F:\mathfrak{C}\rightarrow \mathfrak{D}$ be a monoidal linear 2-functor between finite semisimple monoidal 2-categories. As $F$ is monoidal and right adjoints are preserved by 2-functors, it follows that $F(A)$ is a rigid algebra in $\mathfrak{D}$. If we assume that $F(A)$ is connected, then it is natural to ask how the dimensions of $A$ and $F(A)$ compare, provided they are well-defined. More precisely, $F$ induces a map of $\mathds{k}$-algebras from $End_{A\mathrm{-}A}(m)\cong \mathbb{K}$ to $End_{F(A)\mathrm{-}F(A)}(F(m))\cong \mathbb{K}'$, which is necessarily a finite field extension. If we assume that $\mathbb{K}\cong\mathds{k}$ and $\mathbb{K}'\cong\mathds{k}$, which is automatic if $\mathds{k}$ is algebraically closed, then the dimensions of $A$ and $F(A)$ are both well-defined, and we have the following result.

\begin{Lemma}\label{lem:2fundimension}
Let $F:\mathfrak{C}\rightarrow \mathfrak{D}$ be a monoidal linear 2-functor between finite semisimple monoidal 2-categories, and $A$ a connected rigid algebra in $\mathfrak{C}$ such that $F(A)$ is connected, $\mathbb{K}\cong\mathds{k}$, and $\mathbb{K}'\cong\mathds{k}$. Then, we have $\mathrm{Dim}_{\mathfrak{C}}(A) = \mathrm{Dim}_{\mathfrak{D}}(F(A))$ in $\mathds{k}$.
\end{Lemma}
\begin{proof}
Without loss of generality, we may assume that both $\mathfrak{C}$ and $\mathfrak{D}$ are strict cubical. Thanks to the result of section 4.10 of \cite{GPS}, we may further assume that the underlying 2-functor of $F$ is strict. In particular, $F$ sends adjoint 1-morphisms to adjoint 1-morphisms on the nose. Following definition 2.5 of \cite{SP}, the monoidal structure on the 2-functor $F:\mathfrak{C}\rightarrow \mathfrak{D}$ provides us with an adjoint equivalence $e:F(A)\Box F(A)\simeq F(A\Box A),$ whose adjoint we denote by $e^*$.

Now, note that the multiplication 1-morphism of $F(A)$ is thus given by $$n:F(A)\Box F(A)\xrightarrow{e} F(A\Box A)\xrightarrow{F(m)}F(A).$$ In particular, a right adjoint to $n$ is given $$F(A)\xrightarrow{F(m^*)} F(A\Box A)\xrightarrow{e^*}F(A)\Box F(A),$$ with the canonical unit and counit 2-morphisms. Further, as $e^{**} = e$, we find that any choice of $A$-$A$-bimodule 2-isomorphism $\mu:m^{**}\cong m$ gives an $F(A)$-$F(A)$-bimodule 2-isomorphism $F(\mu)\circ e:n^{**}\cong n$.

\settoheight{\algebra}{\includegraphics[width=60mm]{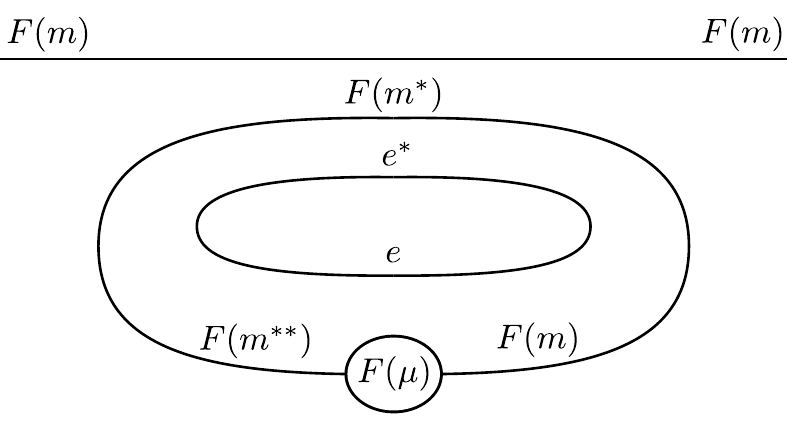}}

\begin{center}
\begin{tabular}{@{}cc@{}}
\raisebox{0.45\algebra}{$\mathrm{Tr}(F(\mu))=$} &\includegraphics[width=60mm]{Pictures/dimension/Ftrace.pdf}.
\end{tabular}
\end{center}

\noindent The inner loop evaluates to 1 as $e$ and $e^*$ form an adjoint equivalence. This shows that $\mathrm{Tr}(F(\mu))= \mathrm{Tr}(\mu)$ in $\mathds{k}$. One checks similarly that $\mathrm{Tr}(F(\mu^{-1})^{*})= \mathrm{Tr}((\mu^{-1})^{*})$ in $\mathds{k}$, thereby establishing the desired equality.
\end{proof}

Following \cite{JFR}, we will say that an algebra $A$ in a compact semisimple monoidal 2-category $\mathfrak{C}$ is strongly connected if the 1-morphism $i:I\rightarrow A$ is the inclusion of a simple summand. Let us examine some examples of strongly connected algebras. If $\mathds{k}$ is an algebraically closed field, then every connected algebra in $\mathbf{2Vect}_G$, the fusion 2-category of $G$-graded 2-vector spaces for a finite group $G$, is strongly connected. However, this is usually not the case. Namely, provided $\mathds{k}$ is, in addition, a field of characteristic zero, then strongly connected rigid algebras in $\mathscr{Z}(\mathbf{2Vect}_G)$, the Drinfel'd center of $\mathbf{2Vect}_G$, are precisely $G$-crossed fusion 1-categories. On the other hand, connected rigid algebras in $\mathscr{Z}(\mathbf{2Vect}_G)$ correspond exactly to those $G$-crossed multifusion 1-categories $\mathcal{C}$ for which the action of $G$ on $End_{\mathcal{C}}(I)$ is transitive.

The notion of a strongly connected algebra is useful because it is often preserved by monoidal 2-functors. Namely, if $F:\mathfrak{C}\rightarrow\mathfrak{D}$ is a monoidal 2-functor between compact semisimple monoidal 2-category whose monoidal units are simple, and $A$ is a strongly connected algebra in $\mathfrak{C}$, then $F(A)$ is a strongly connected algebra. Combining this observation together with lemma \ref{lem:2fundimension} gives the following result. 

\begin{Corollary}
Let $G$ be a finite group, and $\mathds{k}$ an algebraically closed field of characteristic zero. Then, every connected rigid algebra in $\mathbf{2Vect}_G$ is separable. Further, every strongly connected rigid algebra in $\mathscr{Z}(\mathbf{2Vect}_G)$ is separable.
\end{Corollary}

Now, let $\mathds{k}$ be an arbitrary field. Recall from example \ref{ex:algebras2VectG}, that given any finite group $G$ and 4-cocycle $\omega$ for $G$ with coefficients in $\mathds{k}^{\times}$, we can form the compact semisimple 2-category $\mathbf{2Vect}_{G}^{\omega}$, whose objects are $G$-graded perfect 1-categories, and with coherence data twisted by $\omega$. If $H\subseteq G$ is a subgroup, and $\gamma$ is a 3-cochain for $H$ with coefficients in $\mathds{k}^{\times}$ such that $d\gamma = \omega|_H$, then we can view $\mathbf{Vect}_H^{\gamma}$, the monoidal 1-category of $H$-graded finite dimensional $\mathds{k}$-vector spaces with associator twisted by $\gamma$ as an algebra in $\mathbf{2Vect}_{G}^{\omega}$.

\begin{Corollary}\label{cor:gradedrigiddimension}
The connected algebra $\mathbf{Vect}_H^{\gamma}$ in $\mathbf{2Vect}_G^{\omega}$ is rigid has dimension $|H|$.
\end{Corollary}
\begin{proof}
As $\gamma$ trivializes $\omega|_H$, we can use it to construct a monoidal linear 2-functor $F:\mathbf{2Vect}_H\hookrightarrow \mathbf{2Vect}_G^{\omega}$, which induces the inclusion $H\subseteq G$ on equivalence classes of simple objects. In particular, we have $F(\mathbf{Vect}_H) = \mathbf{Vect}_H^{\gamma}$. But $\mathbf{Vect}_H$ is clearly rigid as a monoidal 1-category, so that it is a connected rigid algebra in $\mathbf{2Vect}_H$ by example \ref{ex:gradedrigid}. Now, let us write $U:\mathbf{2Vect}_H\rightarrow \mathbf{2Vect}$ for the monoidal linear 2-functor forgetting the grading. The rigid algebra $U(\mathbf{Vect}_H^{\kappa})$ in $\mathbf{2Vect}$ is exactly the tensor category of finite $H$-graded $\mathds{k}$-vector spaces, which has dimension $|H|$. This proves the result.
\end{proof}

\begin{Corollary}\label{cor:bimoduletwistedgraded}
If $char(\mathds{k})$ and $|H|$ are coprime, the 2-category of $\mathbf{Vect}_H^{\gamma}$-$\mathbf{Vect}_H^{\gamma}$-bimodules in $\mathbf{2Vect}_G^{\omega}$ is compact semisimple.
\end{Corollary}

\appendix

\section{Diagrams}\label{sec:diag}

\subsection{Proof of lemma \ref{lem:coherenceright}}\label{sub:coherencerightdiagrams}

\vfill

\begin{figure}[!hbt]
\centering
\begin{minipage}{.5\textwidth}
  \centering
    \includegraphics[width=55mm]{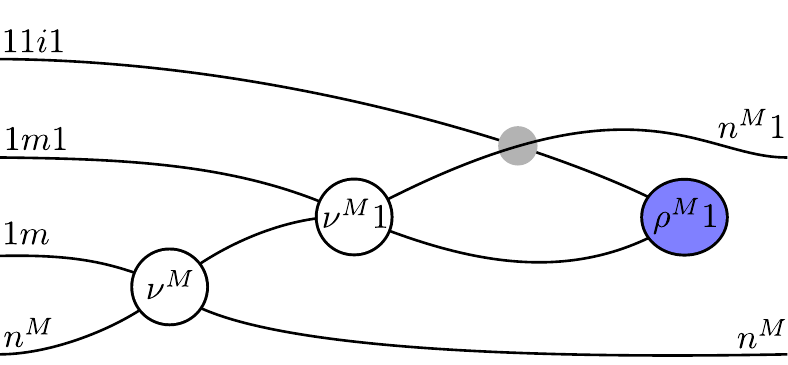}
    \caption{Equality (Part 1)}
    \label{fig:cohright1}
\end{minipage}%
\begin{minipage}{.5\textwidth}
  \centering
    \includegraphics[width=55mm]{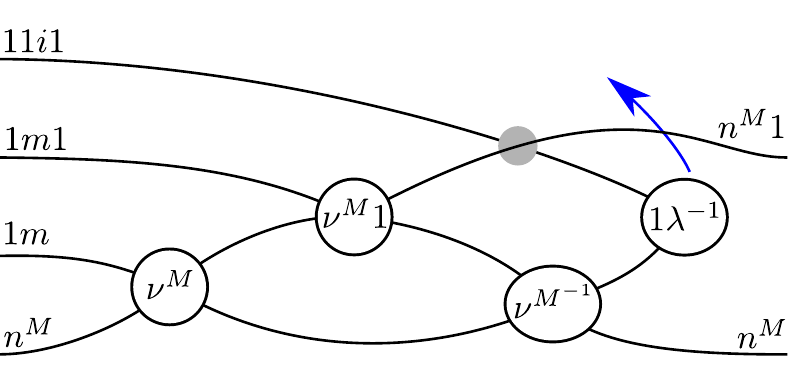}
    \caption{Equality (Part 2)}
    \label{fig:cohright2}
\end{minipage}
\end{figure}

\vfill

\begin{figure}[!hbt]
\centering
\begin{minipage}{.5\textwidth}
  \centering
    \includegraphics[width=55mm]{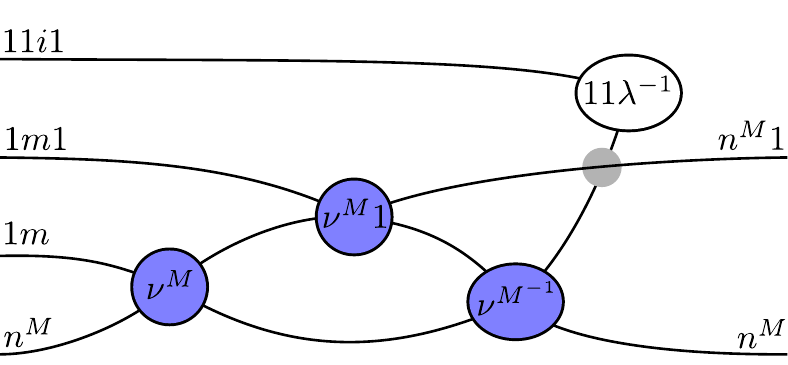}
    \caption{Equality (Part 3)}
    \label{fig:cohright3}
\end{minipage}%
\begin{minipage}{.5\textwidth}
  \centering
    \includegraphics[width=55mm]{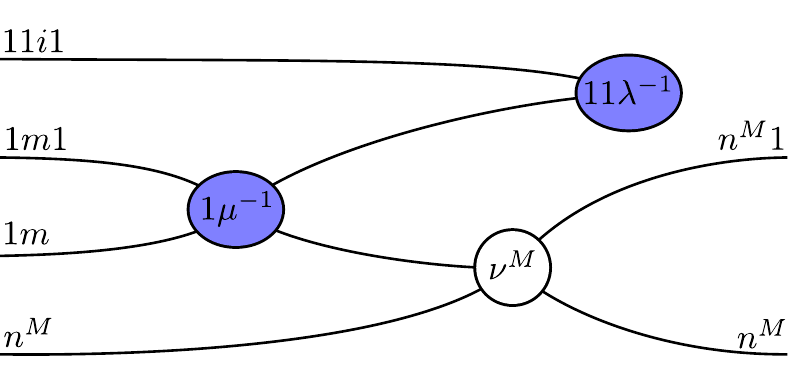}
    \caption{Equality (Part 4)}
    \label{fig:cohright4}
\end{minipage}
\end{figure}

\vfill

\newpage

\vfill

\begin{figure}[!hbt]
\centering
\begin{minipage}{.5\textwidth}
  \centering
    \includegraphics[width=55mm]{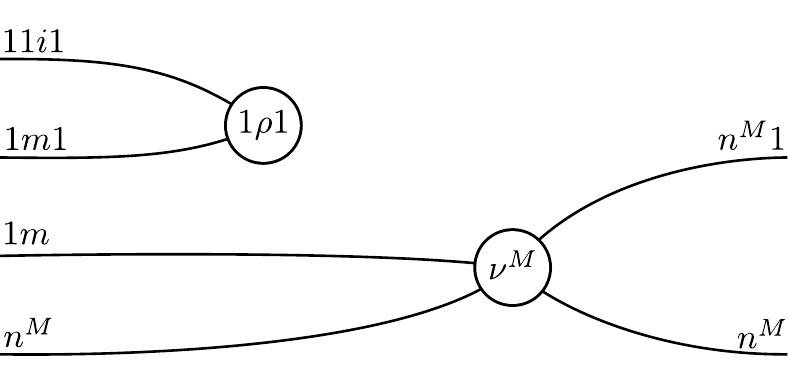}
    \caption{Equality (Part 5)}
    \label{fig:cohright5}
\end{minipage}%
\begin{minipage}{.5\textwidth}
  \centering
    \includegraphics[width=55mm]{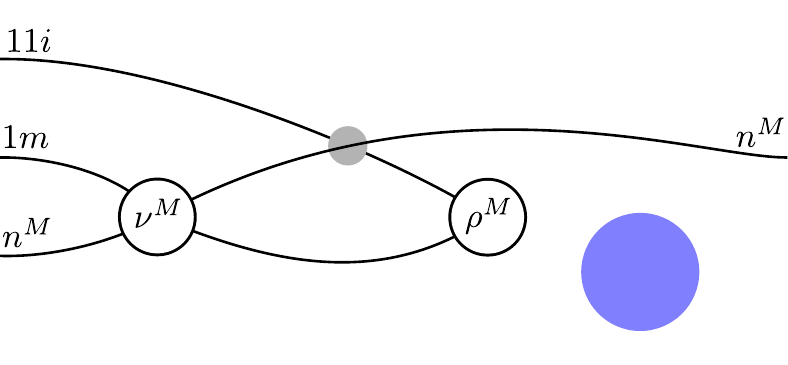}
    \caption{Coherence (Part 1)}
    \label{fig:cohright10}
\end{minipage}
\end{figure}

\vfill

\begin{figure}[!hbt]
\centering
\begin{minipage}{.5\textwidth}
  \centering
    \includegraphics[width=55mm]{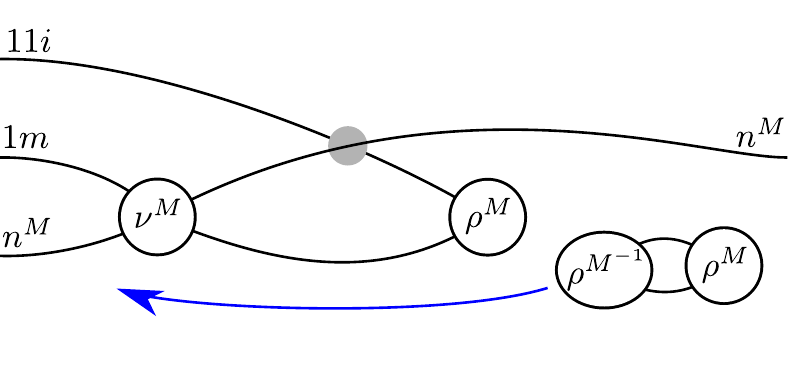}
    \caption{Coherence (Part 2)}
    \label{fig:cohright11}
\end{minipage}%
\begin{minipage}{.5\textwidth}
  \centering
    \includegraphics[width=55mm]{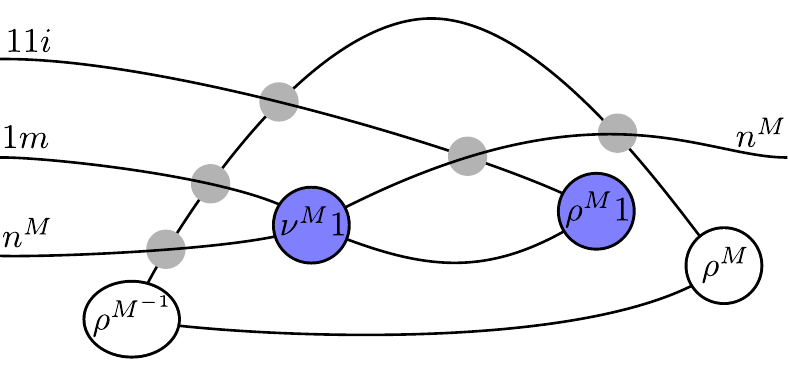}
    \caption{Coherence (Part 3)}
    \label{fig:cohright12}
\end{minipage}
\end{figure}

\vfill

\begin{figure}[!hbt]
\centering
\begin{minipage}{.5\textwidth}
  \centering
    \includegraphics[width=55mm]{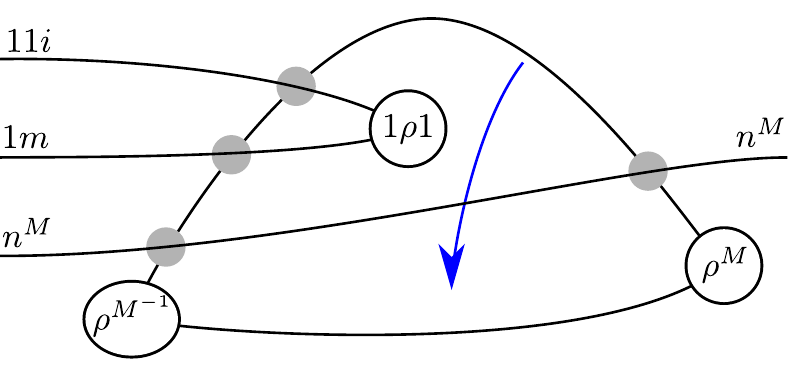}
    \caption{Coherence (Part 4)}
    \label{fig:cohright13}
\end{minipage}%
\begin{minipage}{.5\textwidth}
  \centering
    \includegraphics[width=55mm]{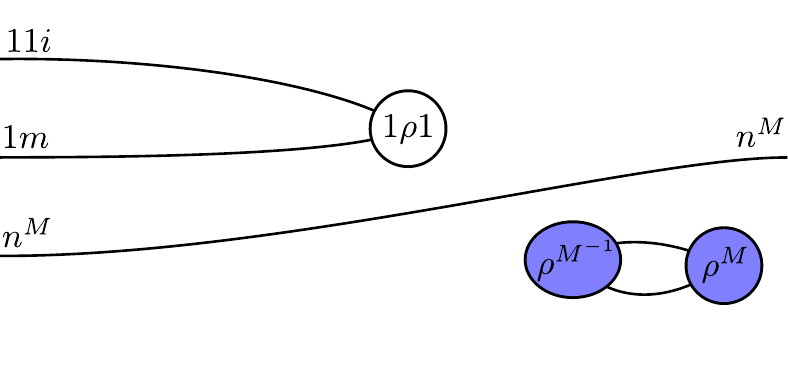}
    \caption{Coherence (Part 5)}
    \label{fig:cohright14}
\end{minipage}
\end{figure}

\vfill

\FloatBarrier

\begin{landscape}
\subsection{Proof of lemma \ref{lem:modulep}}\label{sub:modulepdiagrams}
    \vspace*{\fill}
    \begin{figure}[!hbt]
    \centering
    \includegraphics[width=150mm]{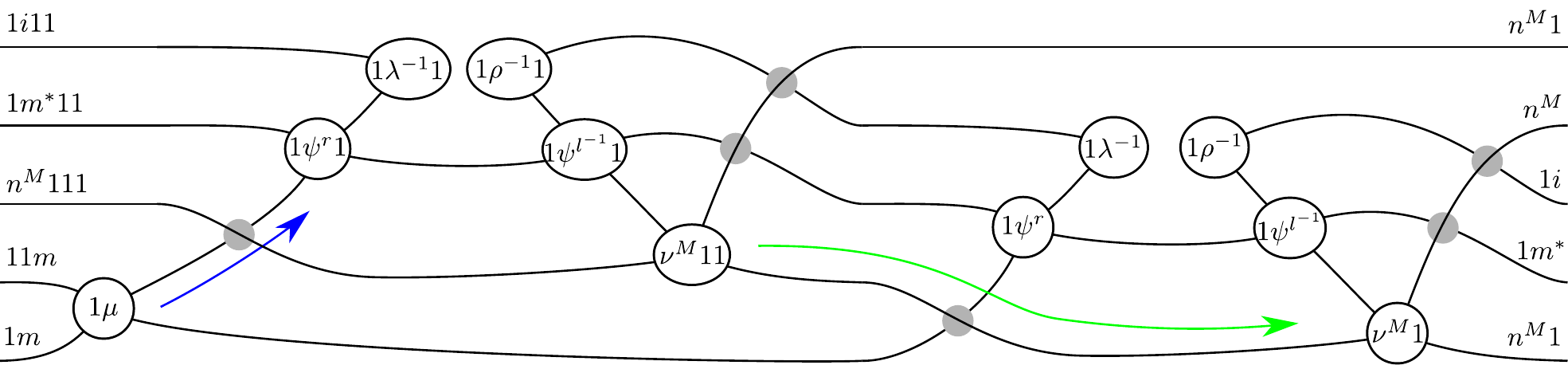}
    \caption{Axiom a (Part 1)}
    \label{fig:cohpsip1}
    \end{figure}
    \begin{figure}[!hbt]
    \centering
    \includegraphics[width=150mm]{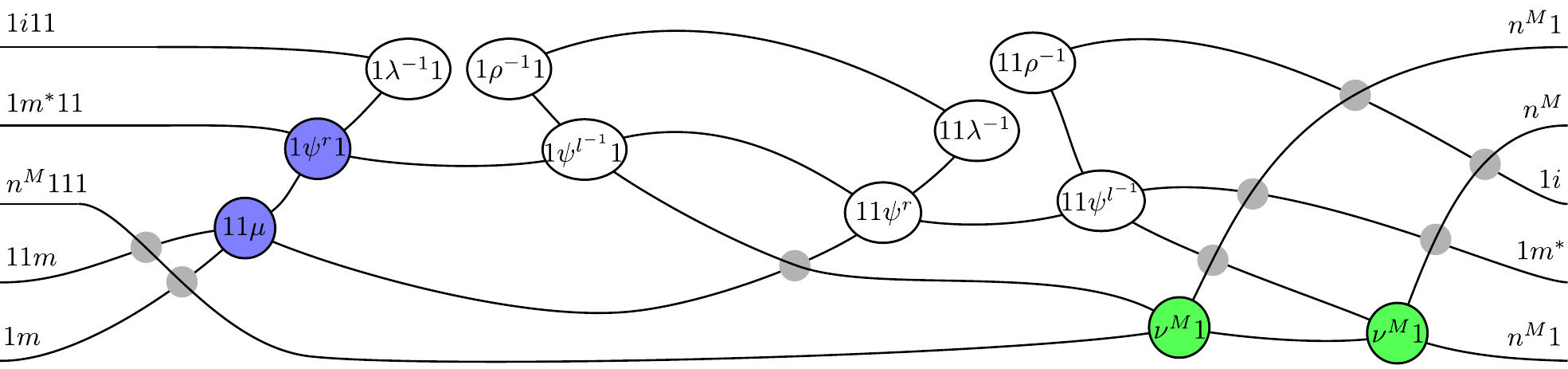}
    \caption{Axiom a (Part 2)}
    \label{fig:cohpsip2}
    \end{figure}
    \vfill
\end{landscape}

\begin{landscape}
    \vspace*{\fill}
    \begin{figure}[!hbt]
    \centering
    \includegraphics[width=150mm]{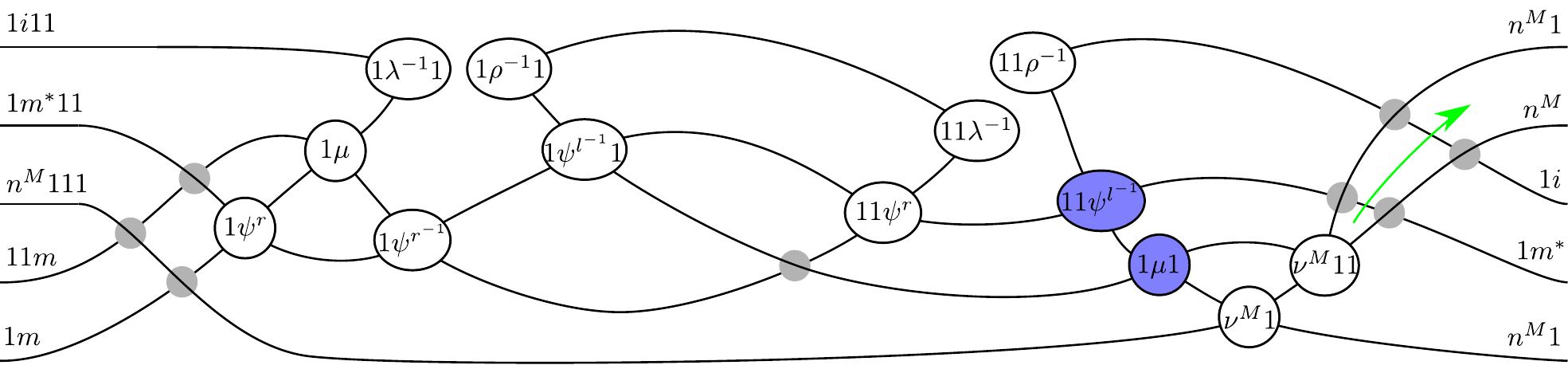}
    \caption{Axiom a (Part 3)}
    \label{fig:cohpsip3}
    \end{figure}
    \begin{figure}[!hbt]
    \centering
    \includegraphics[width=150mm]{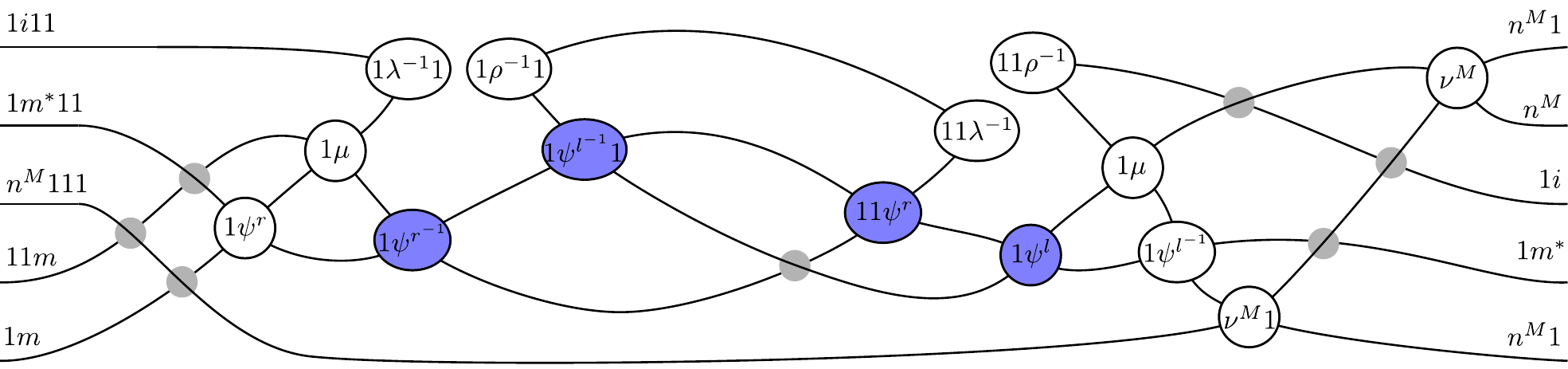}
    \caption{Axiom a (Part 4)}
    \label{fig:cohpsip4}
    \end{figure}
    \vfill
\end{landscape}

\begin{figure}[!hbt]
\centering 
\includegraphics[width=110mm]{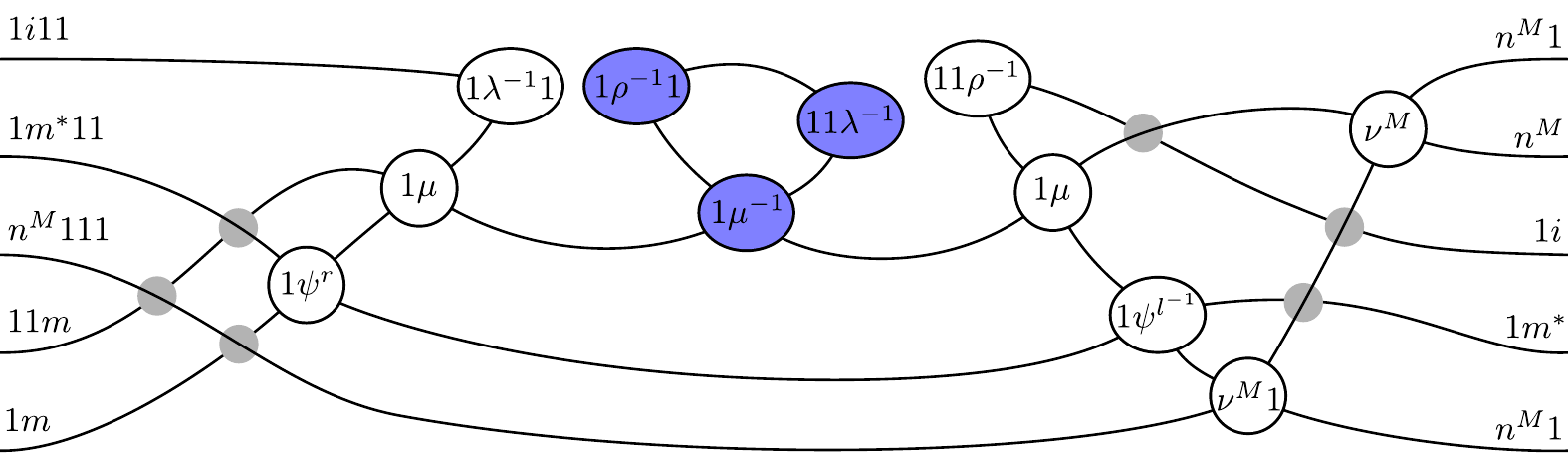}
\caption{Axiom a (Part 5)}
\label{fig:cohpsip5}
\end{figure}

\begin{figure}[!hbt]
\centering 
\includegraphics[width=95mm]{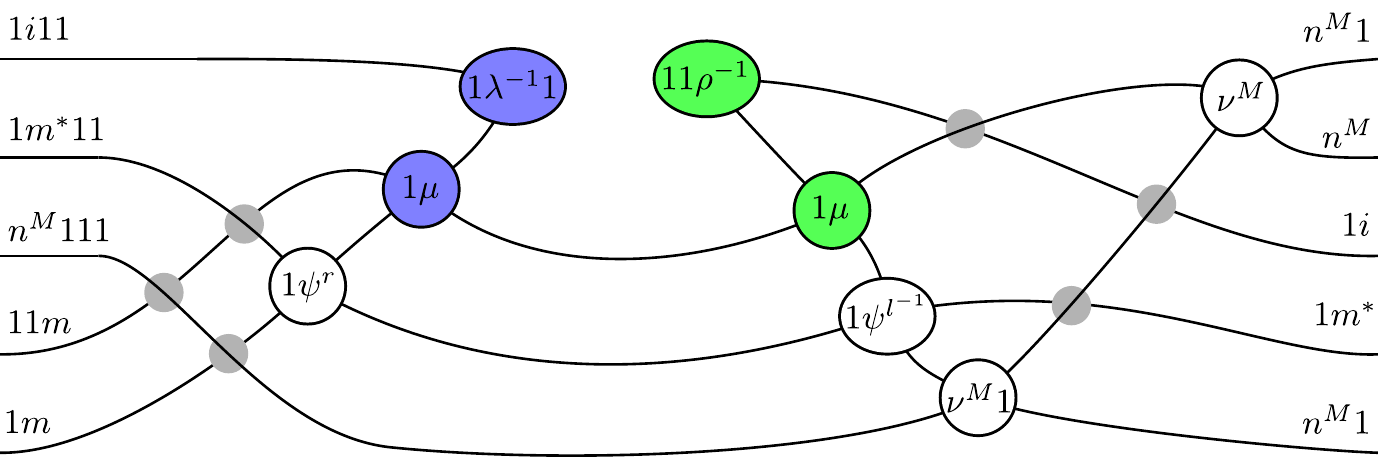}
\caption{Axiom a (Part 6)}
\label{fig:cohpsip6}
\end{figure}
    
\begin{figure}[!hbt]
\centering 
\includegraphics[width=95mm]{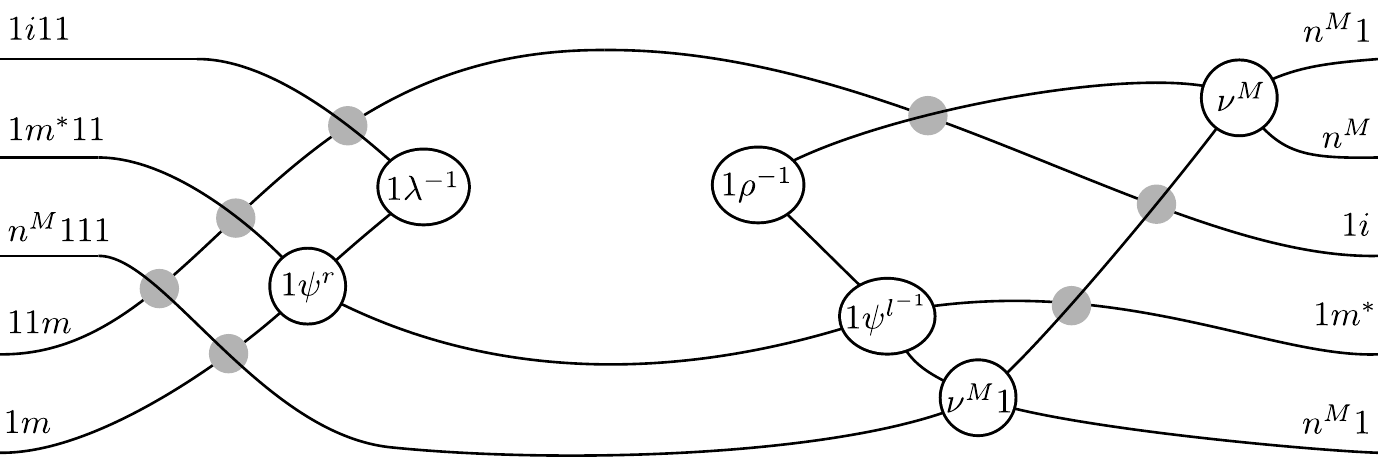}
\caption{Axiom a (Part 7)}
\label{fig:cohpsip7}
\end{figure}

\begin{figure}[!hbt]
\centering 
\includegraphics[width=75mm]{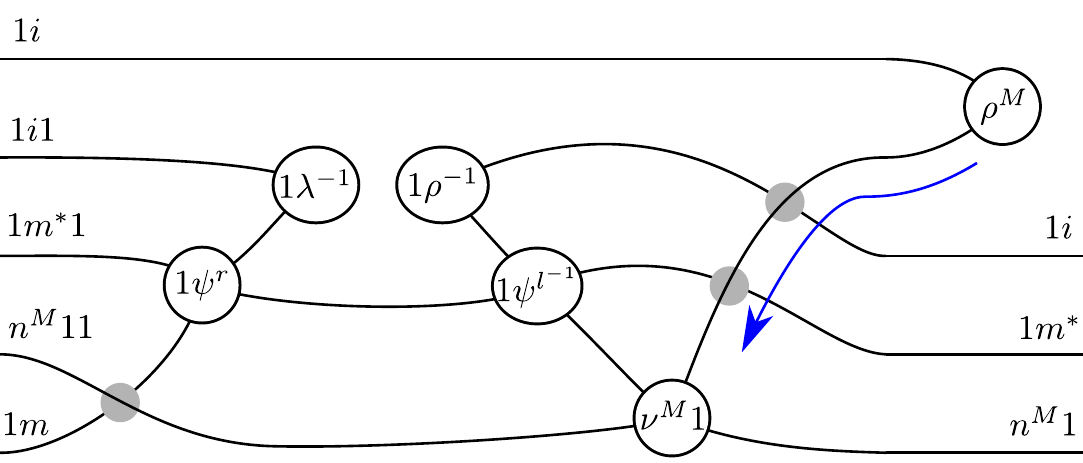}
\caption{Axiom b (Part 1)}
\label{fig:cohupsip1}
\end{figure}
    
\begin{figure}[!hbt]
\centering
\includegraphics[width=82.5mm]{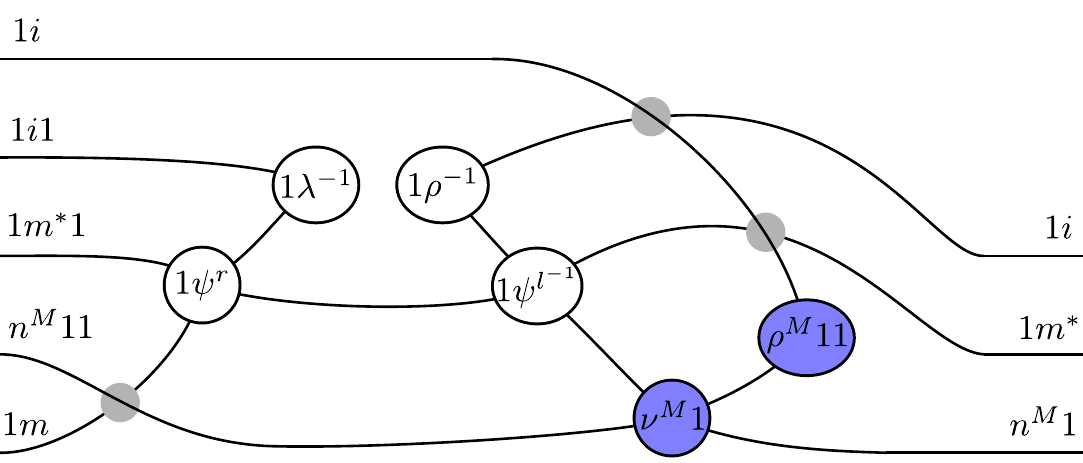}
\caption{Axiom b (Part 2)}
\label{fig:cohupsip2}
\end{figure}

\begin{figure}[!hbt]
\centering
\includegraphics[width=82.5mm]{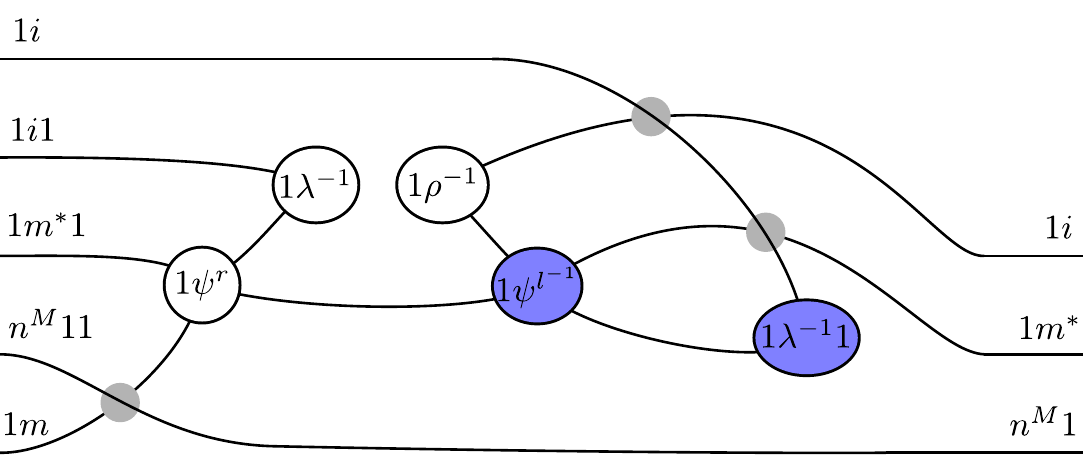}
\caption{Axiom b (Part 3)}
\label{fig:cohupsip3}
\end{figure}

\begin{figure}[!hbt]
\centering
\begin{minipage}{.6\textwidth}
  \centering
\includegraphics[width=67.5mm]{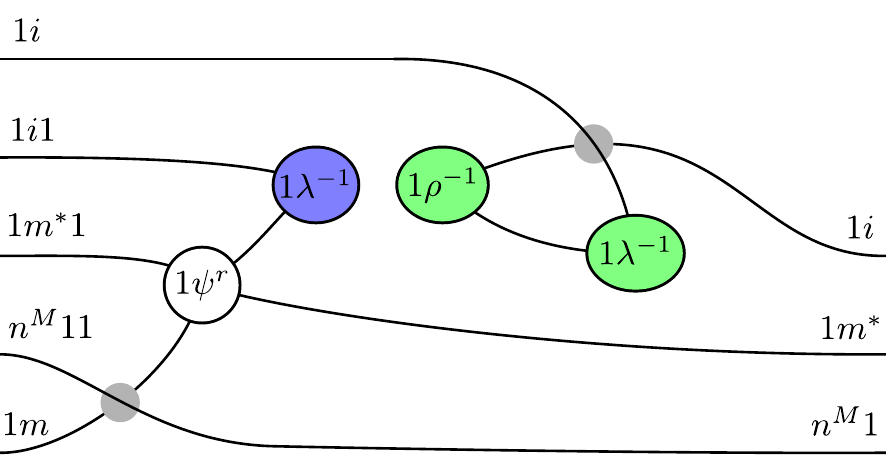}
\caption{Axiom b (Part 4)}
\label{fig:cohupsip4}
\end{minipage}%
\begin{minipage}{.4\textwidth}
  \centering
\includegraphics[width=40mm]{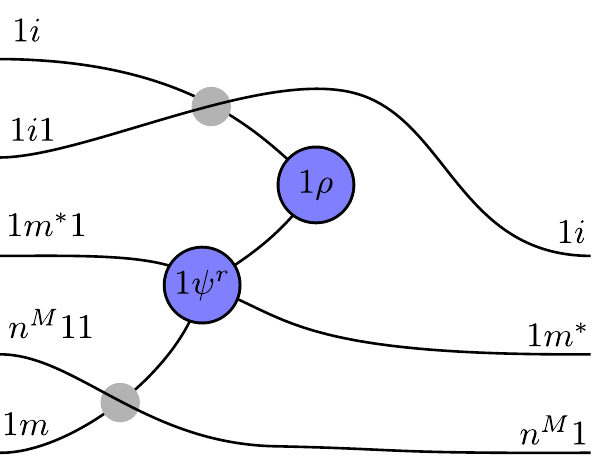}
\caption{Axiom b (Part 5)}
\label{fig:cohupsip5}
\end{minipage}
\end{figure}

\begin{figure}[!hbt]
\centering
\begin{minipage}{.5\textwidth}
  \centering
\includegraphics[width=52.5mm]{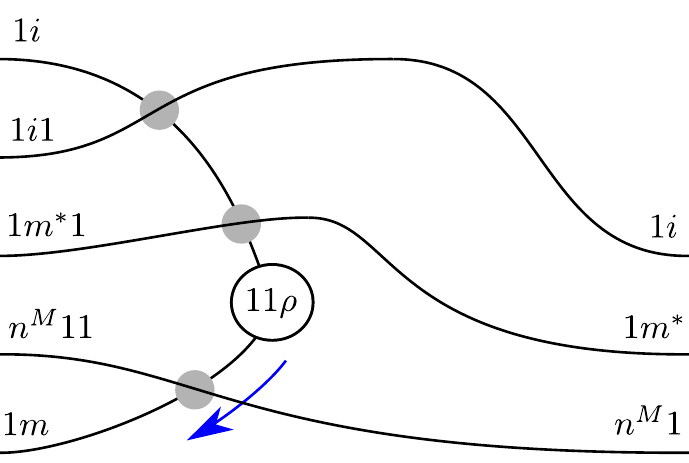}
\caption{Axiom b (Part 6)}
\label{fig:cohupsip6}
\end{minipage}%
\begin{minipage}{.5\textwidth}
  \centering
\includegraphics[width=52.5mm]{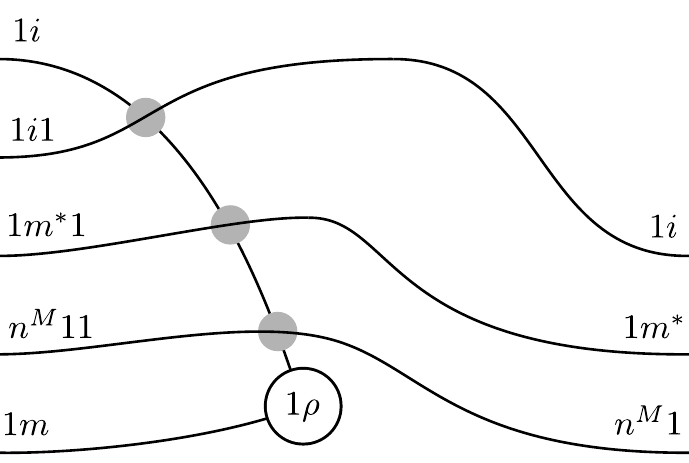}
\caption{Axiom b (Part 7)}
\label{fig:cohupsip7}
\end{minipage}
\end{figure}

\FloatBarrier

\begin{landscape}
\subsection{Proof of lemma \ref{lem:moduleepsiloneta}}\label{sub:moduleepsilonetadiagrams}
    \vspace*{\fill}
    \begin{figure}[!hbt]
    \centering
    \includegraphics[width=142.5mm]{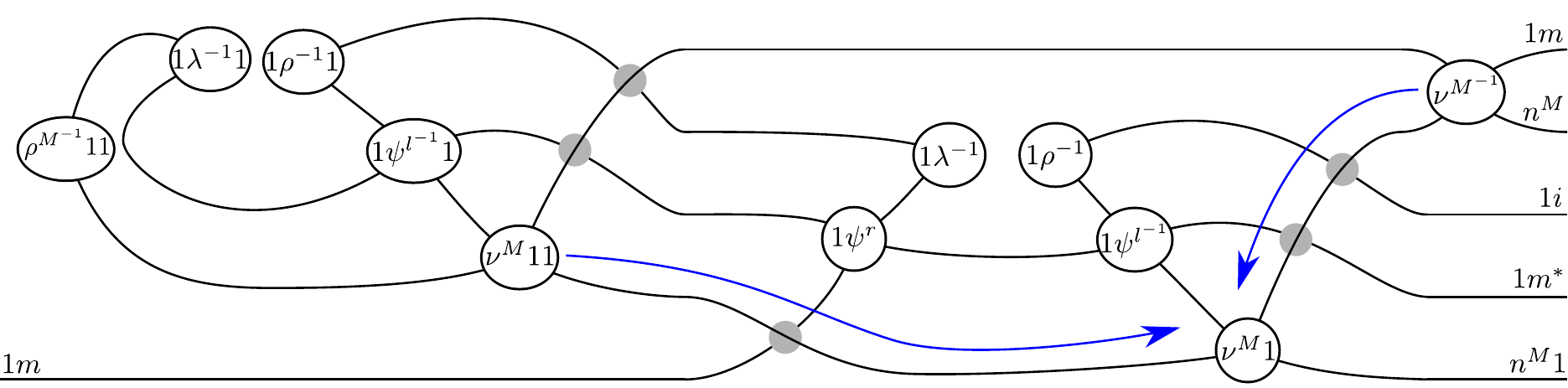}
    \caption{Module 2-Morphism (Part 1)}
    \label{fig:etamodule1}
    \end{figure}
    \begin{figure}[!hbt]
    \centering
    \includegraphics[width=142.5mm]{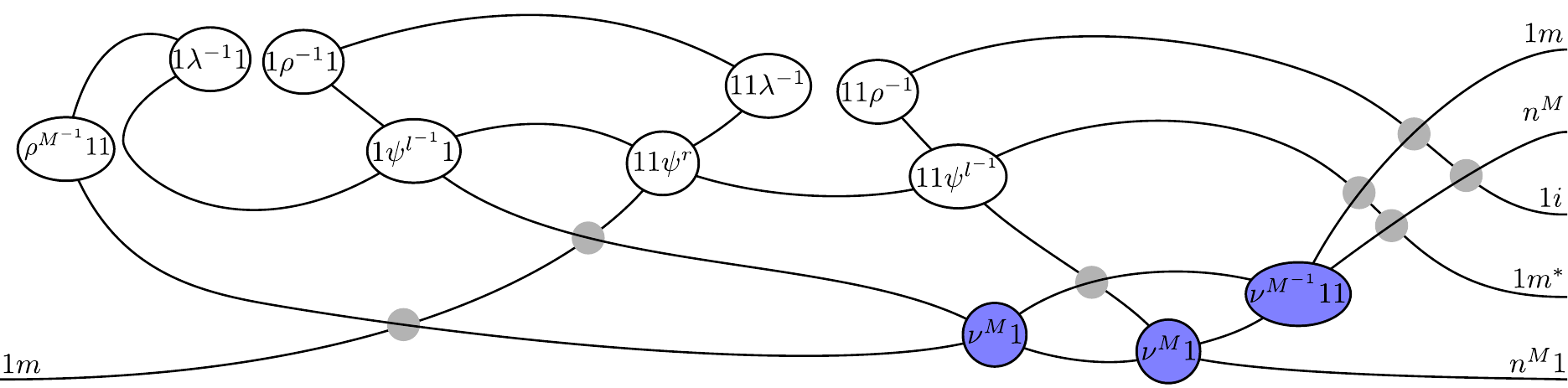}
    \caption{Module 2-Morphism (Part 2)}
    \label{fig:etamodule2}
    \end{figure}
    \vfill
\end{landscape}

\begin{landscape}
    \vspace*{\fill}
    \begin{figure}[!hbt]
    \centering
    \includegraphics[width=142.5mm]{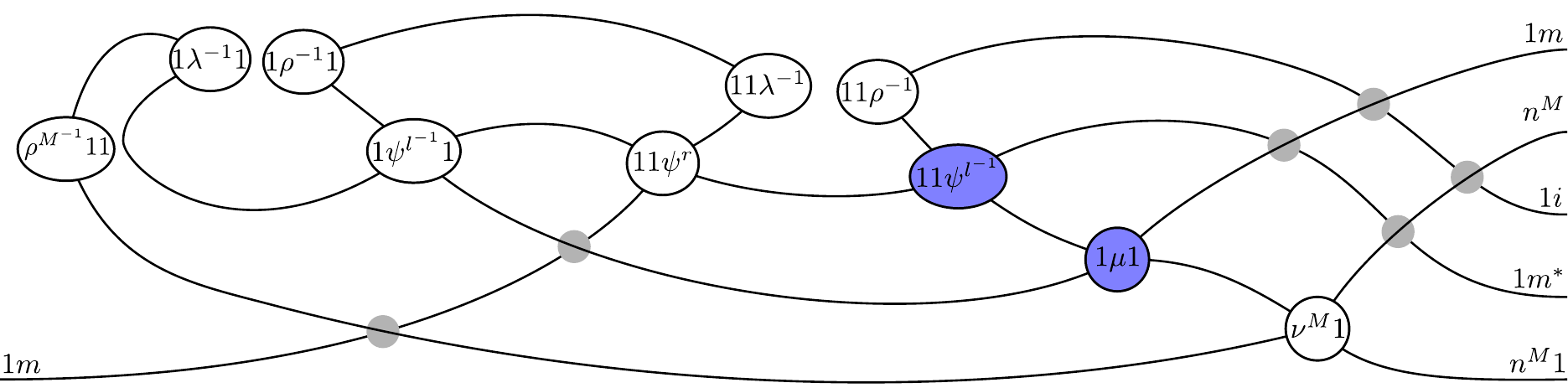}
    \caption{Module 2-Morphism (Part 3)}
    \label{fig:etamodule3}
    \end{figure}
    \begin{figure}[!hbt]
    \centering
    \includegraphics[width=142.5mm]{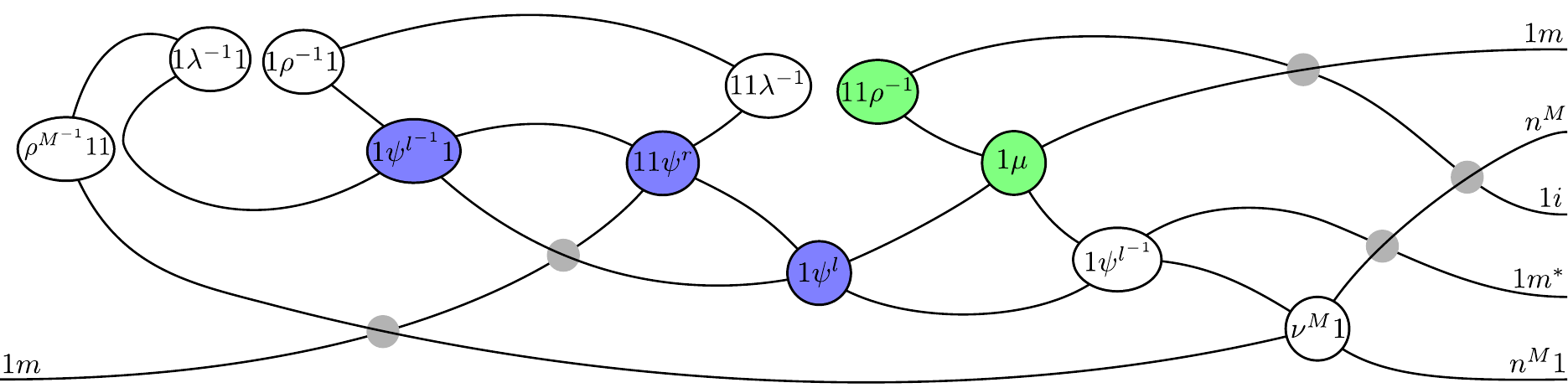}
    \caption{Module 2-Morphism (Part 4)}
    \label{fig:etamodule4}
    \end{figure}
    \vfill
\end{landscape}

\begin{figure}[!hbt]
\centering
\includegraphics[width=112.5mm]{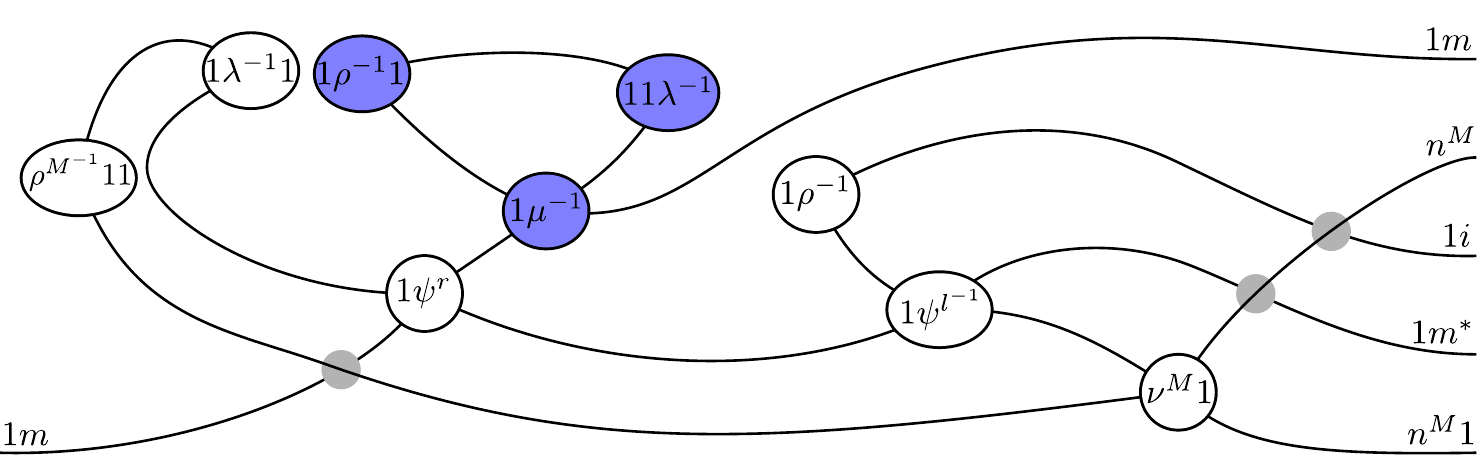}
\caption{Module 2-Morphism (Part 5)}
\label{fig:etamodule5}
\end{figure}

\begin{figure}[!hbt]
\centering
\includegraphics[width=97.5mm]{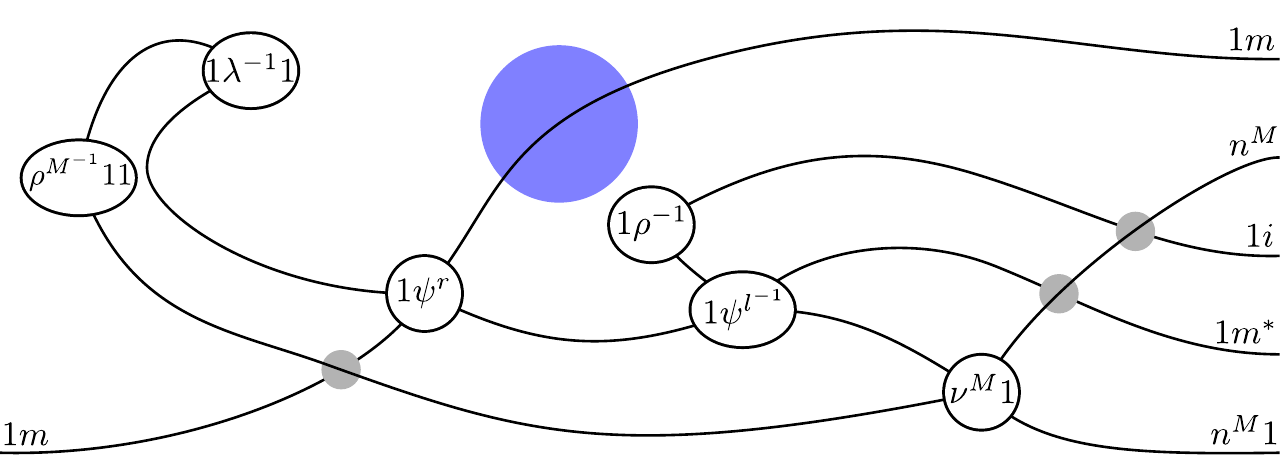}
\caption{Module 2-Morphism (Part 6)}
\label{fig:etamodule6}
\end{figure}

\begin{figure}[!hbt]
\centering
\includegraphics[width=112.5mm]{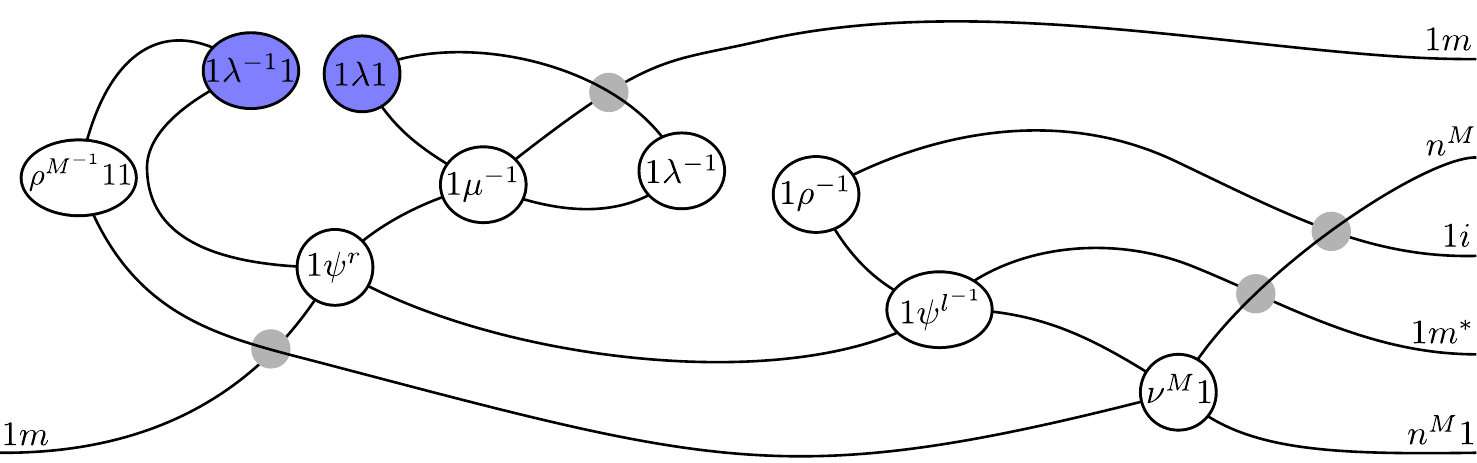}
\caption{Module 2-Morphism (Part 7)}
\label{fig:etamodule7}
\end{figure}

\begin{figure}[!hbt]
\centering
\includegraphics[width=105mm]{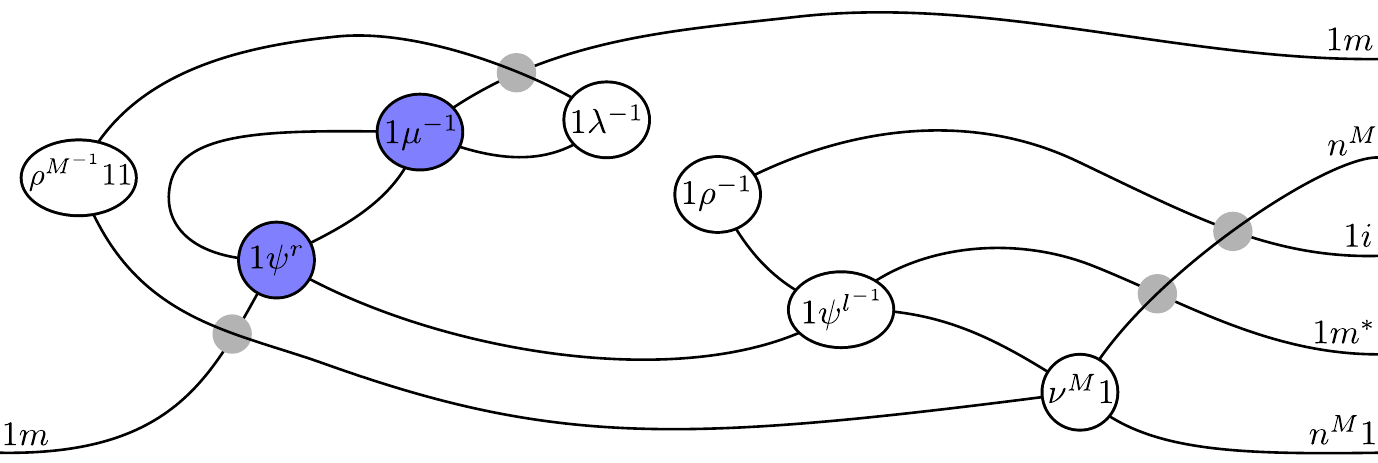}
\caption{Module 2-Morphism (Part 8)}
\label{fig:etamodule8}
\end{figure}

\begin{figure}[!hbt]
\centering
\includegraphics[width=97.5mm]{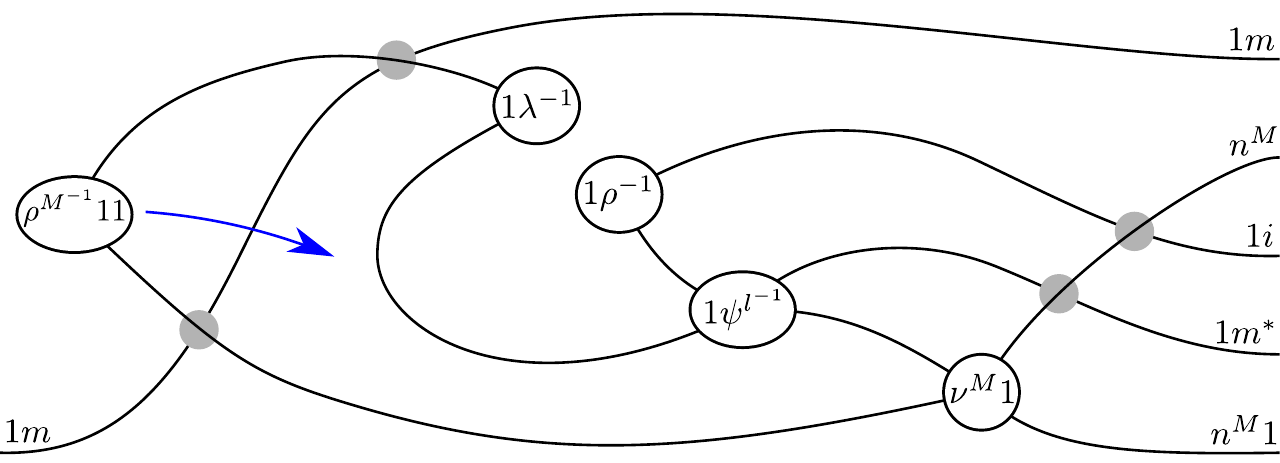}
\caption{Module 2-Morphism (Part 9)}
\label{fig:etamodule9}
\end{figure}

\FloatBarrier
\subsection{Proof of lemma \ref{lem:rigidactionrightadjoint}}\label{sub:rigidactionrightadjointdiagrams}

\begin{figure}[!hbt]
\centering
\includegraphics[width=105mm]{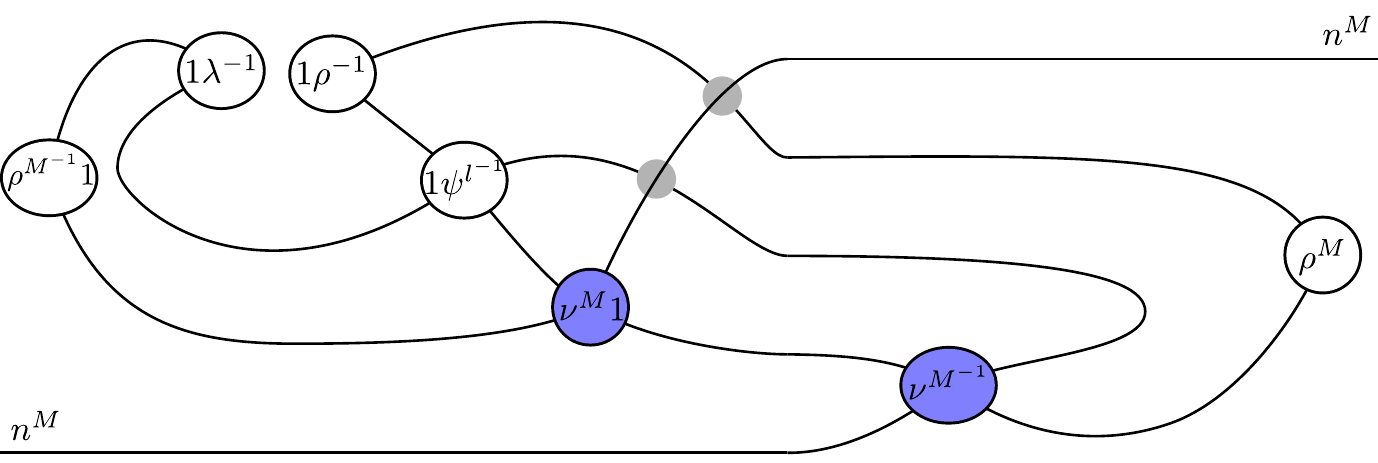}
\caption{Triangle 1 (Part 1)}
\label{fig:triangle1}
\end{figure}
    
\begin{figure}[!hbt]
\centering
\includegraphics[width=105mm]{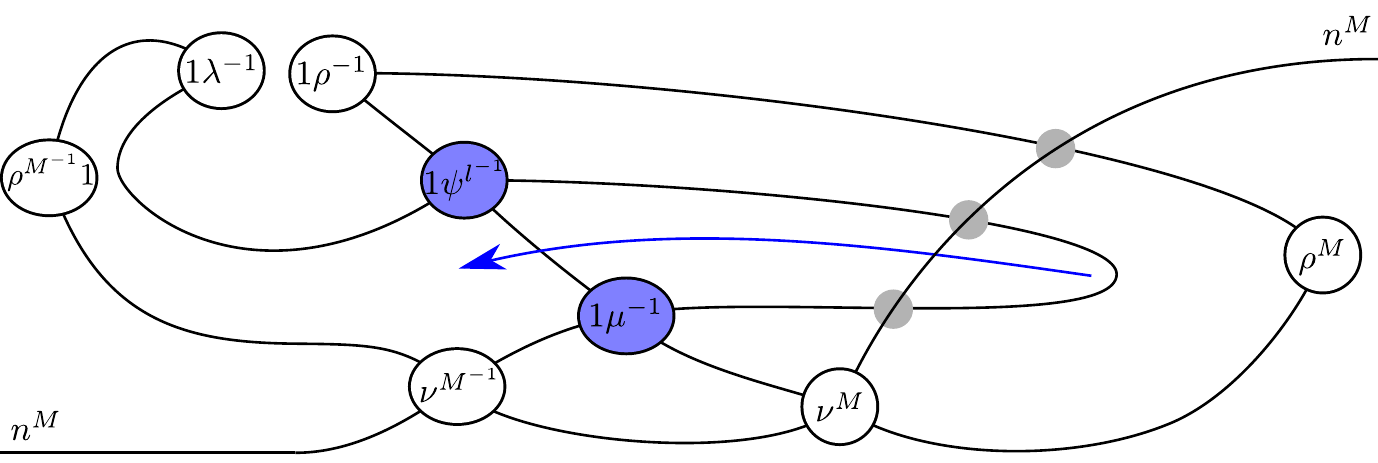}
\caption{Triangle 1 (Part 2)}
\label{fig:triangle2}
\end{figure}

\begin{figure}[!hbt]
\centering
\includegraphics[width=75mm]{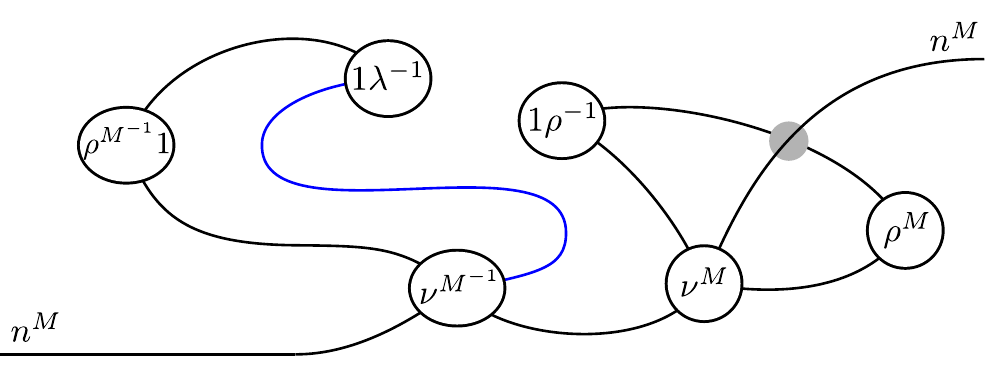}
\caption{Triangle 1 (Part 3)}
\label{fig:triangle3}
\end{figure}
    
\begin{figure}[!hbt]
\centering
\includegraphics[width=75mm]{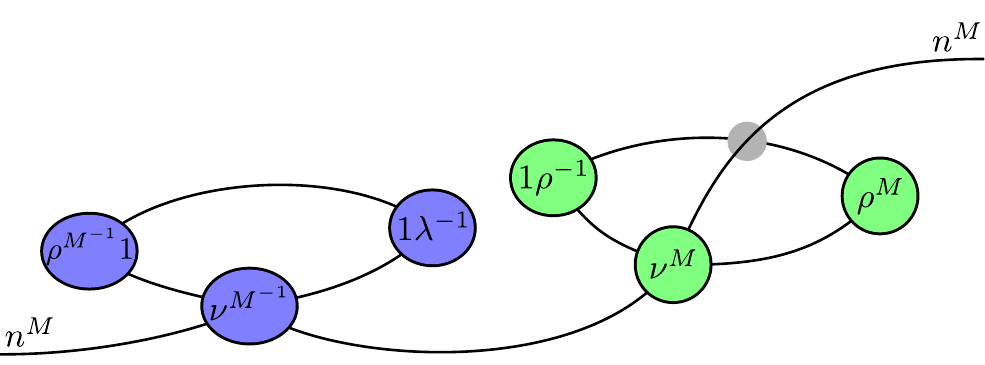}
\caption{Triangle 1 (Part 4)}
\label{fig:triangle4}
\end{figure}

\begin{figure}[!hbt]
\centering
\includegraphics[width=105mm]{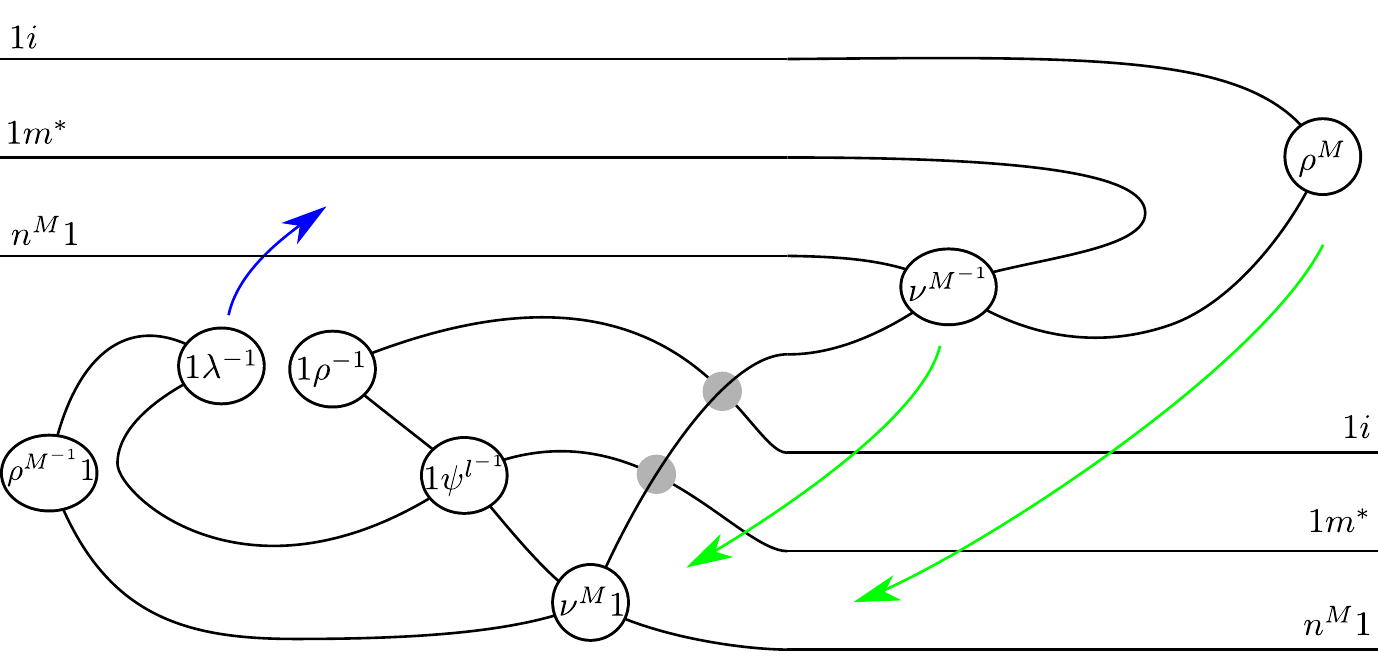}
\caption{Triangle 2 (Part 1)}
\label{fig:triangle10}
\end{figure}
    
\begin{figure}[!hbt]
\centering
\includegraphics[width=97.5mm]{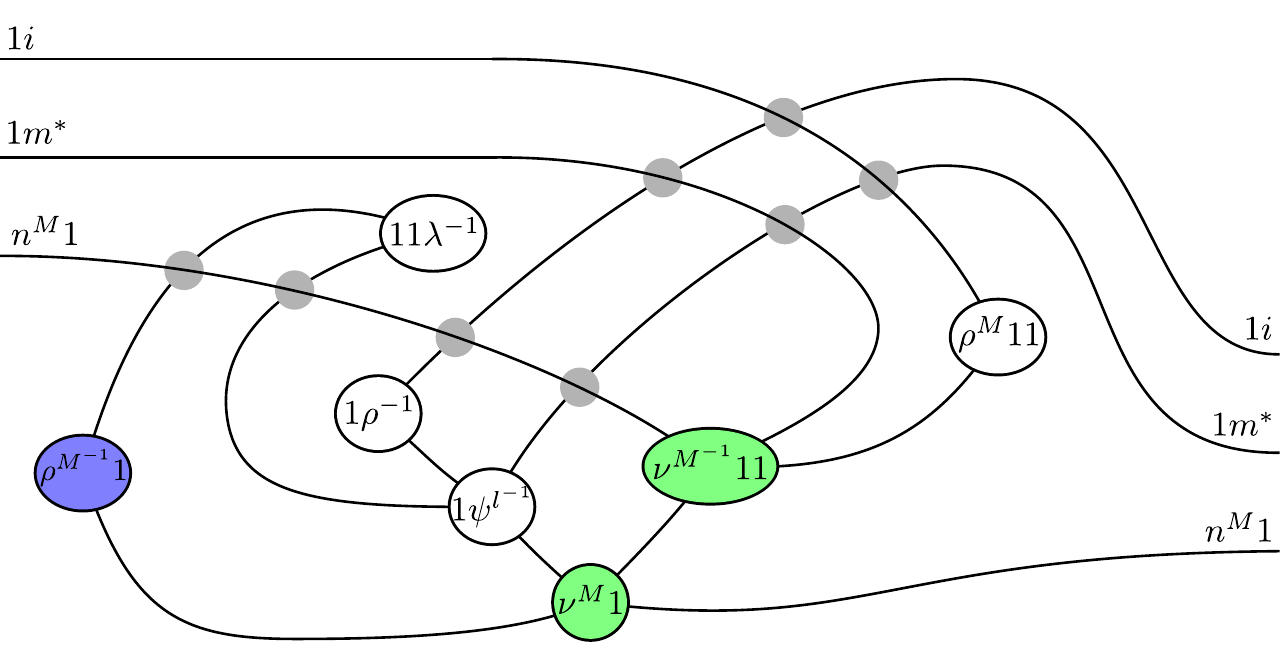}
\caption{Triangle 2 (Part 2)}
\label{fig:triangle11}
\end{figure}

\begin{figure}[!hbt]
\centering
\includegraphics[width=112.5mm]{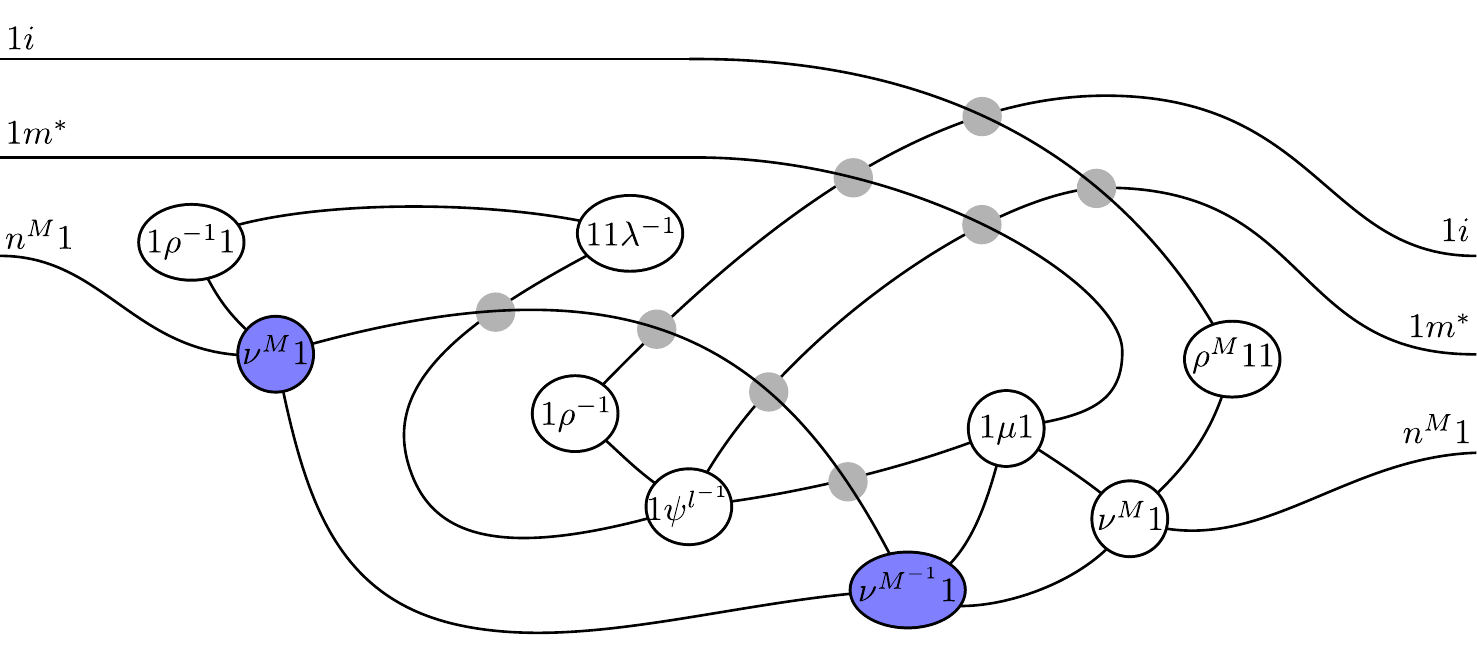}
\caption{Triangle 2 (Part 3)}
\label{fig:triangle12}
\end{figure}
    
\begin{figure}[!hbt]
\centering
\includegraphics[width=105mm]{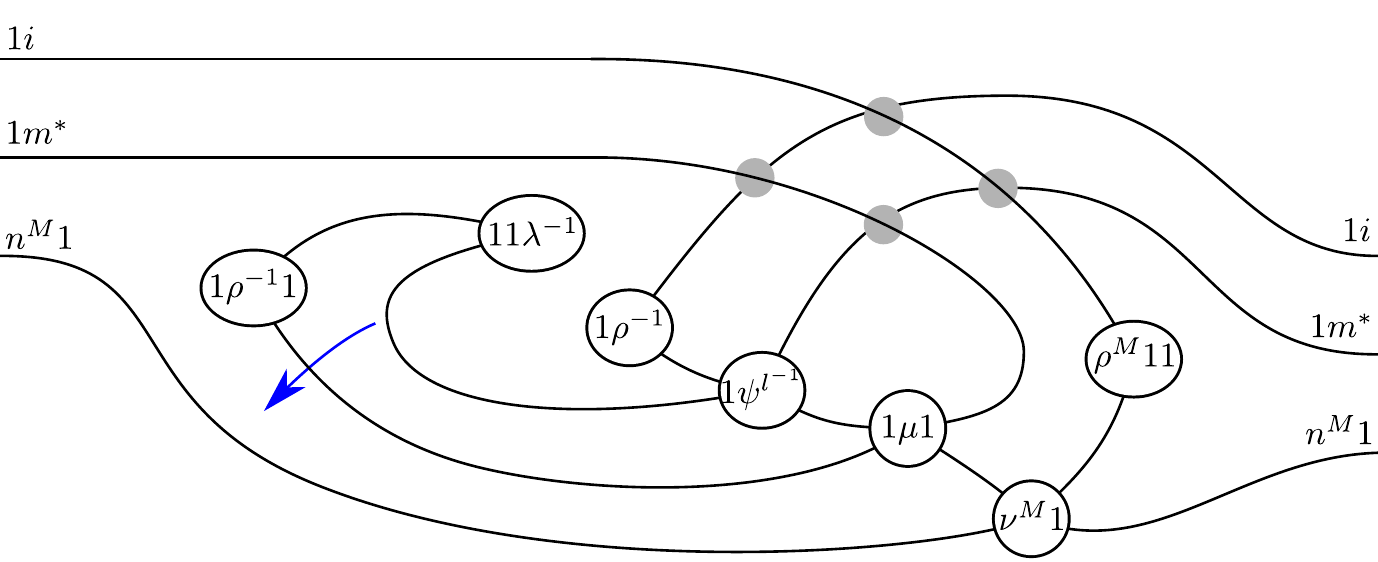}
\caption{Triangle 2 (Part 4)}
\label{fig:triangle13}
\end{figure}

\begin{figure}[!hbt]
\centering
\includegraphics[width=108.75mm]{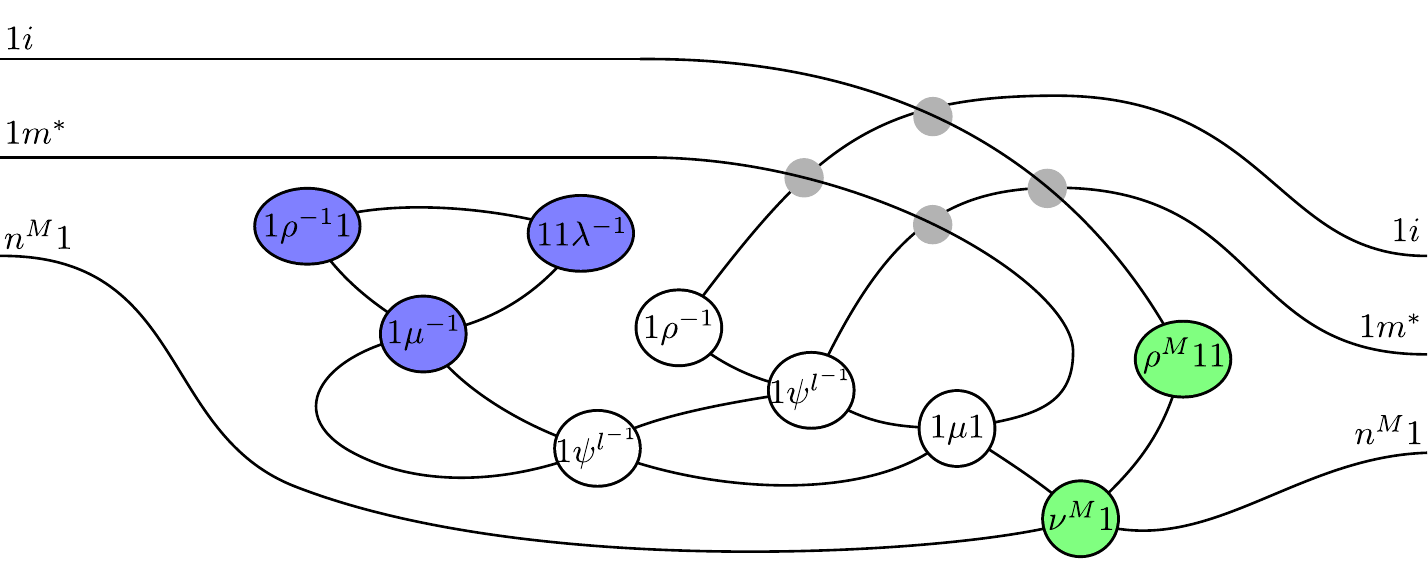}
\caption{Triangle 2 (Part 5)}
\label{fig:triangle14}
\end{figure}

\begin{figure}[!hbt]
\centering
\includegraphics[width=105mm]{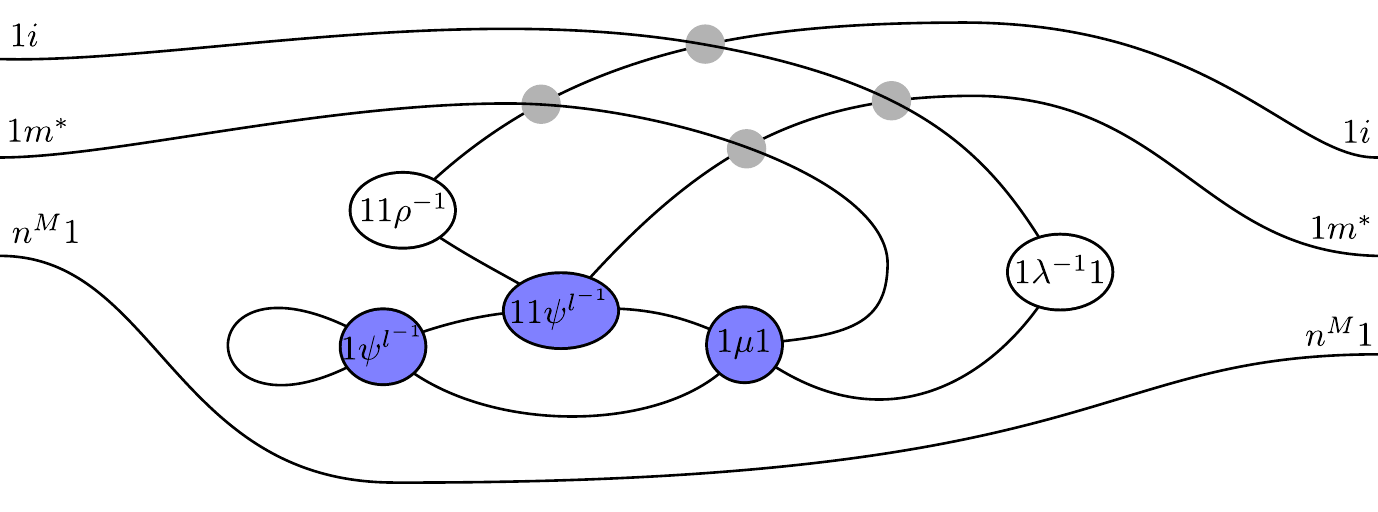}
\caption{Triangle 2 (Part 6)}
\label{fig:triangle15}
\end{figure}
    
\begin{figure}[!hbt]
\centering
\includegraphics[width=78.75mm]{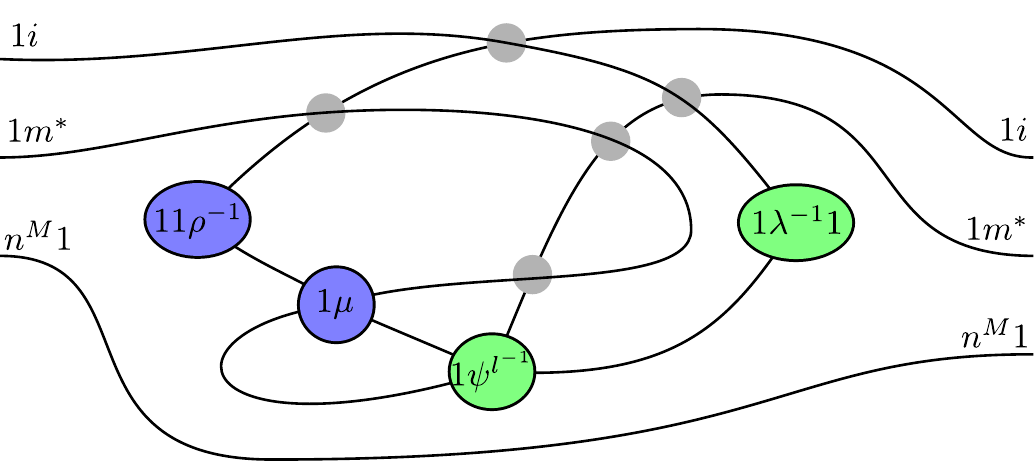}
\caption{Triangle 2 (Part 7)}
\label{fig:triangle16}
\end{figure}

\begin{figure}[!hbt]
\centering
\includegraphics[width=67.5mm]{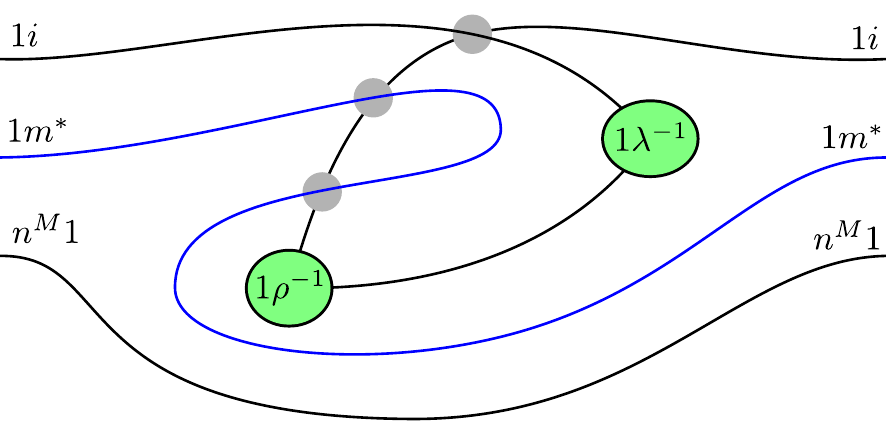}
\caption{Triangle 2 (Part 8)}
\label{fig:triangle17}
\end{figure}

\FloatBarrier

\begin{landscape}
    \vspace*{\fill}
    \subsection{Proof of proposition \ref{prop:moduleleftdajoints}}\label{sub:moduleleftadjointsdiagrams}
    \begin{figure}[!hbt]
    \centering
    \includegraphics[width=180mm]{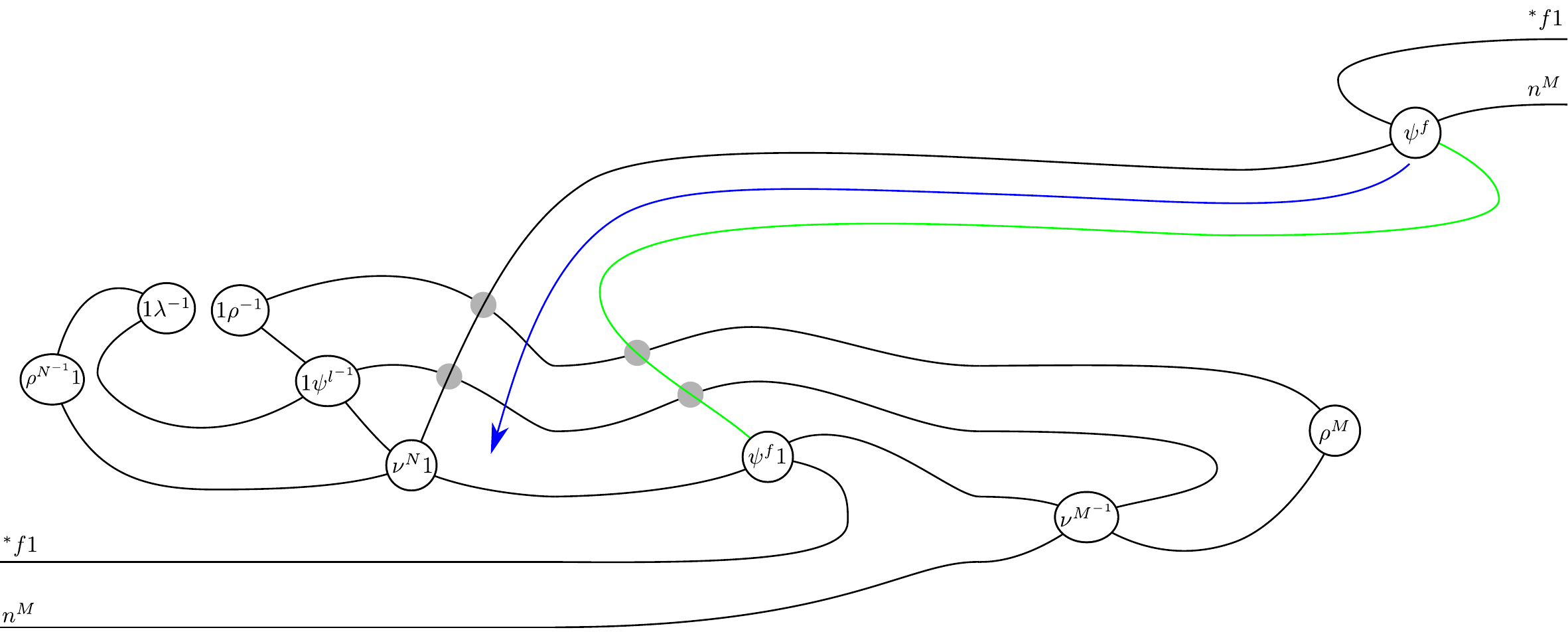}
    \caption{The 2-morphisms $\psi^{^*f}$ and $\xi^f$ are inverses. (Part I.1)}
    \label{fig:adjinverse1}
    \end{figure}
    \vfill
\end{landscape}

\begin{figure}[!hbt]
    \centering
    \includegraphics[width=115mm]{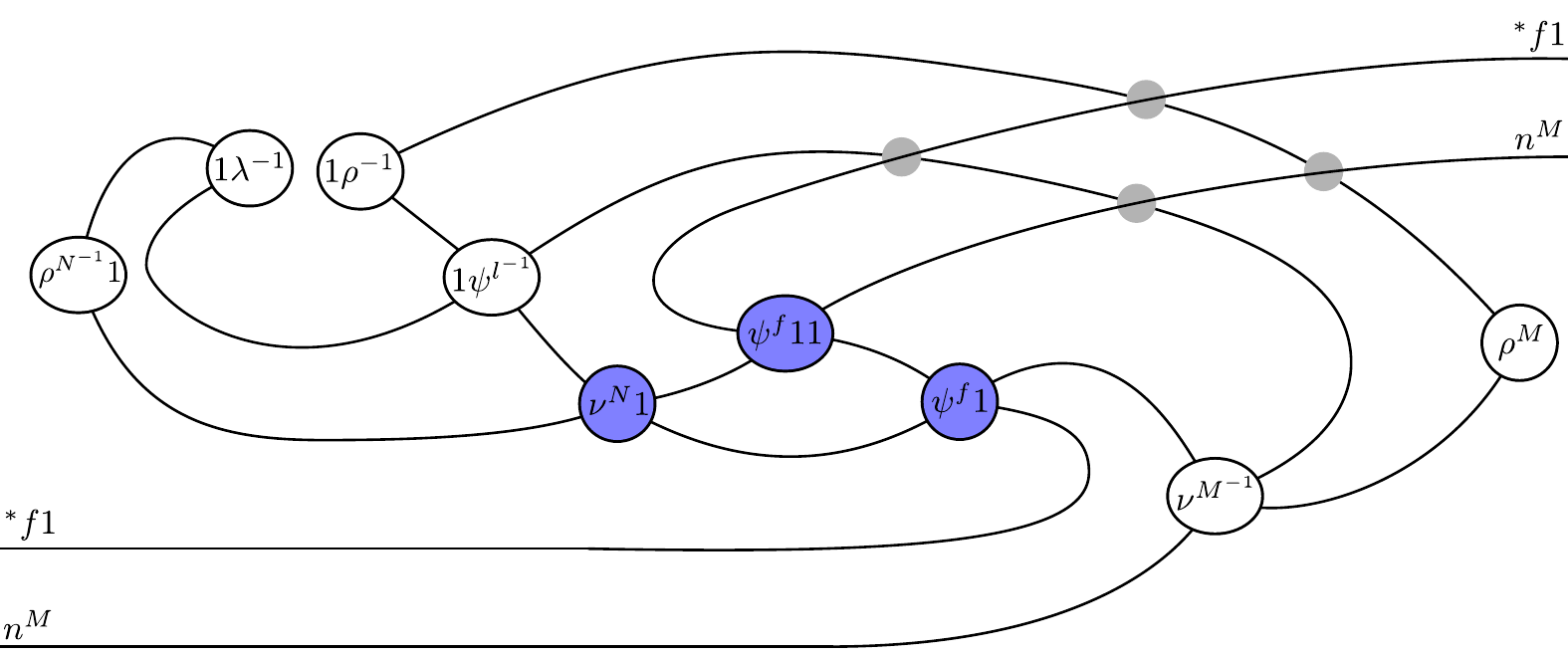}
    \caption{The 2-morphisms $\psi^{^*f}$ and $\xi^f$ are inverses. (Part I.2)}
    \label{fig:adjinverse2}
\end{figure}

\begin{figure}[!hbt]
    \centering
    \includegraphics[width=115mm]{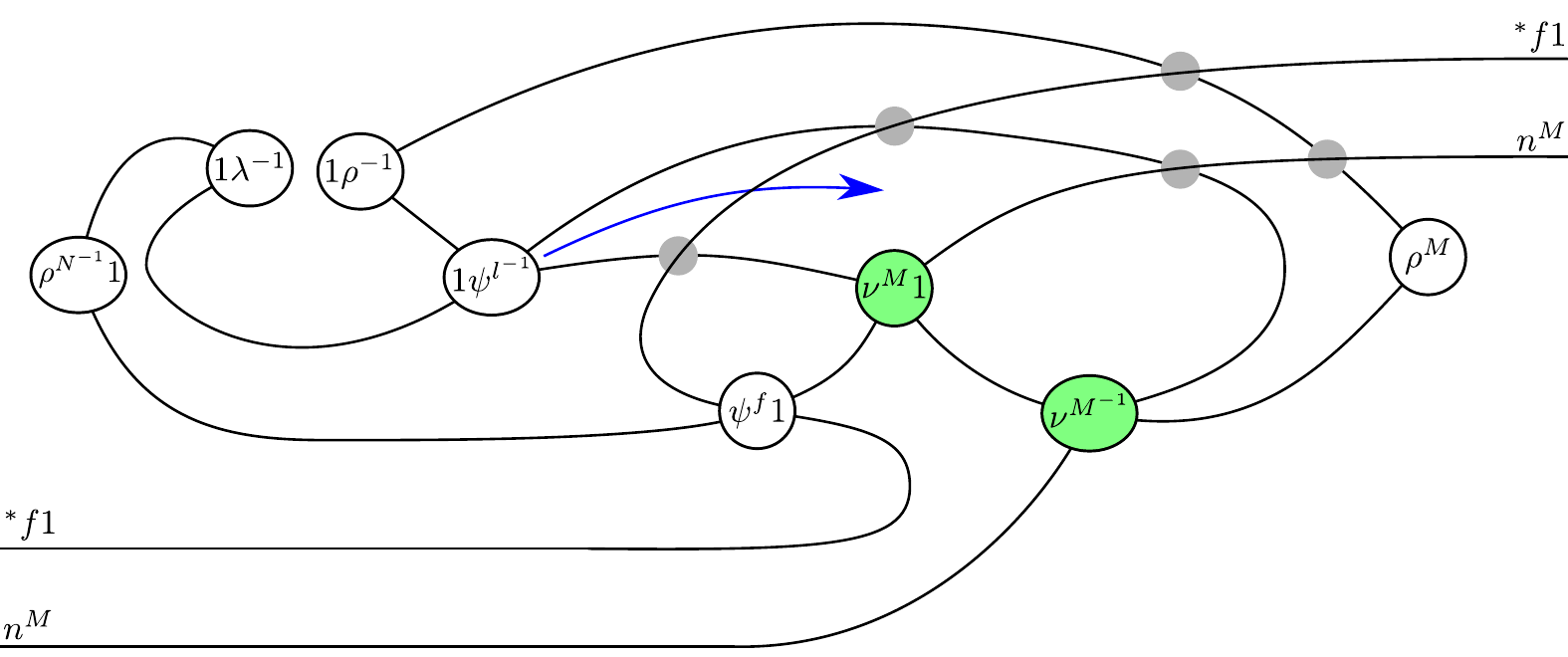}
    \caption{The 2-morphisms $\psi^{^*f}$ and $\xi^f$ are inverses. (Part I.3)}
    \label{fig:adjinverse3}
\end{figure}

\begin{figure}[!hbt]
    \centering
    \includegraphics[width=115mm]{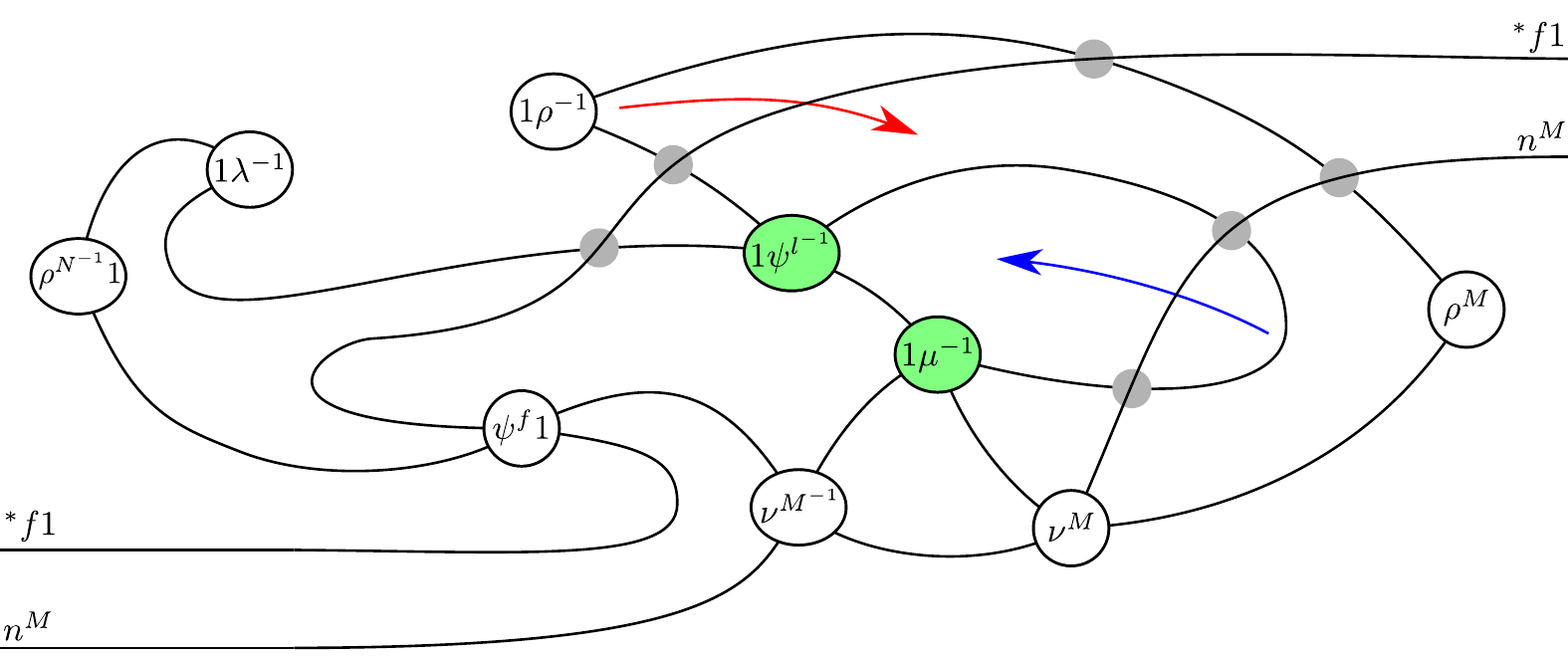}
    \caption{The 2-morphisms $\psi^{^*f}$ and $\xi^f$ are inverses. (Part I.4)}
    \label{fig:adjinverse4}
\end{figure}

\begin{figure}[!hbt]
    \centering
    \includegraphics[width=105mm]{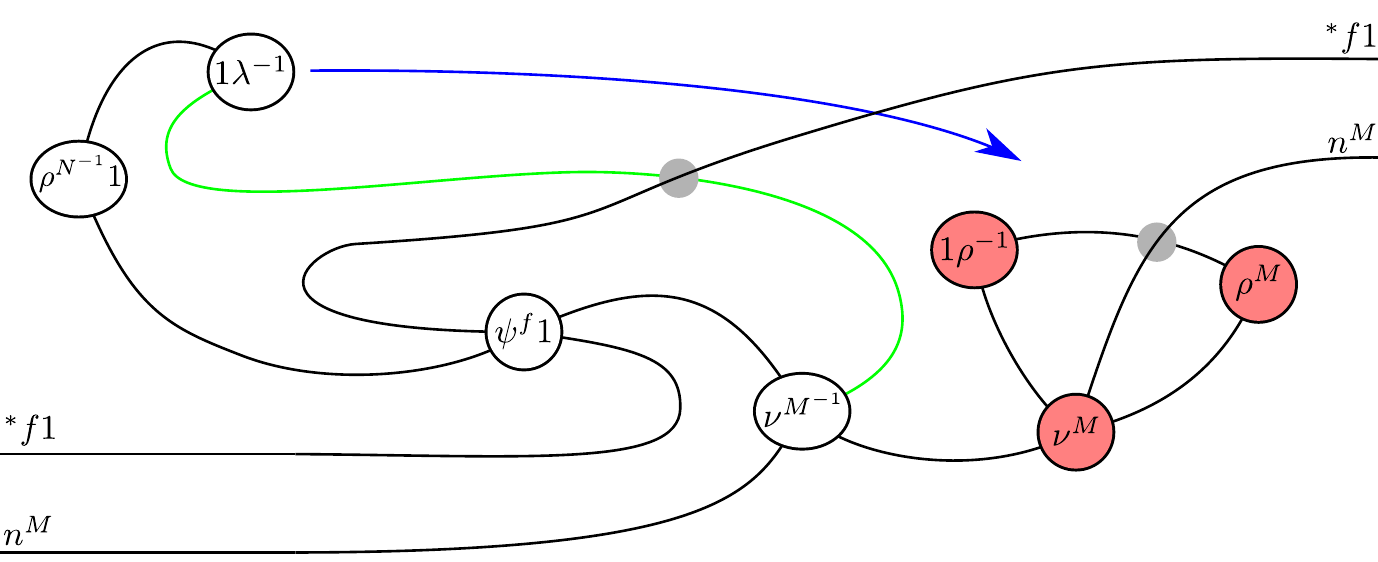}
    \caption{The 2-morphisms $\psi^{^*f}$ and $\xi^f$ are inverses. (Part I.5)}
    \label{fig:adjinverse5}
\end{figure}

\begin{figure}[!hbt]
    \centering
    \includegraphics[width=90mm]{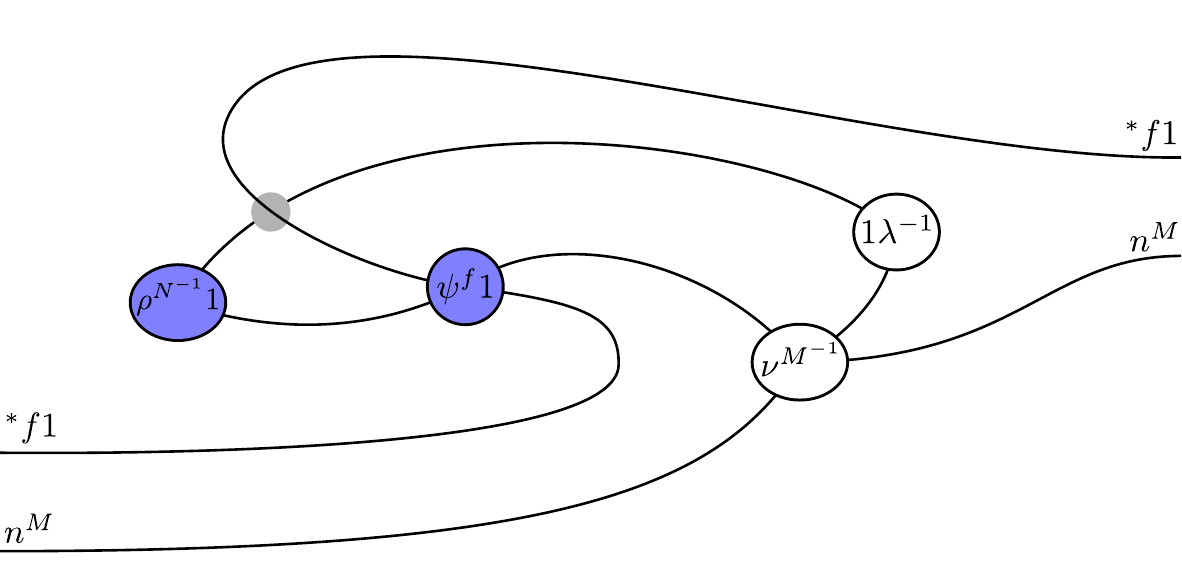}
    \caption{The 2-morphisms $\psi^{^*f}$ and $\xi^f$ are inverses. (Part I.6)}
    \label{fig:adjinverse6}
\end{figure}
    
\begin{figure}[!hbt]
    \centering
    \includegraphics[width=75mm]{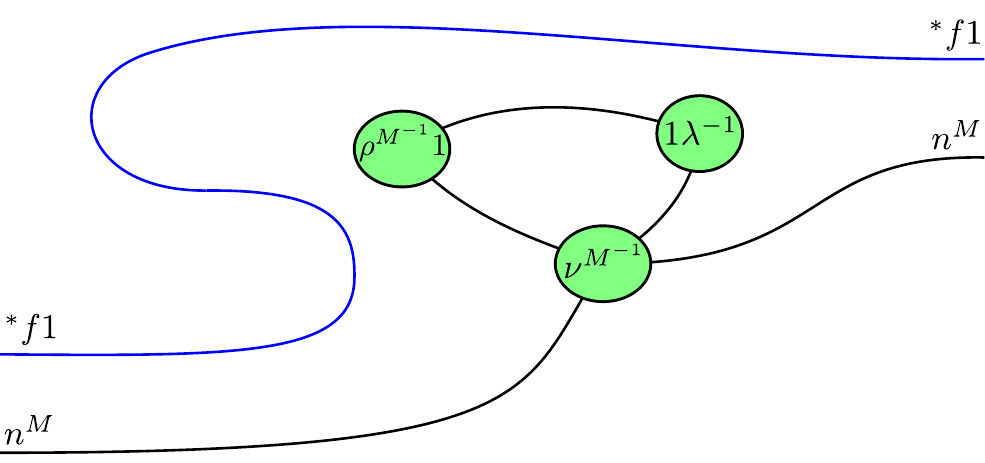}
    \caption{The 2-morphisms $\psi^{^*f}$ and $\xi^f$ are inverses. (Part I.7)}
    \label{fig:adjinverse7}
\end{figure}

\begin{landscape}
    \vspace*{\fill}
    \begin{figure}[!hbt]
    \centering
    \includegraphics[width=180mm]{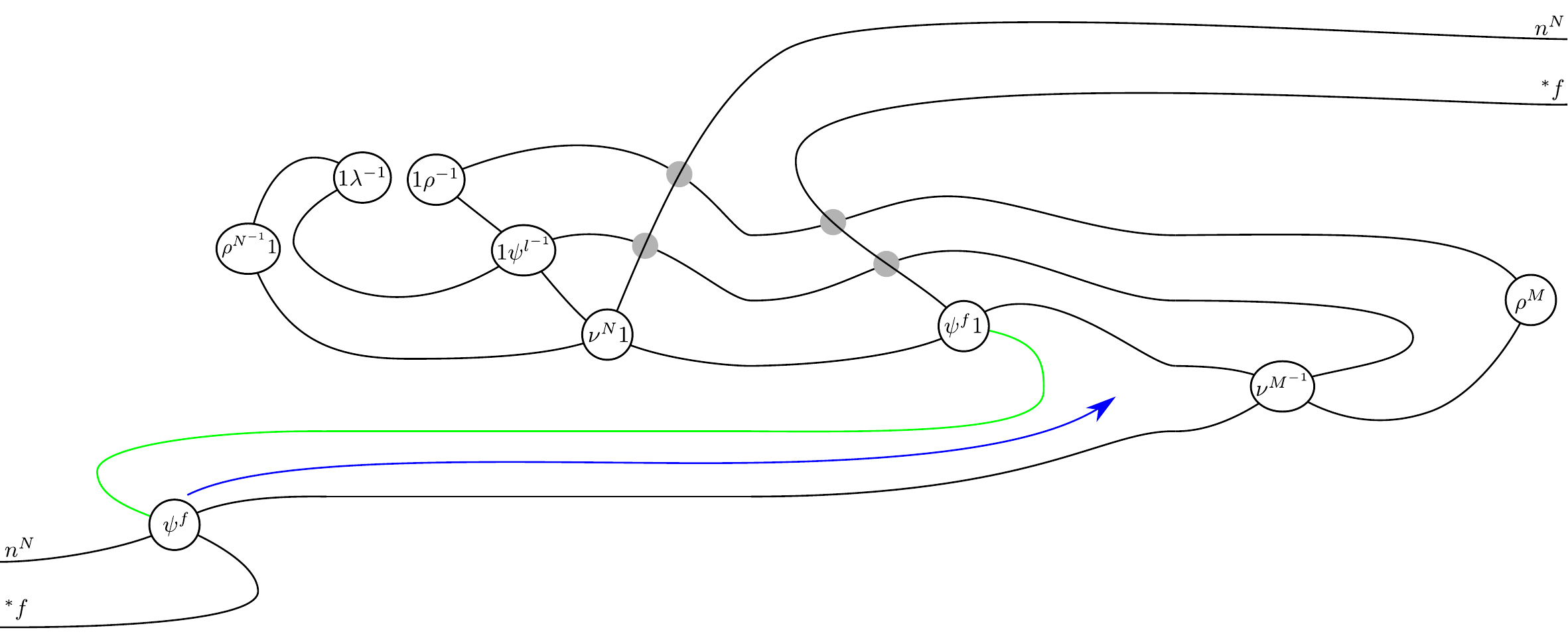}
    \caption{The 2-morphisms $\psi^{^*f}$ and $\xi^f$ are inverses. (Part II.1)}
    \label{fig:adjinverse10}
    \end{figure}
    \vfill
\end{landscape}

\begin{figure}[!hbt]
    \centering
    \includegraphics[width=120mm]{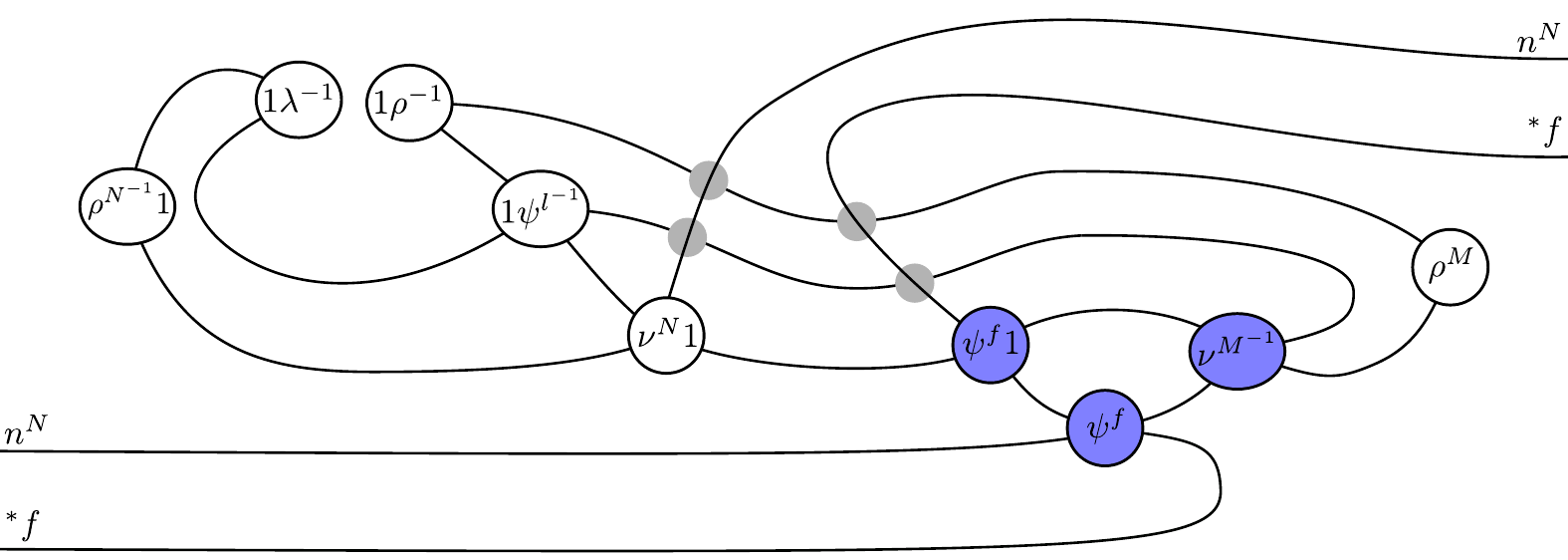}
    \caption{The 2-morphisms $\psi^{^*f}$ and $\xi^f$ are inverses. (Part II.2)}
    \label{fig:adjinverse11}
\end{figure}

\begin{figure}[!hbt]
    \centering
    \includegraphics[width=112.5mm]{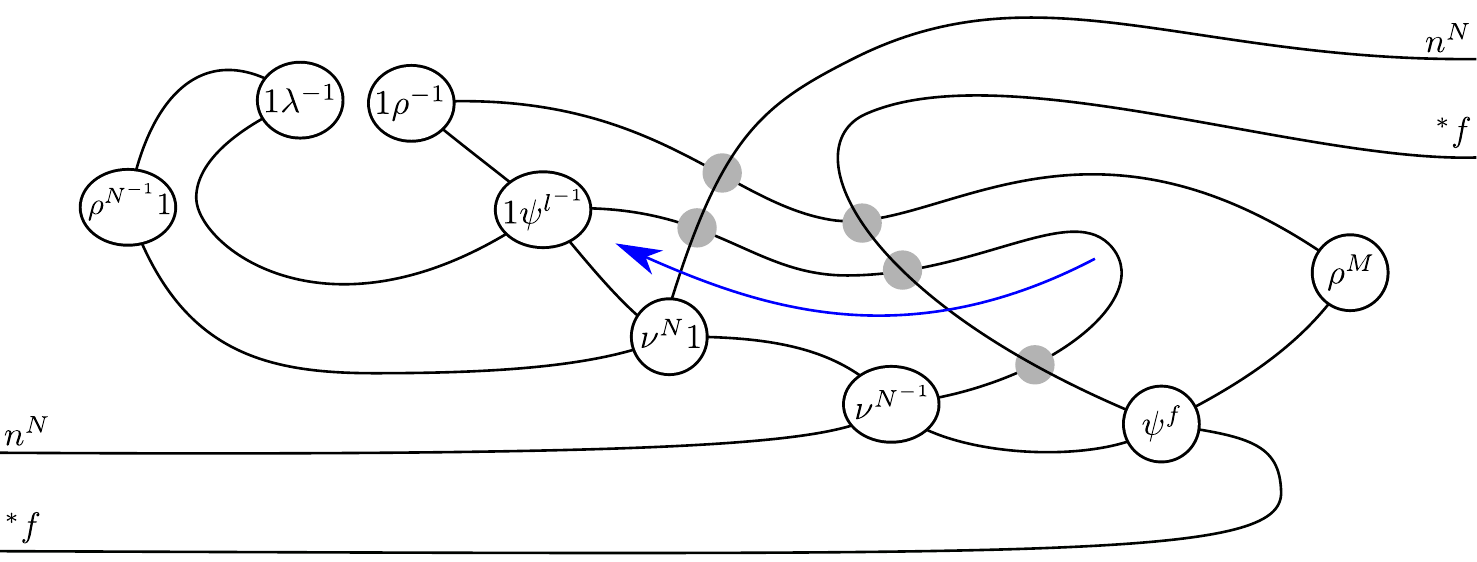}
    \caption{The 2-morphisms $\psi^{^*f}$ and $\xi^f$ are inverses. (Part II.3)}
    \label{fig:adjinverse12}
\end{figure}

\begin{figure}[!hbt]
    \centering
    \includegraphics[width=120mm]{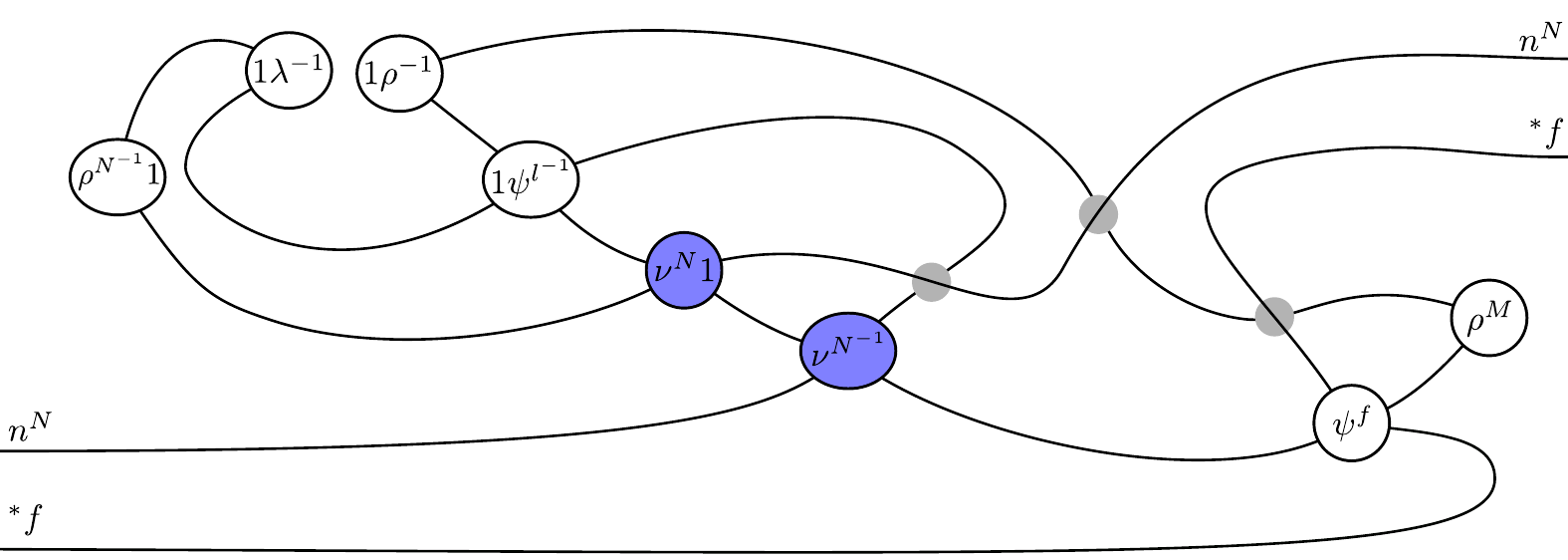}
    \caption{The 2-morphisms $\psi^{^*f}$ and $\xi^f$ are inverses. (Part II.4)}
    \label{fig:adjinverse13}
\end{figure}

\begin{figure}[!hbt]
    \centering
    \includegraphics[width=105mm]{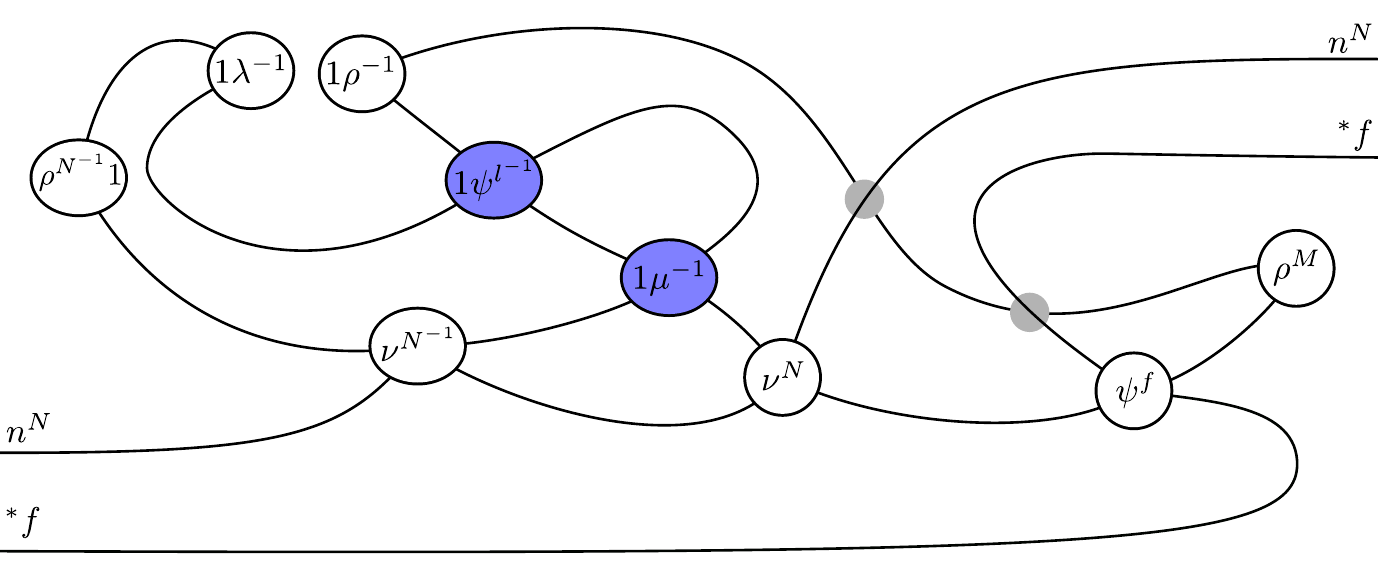}
    \caption{The 2-morphisms $\psi^{^*f}$ and $\xi^f$ are inverses. (Part II.5)}
    \label{fig:adjinverse14}
\end{figure}

\begin{figure}[!hbt]
    \centering
    \includegraphics[width=97.5mm]{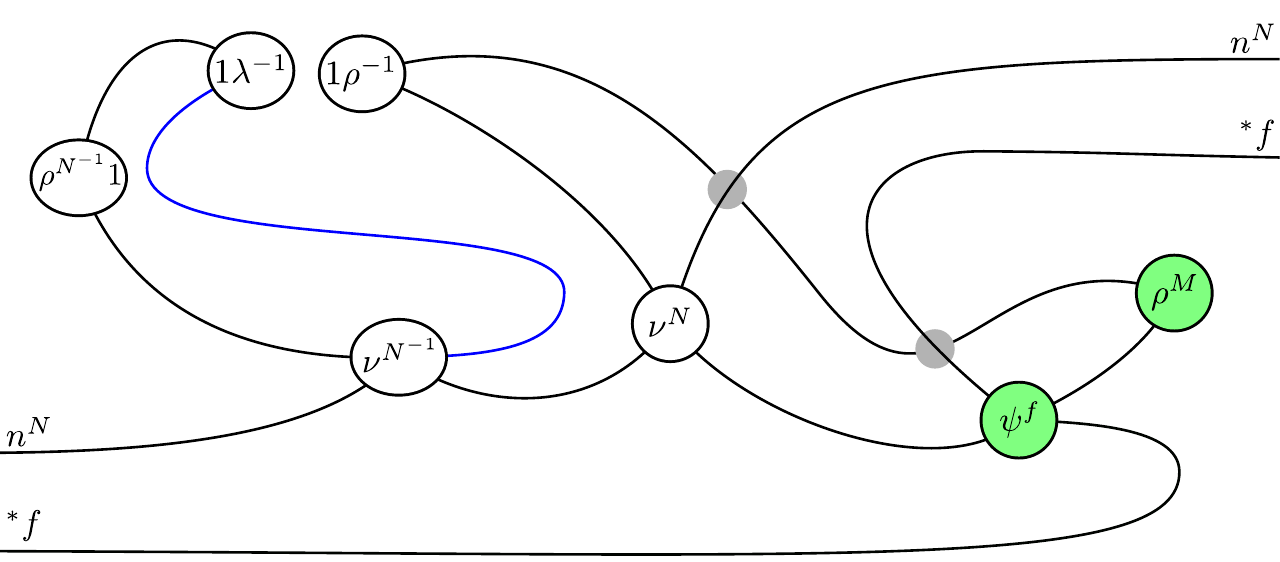}
    \caption{The 2-morphisms $\psi^{^*f}$ and $\xi^f$ are inverses. (Part II.6)}
    \label{fig:adjinverse15}
\end{figure}

\begin{figure}[!hbt]
    \centering
    \includegraphics[width=97.5mm]{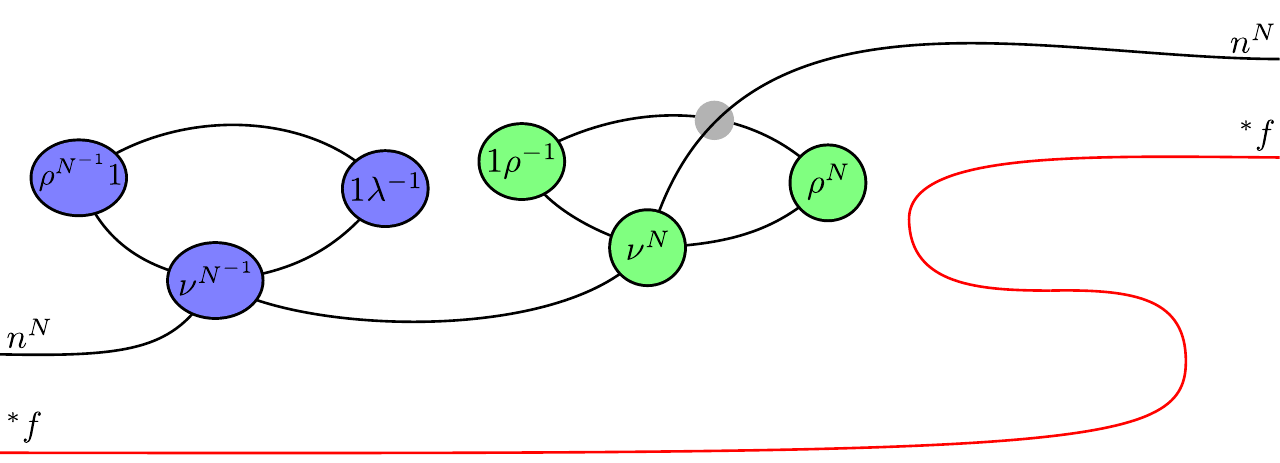}
    \caption{The 2-morphisms $\psi^{^*f}$ and $\xi^f$ are inverses. (Part II.7)}
    \label{fig:adjinverse16}
\end{figure}

\FloatBarrier

\bibliography{bibliography.bib}

\begin{thebibliography}{DGNO10}

\bibitem[BD95]{BD1}
John~C. Baez and James Dolan.
\newblock Higher-dimensional algebra and topological quantum field theory.
\newblock {\em J. Math. Phys.}, 36:6073--6105, 1995.

\bibitem[BD98]{BD2}
John~C. Baez and James Dolan.
\newblock Categorification.
\newblock In {\em Higher Category Theory}, volume 230 of {\em Contemporary
  Mathematics}, pages 1--36. AMS, 1998.

\bibitem[BJS21]{BJS}
Adrien Brochier, David Jordan, and Noah Snyder.
\newblock On dualizability of braided tensor categories.
\newblock {\em Compositio Mathematica}, 3:435--483, 2021.
\newblock arXiv:1804.07538.

\bibitem[BN95]{BN}
John~C. Baez and Martin Neuchl.
\newblock Higher-dimensional algebra i:braided monoidal 2-categories.
\newblock {\em Advances in Mathematics}, 121:196--244, 1995.
\newblock arXiv:q-alg/9511013.

\bibitem[BZBJ18]{BZBJ}
David Ben-Zvi, Adrien Brochier, and David Jordanr.
\newblock Integrating quantum groups over surfaces.
\newblock {\em Journal of Topology}, 11(4):873--916, 2018.
\newblock arXiv:1501.04652.

\bibitem[Dav10]{Da}
Alexei Davydov.
\newblock Centre of an algebra.
\newblock {\em Advances in Mathematics}, 225(1):319--348, 2010.
\newblock arXiv:0908.1250.

\bibitem[D{\'e}c21a]{D5}
Thibault~D. D{\'e}coppet.
\newblock Compact semisimple 2-categories, 2021.
\newblock arXiv:2111.09080.

\bibitem[D{\'e}c21b]{D4}
Thibault~D. D{\'e}coppet.
\newblock Finite semisimple module 2-categories, 2021.
\newblock arXiv:2107.11037.

\bibitem[D{\'e}c22a]{D9}
Thibault~D. D{\'e}coppet.
\newblock Drinfeld centers and {M}orita equivalence classes of fusion
  2-categories, 2022.
\newblock arXiv:2211.04917.

\bibitem[D{\'e}c22b]{D1}
Thibault~D. D{\'e}coppet.
\newblock Multifusion categories and finite semisimple 2-categories.
\newblock {\em Journal of Pure and Applied Algebra}, 226(8), 2022.
\newblock arXiv:2012.15774.

\bibitem[Del90]{Del}
Pierre Deligne.
\newblock Cat{\'e}gories tannakienne.
\newblock In {\em The Grothendieck Festschrift, Vol. II}, volume~87 of {\em
  Prog. Math.}, pages 111--195. Birkh{\"a}user Boston, 1990.

\bibitem[Del22]{Delc}
Clement Delcamp.
\newblock Tensor network approach to electromagnetic duality in (3+1)d
  topological gauge models.
\newblock {\em JHEP}, 149, 2022.
\newblock arXiv:2112.08324.

\bibitem[DGNO10]{DGNO}
Vladimir Drinfeld, Shlomo Gelaki, Dmitri Nikshych, and Victor Ostrik.
\newblock On braided fusion categories {I}.
\newblock {\em Selecta Mathematica}, 16:1–119, 2010.
\newblock arXiv:0906.0620.

\bibitem[DR18]{DR}
Christopher~L. Douglas and David~J. Reutter.
\newblock Fusion 2-categories and a state-sum invariant for 4-manifolds, 2018.
\newblock arXiv: 1812.11933.

\bibitem[Dri87]{Dri}
Vladimir Drinfeld.
\newblock Quantum groups.
\newblock {\em Proceedings of the International Congress of Mathematicians,
  Berkeley}, page 798–820, 1987.

\bibitem[DS97]{DS}
Brian Day and Ross Street.
\newblock Monoidal bicategories and {H}opf algebroids.
\newblock {\em Advances in Mathematics}, 129(AI971649):99--157, 1997.

\bibitem[DSPS19]{DSPS14}
Christopher~L. Douglas, Christopher Schommer-Pries, and Noah Snyder.
\newblock The balanced tensor product of module categories.
\newblock {\em Kyoto J. Math.}, 59:167--179, 2019.
\newblock arXiv: 1406.4204.

\bibitem[DSPS21]{DSPS13}
Christopher~L. Douglas, Christopher Schommer-Pries, and Noah Snyder.
\newblock {\em Dualizable tensor categories}.
\newblock Mem. Amer. Math. Soc. AMS, 2021.
\newblock arXiv: 1312.7188.

\bibitem[EGNO15]{EGNO}
Pavel Etingof, Shlomo Gelaki, Dmitri Nikshych, and Victor Ostrik.
\newblock {\em Tensor Categories}.
\newblock Mathematical Surveys and Monographs. AMS, 2015.

\bibitem[ENO05]{ENO}
Pavel Etingof, Dmitri Nikshych, and Viktor Ostrik.
\newblock On fusion categories.
\newblock {\em Ann. Math.}, 162:581--642, 2005.
\newblock arXiv: math/0203060.

\bibitem[EO04]{EO}
Pavel Etingof and Viktor Ostrik.
\newblock Finite tensor categories.
\newblock {\em Mosc. Math. J.}, 4(3):627–654, 2004.
\newblock arXiv:math/0301027.

\bibitem[Gai12]{G}
Dennis Gaitsgory.
\newblock Sheaves of categories and the notion of 1-affineness.
\newblock In {\em Stacks and Categories in Geometry, Topology, and Algebra},
  volume 643 of {\em Contemporary Mathematics}, pages 127--226. AMS, 2012.
\newblock arXiv:1306.4304.

\bibitem[Gal17]{Gal}
C{\'e}sar Galindo.
\newblock Coherence for monoidal {G}-categories and braided {G}-crossed
  categories.
\newblock {\em Journal of Algebra}, 487:118--137, 2017.
\newblock arXiv:1604.01679.

\bibitem[GJF19]{GJF}
Davide Gaiotto and Theo Johnson-Freyd.
\newblock Condensations in higher categories, 2019.
\newblock arXiv: 1905.09566v2.

\bibitem[GPS95]{GPS}
R.~Gordon, A.~J. Power, and Ross Street.
\newblock {\em Coherence for Tricategories}.
\newblock Memoirs of the American Mathematical Society. AMS, 1995.

\bibitem[GS16]{GS}
Richard Garner and Michael Schulman.
\newblock Enriched categories as a free cocompletion.
\newblock {\em Advances in Mathematics}, 289:1--94, 2016.
\newblock arXiv:1301.3191v2.

\bibitem[Gur13]{Gur}
Nick Gurski.
\newblock {\em Coherence in Three-Dimensional Category Theory}.
\newblock Cambridge Tracts in Mathematics. Cambridge University Press, 2013.

\bibitem[Hou07]{Hou}
Robin Houston.
\newblock {\em Linear Logic without Units}.
\newblock PhD thesis, University of Manchester, 2007.
\newblock arXiv:1305.2231.

\bibitem[HPT16]{HPT}
André Henriques, David Penneys, and James Tener.
\newblock Planar algebras in braided tensor categoriess, 2016.
\newblock arXiv:1607.06041.

\bibitem[JFR21]{JFR}
Theo Johnson-Freyd and David~J. Reutter.
\newblock Minimal non-degenerate extensions, 2021.
\newblock arXiv:2105.15167.

\bibitem[KTZ20]{KTZ}
Liang Kong, Yin Tian, and Shan Zhou.
\newblock The center of monoidal 2-categories in 3+1d {D}ijkgraaf-{W}itten
  theory.
\newblock {\em Advances in Mathematics}, 360(106928), 2020.
\newblock arXiv:1905.04644.

\bibitem[Lur10]{L}
Jacob Lurie.
\newblock On the classification of topological field theories.
\newblock {\em Curr. Dev. Math.}, 1:129--280, 2010.
\newblock arXiv:0905.0465.

\bibitem[MPP18]{MPP}
Scott Morrison, David Penneys, and Julia Plavnik.
\newblock Completion for braided enriched monoidal categories, 2018.
\newblock arXiv:1809.09782.

\bibitem[Ost03]{O}
Victor Ostrik.
\newblock Module categories, weak hopf algebras and modular invariants.
\newblock {\em Transformation Groups}, 8:177–206, 2003.
\newblock arXiv:math/0111139.

\bibitem[SP11]{SP}
Christopher~J. Schommer-Pries.
\newblock {\em The Classification of Two-Dimensional Extended Topological Field
  Theories}.
\newblock PhD thesis, UC Berkeley, 2011.
\newblock arXiv: 1112.1000.

\bibitem[SR72]{Saa}
Neantro Saavedra~Rivano.
\newblock Catégories tannakiennes.
\newblock {\em Bulletin de la Société Mathématique de France}, 100:417--430,
  1972.

\bibitem[Tur00]{T}
Vladimir Turaev.
\newblock Homotopy field theory in dimension 3 and crossed group-categories,
  2000.
\newblock arXiv:math/0005291.

\bibitem[TV92]{TV}
Vladimir~G. Turaev and Oleg~Ya. Viro.
\newblock State sum invariants of 3-manifolds and quantum 6j-symbols.
\newblock {\em Topology}, 31(4):865--902, 1992.

\end{thebibliography}

\end{document}